\documentclass[a4paper,leqno,12pt]{amsart}
\usepackage{amsmath,amsthm}
\usepackage{amssymb}
\usepackage{xspace}
\usepackage{bm}
\usepackage{amstext}
\usepackage{amsfonts}
\usepackage{graphicx}
\usepackage[mathscr]{euscript}
\usepackage{amscd}
\usepackage{latexsym}
\usepackage{amssymb}
\usepackage{enumerate}
\usepackage{mathrsfs}
\usepackage[cmtip,all]{xy}
\usepackage{tikz}
\usetikzlibrary{shapes.geometric}
\begin{document}
%
%
\theoremstyle{plain}
\swapnumbers
    \newtheorem{thm}[figure]{Theorem}
    \newtheorem{prop}[figure]{Proposition}
    \newtheorem{lemma}[figure]{Lemma}
    \newtheorem{keylemma}[figure]{Key Lemma}
    \newtheorem{corollary}[figure]{Corollary}
        \newtheorem{fact}[figure]{Fact}
    \newtheorem{subsec}[figure]{}
    \newtheorem*{propa}{Proposition A}
    \newtheorem*{thma}{Theorem A}
    \newtheorem*{thmb}{Theorem B}
    \newtheorem*{thmc}{Theorem C}
    \newtheorem*{thmd}{Theorem D}
\theoremstyle{definition}
    \newtheorem{defn}[figure]{Definition}
    \newtheorem{examples}[figure]{Examples}
    \newtheorem{notn}[figure]{Notation}
    \newtheorem{summary}[figure]{Summary}
     \newtheorem{question}[figure]{Question}
\theoremstyle{remark}
        \newtheorem{remark}[figure]{Remark}
        \newtheorem{remarks}[figure]{Remarks}
        \newtheorem{example}[figure]{Example}
        \newtheorem{warning}[figure]{Warning}
    \newtheorem{assume}[figure]{Assumption}
    \newtheorem{ack}[figure]{Acknowledgements}
\renewcommand{\thefigure}{\arabic{section}.\arabic{figure}}
%
%
%
\newenvironment{myeq}[1][]
{\stepcounter{figure}\begin{equation}\tag{\thefigure}{#1}}
{\end{equation}}
\newcommand{\myeqn}[2][]
{\stepcounter{figure}\begin{equation}
     \tag{\thefigure}{#1}\vcenter{#2}\end{equation}}
\newcommand{\mydiag}[2][]{\myeq[#1]\xymatrix{#2}}
\newcommand{\mydiagram}[2][]
{\stepcounter{figure}\begin{equation}
     \tag{\thefigure}{#1}\vcenter{\xymatrix{#2}}\end{equation}}
\newcommand{\myssdiag}[2][]
{\stepcounter{figure}\begin{equation}
     \tag{\thefigure}{#1}\vcenter{\xymatrix@R=10pt@C=45pt{#2}}\end{equation}}
\newcommand{\mysdiag}[2][]
{\stepcounter{figure}\begin{equation}
     \tag{\thefigure}{#1}\vcenter{\xymatrix@R=13pt@C=40pt{#2}}\end{equation}}
\newcommand{\mytdiag}[2][]
{\stepcounter{figure}\begin{equation}
     \tag{\thefigure}{#1}\vcenter{\xymatrix@R=1pt@C=4pt{#2}}\end{equation}}
\newcommand{\myrdiag}[2][]
{\stepcounter{figure}\begin{equation}
     \tag{\thefigure}{#1}\vcenter{\xymatrix@R=13pt@C=12pt{#2}}\end{equation}}
\newcommand{\myqdiag}[2][]
{\stepcounter{figure}\begin{equation}
     \tag{\thefigure}{#1}\vcenter{\xymatrix@R=16pt@C=0pt{#2}}\end{equation}}
\newcommand{\myfigure}[2][]
{\stepcounter{figure}\begin{equation}
     \tag{\thefigure}{#1}\vcenter{#2}\end{equation}}
\newcommand{\mywdiag}[2][]
{\stepcounter{figure}\begin{equation}
     \tag{\thefigure}{#1}\vcenter{\xymatrix@R=20pt@C=8pt{#2}}\end{equation}}
\newcommand{\myzdiag}[2][]
{\stepcounter{figure}\begin{equation}
     \tag{\thethm}{#1}\vcenter{\xymatrix@R=5pt@C=20pt{#2}}\end{equation}}
%
\newenvironment{mysubsection}[2][]
{\begin{subsec}\begin{upshape}\begin{bfseries}{#2.}
\end{bfseries}{#1}}
{\end{upshape}\end{subsec}}
\newenvironment{mysubsect}[2][]
{\begin{subsec}\begin{upshape}\begin{bfseries}{#2\vsn.}
\end{bfseries}{#1}}
{\end{upshape}\end{subsec}}
\newcommand{\supsect}[2]
{\vspace*{-5mm}\quad\\\begin{center}\textbf{{#1}}\vsm.~~~~\textbf{{#2}}\end{center}}
\newcommand{\sect}{\setcounter{figure}{0}\section}
%
%
\newcommand{\wh}{\ -- \ }
\newcommand{\wwh}{-- \ }
\newcommand{\w}[2][ ]{\ \ensuremath{#2}{#1}\ }
\newcommand{\ww}[1]{\ \ensuremath{#1}}
\newcommand{\www}[2][ ]{\ensuremath{#2}{#1}\ }
\newcommand{\wwb}[1]{\ \ensuremath{(#1)}-}
\newcommand{\wb}[2][ ]{\ (\ensuremath{#2}){#1}\ }
\newcommand{\wref}[2][ ]{\ (\ref{#2}){#1}\ }
\newcommand{\wwref}[3][ ]{\ (\ref{#2})-(\ref{#3}){#1}\ }
%
%
\newcommand{\hs}{\hspace*{5 mm}}
\newcommand{\hsl}{\hspace*{-3 mm}}
\newcommand{\hsn}{\hspace*{1 mm}}
\newcommand{\hsm}{\hspace*{2 mm}}
\newcommand{\hsp}{\hspace*{9 mm}}
\newcommand{\hsq}{\hspace*{13 mm}}
\newcommand{\vsn}{\vspace{2 mm}}
\newcommand{\vs}{\vspace{5 mm}}
\newcommand{\vsm}{\vspace{3 mm}}
\newcommand{\vsp}{\vspace{9 mm}}
%
%
\newcommand{\hra}{\hookrightarrow}
\newcommand{\xra}[1]{\xrightarrow{#1}}
\newcommand{\xla}[1]{\xleftarrow{#1}}
\newcommand{\xepic}[1]{\xrightarrow{#1}\hspace{-5 mm}\to}
\newcommand{\lora}{\longrightarrow}
\newcommand{\mapsfrom}{\reflectbox{\mbox{$\mapsto$}}}
\newcommand{\lra}[1]{\langle{#1}\rangle}
\newcommand{\llrra}[1]{\langle\langle{#1}\rangle\rangle}
\newcommand{\llrr}[2]{\llrra{#1}\sb{#2}}
\newcommand{\llrrp}[2]{\llrra{#1}'\sb{#2}}
\newcommand{\lrf}{\langle\langle f\lo{0,1}\rangle\rangle}
\newcommand{\lrfn}[1]{\lrf\sb{#1}}
\newcommand{\lras}[1]{\langle{#1}\rangle\sb{\ast}}
\newcommand{\lrau}[1]{\langle{#1}\rangle\sp{\ast}}
\newcommand{\vlam}{\vec{\lambda}}
\newcommand{\epic}{\to\hspace{-3.5 mm}\to}
\newcommand{\xhra}[1]{\overset{#1}{\hookrightarrow}}
\newcommand{\efp}{\to\hspace{-1.5 mm}\rule{0.1mm}{2.2mm}\hspace{1.2mm}}
\newcommand{\efpic}{\mbox{$\to\hspace{-3.5 mm}\efp$}}
\newcommand{\up}[1]{\sp{(#1)}}
\newcommand{\bup}[1]{\sp{[{#1}]}}
\newcommand{\lo}[1]{\sb{(#1)}}
\newcommand{\lolr}[1]{\sb{\lra{#1}}}
\newcommand{\bp}[1]{\sb{[#1]}}
\newcommand{\hfsm}[2]{{#1}\ltimes{#2}}
\newcommand{\sms}[2]{{#1}\wedge{#2}}
\newcommand{\rest}[1]{\lvert\sb{#1}}
\newcommand{\ostar}[2]{{#1}\otimes{#2\sb{\ast}}}
\newcommand{\ii}{\wwb{\infty,1}}
\newcommand{\adj}[2]{\substack{{#1}\\ \rightleftharpoons \\ {#2}}}
%

%
%
\newcommand{\ab}{\operatorname{ab}}
\newcommand{\Ab}{\operatorname{Ab}}
\newcommand{\Aut}{\operatorname{Aut}}
\newcommand{\Cof}[1]{\operatorname{Cof}(#1)}
\newcommand{\Coker}{\operatorname{Coker}}
\newcommand{\colim}{\operatorname{colim}}
\newcommand{\colimit}[1]
{\raisebox{-1.7ex}{$\stackrel{\textstyle\colim}{\scriptstyle{#1}}$}}
\newcommand{\comp}{\mbox{\sf comp}}
\newcommand{\Cone}{\operatorname{Cone}}
\newcommand{\csk}[1]{\operatorname{csk}\sp{#1}}
\newcommand{\dd}{\operatorname{d}}
\newcommand{\Dec}{\operatorname{Dec}}
\newcommand{\diag}{\operatorname{diag}}
\newcommand{\ev}{\operatorname{ev}}
\newcommand{\exc}{\operatorname{ex}}
\newcommand{\Ext}{\operatorname{Ext}}
\newcommand{\Fib}{\operatorname{Fib}}
\newcommand{\fin}{\operatorname{fin}}
\newcommand{\fwd}{\operatorname{fwd}}
\newcommand{\gr}{\operatorname{gr}}
\newcommand{\ho}{\operatorname{ho}}
\newcommand{\hocofib}{\operatorname{hocofib}}
\newcommand{\hocolim}{\operatorname{hocolim}}
\newcommand{\holim}{\operatorname{holim}}
\newcommand{\Hom}{\operatorname{Hom}}
\newcommand{\inc}{\operatorname{inc}}
\newcommand{\Id}{\operatorname{Id}}
\newcommand{\Image}{\operatorname{Im}}
\newcommand{\init}{\operatorname{init}}
\newcommand{\vi}{v\sb{\init}}
\newcommand{\Ker}{\operatorname{Ker}}
\newcommand{\md}{\operatorname{mid}}
\newcommand{\Obj}[1]{\operatorname{Obj}\,{#1}}
\newcommand{\op}{\sp{\operatorname{op}}}
\newcommand{\pt}{\operatorname{pt}}
\newcommand{\red}{\operatorname{red}}
\newcommand{\sgn}[1]{\operatorname{sgn}({#1})}
\newcommand{\gsn}[1]{\operatorname{gsn}({#1})}
\newcommand{\sk}[1]{\operatorname{sk}\sb{#1}}
\newcommand{\SL}[1]{\operatorname{SL}\sb{#1}(\ZZ)}
\newcommand{\Sq}[1]{\operatorname{Sq}\sp{#1}}
\newcommand{\Tot}{\operatorname{Tot}}
\newcommand{\uTot}{\underline{\Tot}}
%
%
\newcommand{\map}{\operatorname{map}}
\newcommand{\mapa}{\map\sb{\ast}}
%
%
\newcommand{\Ei}[3]{E\sb{#1}\sp{{#2},{#3}}}
\newcommand{\Eis}[2]{E\sb{#1}(#2)}
\newcommand{\Eu}[3]{E\sp{#1}\sb{{#2},{#3}}}
\newcommand{\Eus}[1]{E\sp{#1}}
\newcommand{\Euz}[1]{E\sp{0}({#1})}
\newcommand{\Eot}[2]{\Eu{1}{#1}{#2}}
\newcommand{\Ett}[2]{\Eu{2}{#1}{#2}}
%
%
\newcommand{\cM}[1]{C\sb{#1}}
\newcommand{\cZ}[1]{Z\sb{#1}}
\newcommand{\dif}[1]{\partial\sb{#1}}
\newcommand{\oG}[1]{\overline{G}\sb{#1}}
%
%
\newcommand{\A}{\mathcal{A}}
\newcommand{\tA}{\widetilde{A}}
\newcommand{\hA}{\widehat{A}}
\newcommand{\halpha}{\widehat{\alpha}}
\newcommand{\hhalpha}{\widehat{\halpha}}
\newcommand{\B}{\mathcal{B}}
\newcommand{\hB}{\widehat{B}}
\newcommand{\C}{\mathcal{C}}
\newcommand{\hC}{\widehat{\C}}
\newcommand{\mC}{\mathscr{C}}
\newcommand{\D}{\mathcal{D}}
\newcommand{\hD}{\widehat{D}}
\newcommand{\tD}{\widetilde{D}}
\newcommand{\E}{\mathcal{E}}
\newcommand{\hE}{\widehat{E}}
\newcommand{\cF}[1]{\mathcal{F}\sp{#1}}
\newcommand{\Fc}[1]{\mathcal{F}\sb{#1}}
\newcommand{\hf}{\widehat{f}}
\newcommand{\hF}{\widehat{F}}
\newcommand{\G}{\mathcal{G}}
\newcommand{\hg}{\widehat{g}}
\newcommand{\hh}{\widehat{h}}
\newcommand{\uI}{\underline{I}}
\newcommand{\cII}{\mathcal{I}}
\newcommand{\cI}[2]{\cII\sp{#1}\sb{#2}}
\newcommand{\wI}[2]{\widehat{\cII}\sp{#1}\sb{#2}}
\newcommand{\cJJ}{\mathcal{J}}
\newcommand{\hI}[2]{\cJJ\sp{#1}\sb{#2}}
\newcommand{\hIp}{\widehat{I'}}
\newcommand{\hJ}{\widehat{J}}
\newcommand{\hhJp}{\widehat{\widehat{J'}}}
\newcommand{\uJ}{\underline{J}}
\newcommand{\hK}{\widehat{K}}
\newcommand{\hhK}{\widehat{\widehat{K}}}
\newcommand{\uK}{\underline{K}}
\newcommand{\KK}[1]{{\mathcal{K}}\sb{#1}}
\newcommand{\hL}{\widehat{L}}
\newcommand{\cL}{{\mathcal L}}
\newcommand{\Li}{\cL\sb{\infty}}
\newcommand{\M}{\mathcal{M}}
\newcommand{\hM}{\widehat{M}}
\newcommand{\hhM}{\widehat{\widehat{M}}}
\newcommand{\uM}{\underline{M}}
\newcommand{\eM}{{\EuScript M}}
\newcommand{\hN}{\widehat{N}}
\newcommand{\hhN}{\widehat{\widehat{N}}}
\newcommand{\uN}{\underline{N}}
\newcommand{\MA}{\Map\sb{\A}}
\newcommand{\MAB}{\Map\sb{\A}\sp{\B}}
\newcommand{\MB}{\Map\sp{\B}}
\newcommand{\eN}{{\EuScript N}}
\newcommand{\NA}{\eN\sb{A}}
\newcommand{\OO}{\mathcal{O}}
\newcommand{\homeg}{\widehat{\omega}}
\newcommand{\hhomeg}{\widehat{\homeg}}
\newcommand{\cP}[1]{\mathcal{P}\sp{#1}}
\newcommand{\hP}{\widehat{P}}
\newcommand{\hhP}{\widehat{\widehat{P}}}
\newcommand{\uP}{\underline{P}}
\newcommand{\hR}{\widehat{R}}
\newcommand{\hhR}{\widehat{\widehat{R}}}
\newcommand{\uR}{\underline{R}}
\newcommand{\eS}{{\EuScript S}}
\newcommand{\Ss}{\mathcal{S}}
\newcommand{\Sa}{\Ss\sb{\ast}}
\newcommand{\Sr}{\Ss\sp{\red}}
\newcommand{\hT}{\widehat{T}}
\newcommand{\hhT}{\widehat{\widehat{T}}}
\newcommand{\uT}{\underline{T}}
\newcommand{\U}{\mathcal{U}}
\newcommand{\eW}{{\EuScript W}}
\newcommand{\mX}{\mathscr{X}}
\newcommand{\eY}{{\EuScript Y}}
\newcommand{\eZ}{{\EuScript Z}}
%
%
\newcommand{\hy}[2]{{#1}\text{-}{#2}}
\newcommand{\Alg}[1]{{#1}\text{-}{\mbox{\sf Alg}}}
\newcommand{\Pa}[1][ ]{$\Pi$-algebra{#1}}
\newcommand{\PAlg}{\Alg{\Pi}}
\newcommand{\pis}{\pi\sb{\ast}}
\newcommand{\gS}[1]{{\EuScript S}\sp{#1}}
\newcommand{\Mod}[1]{{#1}\text{-}{\mbox{\sf Mod}}}
\newcommand{\us}{u\sp{\ast}}
\newcommand{\Set}{\mbox{\sf Set}}
\newcommand{\Seta}{\Set\sb{\ast}}
\newcommand{\Cat}{\mbox{\sf Cat}}
\newcommand{\DK}{\mbox{\sf DK}}
\newcommand{\Grp}{\mbox{\sf Gp}}
\newcommand{\Map}{\mbox{\sf Map}}
\newcommand{\OC}{\hy{\OO}{\Cat}}
\newcommand{\SC}{\hy{\Ss}{\Cat}}
\newcommand{\SaC}{\hy{\Sa}{\Cat}}
\newcommand{\SO}{(\Ss,\OO)}
\newcommand{\SaO}{(\Sa,\OO)}
\newcommand{\SOC}{\hy{\SO}{\Cat}}
\newcommand{\SaOC}{\hy{\SaO}{\Cat}}
\newcommand{\Top}{\mbox{\sf Top}}
\newcommand{\Tz}{\Top\sb{0}}
\newcommand{\Topa}{\Top\sb{\ast}}
%
%
\newcommand{\CC}{\mathbb C}
\newcommand{\CP}[1]{\CC\mathbf{P}\sp{#1}}
\newcommand{\FF}{\mathbb F}
\newcommand{\LL}{\mathbb L}
\newcommand{\NN}{\mathbb N}
\newcommand{\QQ}{\mathbb Q}
\newcommand{\RR}{\mathbb R}
\newcommand{\ZZ}{\mathbb Z}
\newcommand{\Zp}{\ZZ\lo{p}}
\newcommand{\Zt}{\ZZ\lo{2}}
%
%
\newcommand{\valph}{\vec{\alpha}}
\newcommand{\Del}{\mathbf{\Delta}}
\newcommand{\Deln}[1]{\Del\sp{#1}}
\newcommand{\Delnk}[2]{\Deln{#1}\lo{#2}}
\newcommand{\Dop}{\Delta\op}
\newcommand{\res}{\operatorname{res}}
\newcommand{\Dres}{\Delta\sb{\res}}
\newcommand{\Drop}{\Dres\op}
\newcommand{\Dp}{\Delta\sb{+}}
\newcommand{\Dresp}{\Delta\sb{\res,+}}
\newcommand{\Dn}[1]{\Delta\lolr{#1}}
\newcommand{\Dnop}[1]{\Dn{#1}\op}
\newcommand{\Drn}[1]{\Delta\lolr{#1}}
\newcommand{\Dron}[1]{\Delta\lolr{#1}\op}
\newcommand{\Drnp}[1]{\Delta\lolr{#1+}}
\newcommand{\Dronp}[1]{\Delta\lolr{#1+}\op}
\newcommand{\Du}{\Del\sp{\bullet}}
%
%
\newcommand{\bA}{{\mathbf A}}
\newcommand{\bB}{{\mathbf B}}
\newcommand{\bC}{{\mathbf C}}
\newcommand{\bD}{{\mathbf D}}
\newcommand{\cD}{D}
\newcommand{\bE}{{\mathbf E}}
\newcommand{\be}[1]{{\mathbf e}\sp{#1}}
\newcommand{\bF}{{\mathbf F}}
\newcommand{\Fv}[1]{\bF\bp{#1}}
\newcommand{\Fk}[1]{F\sp{#1}}
\newcommand{\Fpp}{\hspace{0.5mm}'\hspace{-0.7mm}F}
\newcommand{\Fpk}[1]{\hspace{0.5mm}'\hspace{-0.7mm}F\sp{#1}}
\newcommand{\Fn}[2]{F\sp{#1}\bp{#2}}
\newcommand{\bG}{{\mathbf G}}
\newcommand{\Gv}[1]{\bG\bp{#1}}
\newcommand{\hGv}[1]{\widehat{\bG}\bp{#1}}
\newcommand{\Gn}[2]{G\sp{#1}\bp{#2}}
\newcommand{\hGn}[2]{\widehat{G}\sp{#1}\bp{#2}}
\newcommand{\tg}[1]{\widetilde{g}\sp{#1}}
\newcommand{\bH}{{\mathbf H}}
\newcommand{\wH}{\widehat{H}}
\newcommand{\Hv}[1]{\bH\bp{#1}}
\newcommand{\hHv}[1]{\wH\bp{#1}}
\newcommand{\Hn}[2]{H\sp{#1}\bp{#2}}
\newcommand{\tH}[1]{\widetilde{H}\sp{#1}}
\newcommand{\tHn}[2]{\tH{#1}\bp{#2}}
\newcommand{\bi}{{\mathbf i}}
\newcommand{\bj}{{\mathbf j}}
\newcommand{\bK}{{\mathbf K}}
\newcommand{\bL}{{\mathbf L}}
\newcommand{\KP}[2]{\bK({#1},{#2})}
\newcommand{\KR}[1]{\KP{R}{#1}}
\newcommand{\KZ}[1]{\KP{\ZZ}{#1}}
\newcommand{\bM}[1]{{\mathbf M}\sp{#1}}
\newcommand{\bP}{{\mathbf P}}
\newcommand{\bQ}{{\mathbf Q}}
\newcommand{\bS}[1]{{\mathbf S}\sp{#1}}
\newcommand{\vS}{\mathcal S}
\newcommand{\bT}{\mathbf{\Theta}}
\newcommand{\TA}{\bT\sp{\A}}
\newcommand{\TB}{\bT\sb{\B}}
\newcommand{\ThB}{\Theta\sb{\B}}
\newcommand{\TAB}{\bT\sp{\A}\sb{\B}}
\newcommand{\TR}{\Theta\sb{R}}
\newcommand{\TRl}{\Theta\sb{R}\sp{\lambda}}
\newcommand{\bU}{{\mathbf U}}
\newcommand{\bV}{{\mathbf V}}
\newcommand{\bW}{{\mathbf W}}
\newcommand{\bX}{{\mathbf X}}
\newcommand{\hX}{\widehat{\bX}}
\newcommand{\tX}{\widetilde{\bX}}
\newcommand{\bY}{{\mathbf Y}}
\newcommand{\hY}{\widehat{\bY}}
\newcommand{\bZ}{{\mathbf Z}}
\newcommand{\uZ}[1]{\bZ\up{#1}}
\newcommand{\hZ}{\widehat{\bZ}}
%
%
\newcommand{\bdel}{\bar{\delta}}
\newcommand{\eps}[1]{\epsilon\sb{#1}}
\newcommand{\iol}[1]{\iota\sb{#1}}
\newcommand{\iot}[1]{\iota\sp{#1}}
\newcommand{\iotl}[2]{\iota\sp{#1}\lo{#2}}
\newcommand{\iottl}[2]{\iota\sp{#1}\sb{#2}}
\newcommand{\vare}{\varepsilon}
\newcommand{\var}[1]{\vare\sb{#1}}
\newcommand{\tvar}[1]{\widetilde{\vare}\bup{#1}}
\newcommand{\ett}[1]{\eta\sb{#1}}
\newcommand{\veta}{\vec{\eta}}
\newcommand{\blam}{\bar{\lambda}}
\newcommand{\wvarp}{\widehat{\varphi}}
\newcommand{\wrho}{\widehat{\rho  }}
\newcommand{\tS}{\widetilde{\Sigma}}
\newcommand{\wsig}{\widehat{\sigma}}
\newcommand{\gam}[1]{\gamma\lo{#1}}
\newcommand{\gamm}[1]{\gamma\sb{#1}}
%
%
%
\newcommand{\bd}{\mathbf{d}\sb{0}}
\newcommand{\co}[1]{c({#1})\sb{\bullet}}
\newcommand{\cod}[1]{c\sb{+}({#1})\sb{\bullet}}
\newcommand{\cu}[1]{c({#1})\sp{\bullet}}
\newcommand{\szero}[1]{\sigma\sp{\ast}(#1)}
\newcommand{\Cd}[1]{\bC[{#1}]\sb{\bullet}}
\newcommand{\Fd}{\mathcal{F}\sb{\bullet}}
\newcommand{\Gd}{G\sb{\bullet}}
\newcommand{\Gu}{G\sp{\bullet}}
\newcommand{\oT}[1]{\overline{T}\sb{#1}}
\newcommand{\Ud}{U\sb{\bullet}}
\newcommand{\wUd}{\widetilde{U}\sb{\bullet}}
\newcommand{\Vd}{V\sb{\bullet}}
\newcommand{\bVd}{\bV\sb{\bullet}}
\newcommand{\oV}[1]{\overline{\bV}\sb{#1}}
\newcommand{\Wd}{\bW\sb{\bullet}}
\newcommand{\rWd}{W\sb{\bullet}}
\newcommand{\wW}{\widehat{\bW}}
\newcommand{\wWd}{\wW\sb{\bullet}}
\newcommand{\tW}{\widetilde{\bW}}
\newcommand{\tWd}{\tW\sb{\bullet}}
\newcommand{\oW}[1]{\overline{\bW}\sb{#1}}
\newcommand{\Xd}{X\sb{\bullet}}
\newcommand{\hXd}{\hX\sb{\bullet}}
\newcommand{\Yd}{\bY\sb{\bullet}}
\newcommand{\Zd}{\bZ\sb{\bullet}}
\newcommand{\Zud}[1]{\uZ{#1}\sb{\bullet}}
\newcommand{\hZd}{\hZ\sb{\bullet}}
\newcommand{\bk}{[\mathbf{k}]}
\newcommand{\bkm}{[\mathbf{k}-\mathbf{1}]}
\newcommand{\bmm}{[\mathbf{m}]}
\newcommand{\bnp}{[\mathbf{n}+\mathbf{1}]}
\newcommand{\bn}{[\mathbf{n}]}
\newcommand{\bnm}{[\mathbf{n}-\mathbf{1}]}
\newcommand{\bone}{[\mathbf{1}]}
\newcommand{\bze}{[\mathbf{0}]}
\newcommand{\bmo}{[-\mathbf{1}]}
\newcommand{\od}{\overline{\partial}}
\newcommand{\odz}[1]{\od\sb{#1}}
%
%
\newcommand{\PP}[1]{\mathcal{P}\sp{#1}}
%
%
\title{The algebra of higher homotopy operations}
\author{Samik Basu, David Blanc, and Debasis Sen}
\address{Stat-Math Unit\\ Indian Statistical Institute\\ Kolkata 700108\\ India}
\email{samikbasu@isical.ac.in}
\address{Department of Mathematics\\ Faculty of Natural Sciences\\ University of Haifa\\
PO Box 3338\\ 3498838 Haifa\\ Israel}
\email{blanc@math.haifa.ac.il}
\address{Department of Mathematics \& Statistics\\
Indian Institute of Technology, Kanpur\\ Uttar Pradesh 208016\\ India}
\email{debasis@iitk.ac.in}

\date{\today}

\makeatletter
\@namedef{subjclassname@2020}{%
  \textup{2020} Mathematics Subject Classification}
\makeatother
\subjclass[2020]{Primary: 55Q35; \ secondary: 55Q40, 55Q15, 55S30}
\keywords{Higher homotopy operations, Toda brackets, simplicial spaces,
homotopy groups of spheres, higher Whitehead products}

\begin{abstract}
We explain how the simplicial higher order unstable homotopy operations
defined in \cite{BBSenHS} may be composed and inserted one in another,
thus forming a coherent if complicated algebraic structure.
\end{abstract}

\maketitle

\setcounter{section}{0}

%
%
\section*{Introduction}
\label{cint}

Secondary and higher order operations have played an important role in homotopy theory
and other areas of mathematics since Massey products and Toda brackets first appeared
in \cite{MassN} and \cite{TodG}, respectively. Many generalizations of these two examples
have been proposed over the years (see, e.g.,
\cite{SpanH,MaunC,GWalkL,BMarkH,BBGondH,BJTurnC,BBSenT}), but we shall
concentrate here on the simplicial version considered in \cite{BBSenHS}, which
were shown there to generate all elements in the unstable homotopy groups of spheres
(starting with a small set of indecomposables), in the same way that higher Toda
brackets generate the stable homotopy groups from the Hopf maps (as Joel Cohen showed
in \cite{JCohDS}).

In both cases, the sense in which the higher operations \emph{generate} all homotopy
classes was left somewhat opaque, and our goal here is to clarify this point by
defining an \emph{algebra} of higher homotopy operations, in which they are shown to
interact in certain specified ways, beyond the usual primary composition structure.

We start with the definition of the simplicial higher order
homotopy operations from \cite[\S 1]{BBSenHS}, stated in the language of
$\infty$-categories, along the lines indicated in \cite{BMeadS}.
Even though here we are concerned only with higher homotopy operations in the
strict sense \wh that is, acting on the homotopy groups of pointed topological
spaces \wh much of what we do is applicable in wider settings, where we may not
even have a model category structure available.

The interpretation in terms of spectral sequences helps to clarify the relation between
the lower level indeterminacy (in terms of inductive choices of higher homotopies),
and the higher level indeterminacy in terms of the algebraic resolutions or cell
structure (see \cite[\S 3]{BJTurnHH} and compare \cite[\S 3]{BSenM}).

We do not claim these simplicial higher operations are the most general possible \wh
see \cite{BMarkH,BJTurnC} for other options.  Nevertheless, they appear to be
sufficient for the following purposes:

\begin{enumerate}
\renewcommand{\labelenumi}{(\alph{enumi})~}
\item If \w{f:\bS{n}\to\bX} is a value of a higher order operation, we show how addition,
  post- or pre-composition with $f$, and insertion of $f$ in another operation,
  can be expressed as the value of another such operation.
\item We explain how (iterated) Whitehead products can be expressed quite explicitly
  as values of higher operations (in a variety of ways); this is relatively formal.
\item Finally, we use Hilton's Theorem on the general form of the \emph{primary}
  homotopy operations to provide a ``higher Hilton decomposition'' of a general
  higher order operation (using the formal description of the Whitehead product
  component).
\end{enumerate}

As an application of our methods we study the only known regular sequence of
higher order operations: namely, the higher order Whitehead products.
We were able to obtain a closed formula (using our simplicial approach) only for
the rational version, but we conjecture that it may be valid integrally.  Finally,
we indicate how the more general Lie-Massey products may be treated by our methods.

\begin{notn}\label{snac}
Let $\Delta$ denote the category of non-empty finite ordered sets and order-preserving
maps (cf.\ \cite[\S 2]{MayS}), \w{\Dres} the subcategory with the same
objects but only monic maps, and \w{\Drn{n}} the subcategory of \w{\Dres} consisting of
non-empty ordered sets with at most \w{n+1} elements; we denote the ordered set
\w{0<1<\dotsc<n} with \w{n+1} elements by \w[.]{\bn} Similarly, \w{\Dp} is the
category of \emph{all} finite ordered sets, with the empty set denoted by \w[.]{\bmo}
Write \w{\Dresp} for the corresponding subcategory of monic maps, and \w{\Drnp{n}} for
the subcategory of \w{\Dresp} consisting of ordered sets with at most \w{n+1} elements.

A \emph{simplicial object} \w{\Vd} in a category $\C$ is a functor
\w{\Dop\to\C} and an \emph{augmented simplicial object} is a functor
\w[,]{\Dp\op\to\C} usually written \w{\Vd\to X} (where $X$ is the value at
\w[).]{\bmo}

A \emph{restricted} simplicial object
is a functor \w[,]{\Dres\op\to\C} and a \emph{restricted augmented} simplicial object
is a functor \w[;]{\Dresp\op\to\C} in both cases, we have face maps but no degeneracies.
Similarly, an \emph{$n$-truncated restricted} simplicial object is a functor
\w[,]{\Dron{n}\to\C} and an \emph{$n$-truncated restricted augmented} simplicial object
is a functor \w[.]{\Dronp{n}\to\C} There is a natural embedding
\w[,]{\co{-}:\C\to\C\sp{\Dop}} with \w{\co{A}} the constant simplicial object, and
for each \w[,]{n\geq 0} a truncation functor
\w{\tau\sb{n}:\C\sp{\Dresp\op} \to \C\sp{\Dronp{n}}} (and similarly for the unrestricted
or augmented versions).

In each case, \w{V\sb{n}} is the value of \w{\Vd} at \w{\bn}
and \w{d\sb{i}:V\sb{n}\to V\sb{n-1}} for the $i$-th face map (corresponding
to the map \w{\bnm\to\bn} which skips $i$ in the target). For \w{\Vd\to X}
an augmented object, we write \w{\vare:V\sb{0}\to X} for the augmentation (corresponding
to the unique map \w{\bmo\to\bze} in \w[).]{\Dp}
The inclusion \w{\Delta \to \Dp} induces \w{\C\sp{\Dp\op} \to \C\sp{\Dop}} (forgetting
the augmentation), with right adjoint the d\'{e}calage
\w{\Dec:\C\sp{\Dop}\to\C\sp{\Dp\op}} (omitting the last face and degeneracy maps
in each dimension \w{n\geq 0} and decreasing $n$ by one).

There is a similar right adjoint \w{\Dec\sb{\res}:\C\sp{\Dres\op}\to\C\sp{\Dresp\op}}
for restricted simplicial objects. When $\C$ is pointed, we have a functor
\w[,]{T:\C\sp{\Dresp\op} \to \C\sp{\Dres\op}} right inverse to \w[,]{\Dec\sb{\res}}
which adds an extra (zero) face map \w{d\sb{n+1}:V\sb{n}\to V\sb{n-1}}
in each dimension \w{n\geq -1} and increases $n$ by one.

The category of topological spaces will be denoted by \w[,]{\Top}
that of pointed spaces by \w[,]{\Topa} and that of pointed connected
spaces by \w[.]{\Tz} The category of simplicial sets will be
denoted by \w[,]{\Ss=\Set\sp{\Dop}} that of pointed simplicial
sets by \w[,]{\Sa=\Seta\sp{\Dop}} that of reduced simplicial
sets by \w{\Sr} (see \cite[III, \S 3]{GJarS}), and that of simplicial groups
by \w[,]{\G=\Grp\sp{\Dop}} with \w{G:\Sa\to\G} Kan's loop functor
(see \cite[V \S 5]{GJarS}), and \w{\U:\G\to\Sa} the forgetful functor (so \w{\U GK}
is a model for the loop space on $K$). We write \w{\mapa(\bX,\bY)} for the standard
function complex in \w[,]{\Sa} \w[,]{\Tz} or $\G$ (see \cite[I, \S 1.5]{GJarS}).

For any category $\C$ with enough colimits, \w{\C\op} is simplicial (that is
(co)tensored over $\Ss$) in the sense of
\cite[II, \S 1]{QuiH}: in particular, for each \w{K\in\Ss} and \w{\Xd\in\C\op}
the object \w{\Yd:=\Xd\otimes K\in\C\op} has
\w[.]{\bY\sb{n}\cong\coprod\sb{\sigma\in K\sb{n}}\bX\sb{n}} If $\C$ is pointed
and \w[,]{K\in\Sa} the pointed version \w{\ostar{\Xd}{K}} omits \w{\ast\in K\sb{n}}
for each $n$.
\end{notn}

\begin{mysubsection}{Organization}
\label{sorg}
In Section \ref{choo} we recall our earlier definition of (simplicial) higher
order operations in a model-free $\infty$-categorical approach, and in Section
\ref{csr} we recall some additional facts about simplicial resolutions needed
for our constructions.  In Section \ref{chho} we restrict attention to higher
\emph{homotopy} operations (based on maps out of spheres), and describe
the basic procedure for composition of their values.  Section \ref{cwp} deals with
Whitehead products, which are needed for our analysis of Hilton's Theorem in
Section \ref{cahho}, providing a meaning to the notion of an \emph{algebra} of
higher homotopy operations. Section \ref{chowp} deals with the example of the
higher order Whitehead products, while Section \ref{clmp} explains what changes are
needed to deal with the more general Lie-Massey products. Finally, Section \ref{ccpn}
relates our earlier result from \cite[\S 8]{BBSenHS} on complex projective spaces
to the present Section \ref{chowp}.
\end{mysubsection}

%
%
\sect{Higher order operations}
\label{choo}

We start with our definition of simplicial higher order homotopy operations
from \cite{BBSenHS}, stated in a model-independent version of $\infty$-categories:

\begin{mysubsection}{Models of \ww{(\infty,1)}-categories}
\label{sinfcat}
There are a number of versions of \ww{(\infty,1)}-categories, each with its own model
structure, including simplicial categories (see \cite{BergM}), quasi-categories
(see \cite{JoyQ,LurH}), Segal categories (see \cite{SimpH}), and complete Segal
spaces (see \cite{RezkM}). All of these are Quillen equivalent, and each has its
own technical advantages and difficulties (see the survey in \cite{BergnHI}).

However, for our purposes here we need very little of the theory, and can essentially
work with an approach which is independent of the model chosen.
There are in fact several axiomatic formulations of the abstract notion of a
\emph{model} of \ww{(\infty,1)}-category theory (see \cite{ToenA,RVeriI}),
but we shall only need the following:
\begin{enumerate}
\renewcommand{\labelenumi}{(\alph{enumi})~}
\item Our model consists of a category $\M$ (e.g., the category $\Ss$ of simplicial sets,
or the category \w{s\Cat} of small simplicial categories), with a full subcategory
\w{\M\sb{0}} of \emph{$\infty$-categories} (e.g., quasi-categories, or fibrant
simplicial categories).
\item A \emph{homotopy category} functor \w[,]{\ho:\M\sb{0}\to\Cat} with the \emph{nerve}
functor \w{\B:\Cat\to\M\sb{0}} as right adjoint.
\item An \emph{object set} functor \w[;]{\Obj{}:\M\to\Set}
\item For each \w{\C\in\M\sb{0}} and \w[,]{x,y\in\Obj{\C}} a Kan complex
  \w[,]{\Map\sb{\C}(x,y)} equipped with a natural isomorphism
  \w[.]{\pi\sb{0}\Map\sb{\C}(x,y)\cong\Hom\sb{\ho\C}(x,y)}
\end{enumerate}

In addition, we assume:
\begin{enumerate}
\renewcommand{\labelenumi}{\arabic{enumi}.~}
\item Every \w{\C\in\M\sb{0}} has a fibrant simplicial category $\mC$ with
  \w[;]{\Obj{\C}=\Obj{\mC}}
\item For every \w[,]{x,y\in\Obj{\C}} \w{\Map\sb{\C}(x,y)}
is homotopy equivalent to \w[.]{\mC(x,y)}
\end{enumerate}
See \cite[\S 2]{BMeadS} for further details.
\end{mysubsection}

\begin{remark}\label{rinftop}
Note that the composition \w{\Map\sb{\C}(y,z)\times\Map\sb{\C}(x,y)\to\Map\sb{\C}(x,z)}
in $\C$ need be neither well-defined nor strictly associative, so a diagram in $\C$,
indexed by \w[,]{I\in\Cat} say, must come equipped with all higher coherences,
encoded by a functor \w{F:\B I\to\C} in $\M$. These also figure in
the universal property of the (co)limit of $F$, when it exists.

In this paper we are interested only in an $\infty$-category model $\C$ for pointed
topological spaces with all countable (co)limits, which may be identified with the
corresponding homotopical (co)limits in $\mC$ (in the sense of \cite{DHKSmitH}).
In particular $\C$ will be (co)tensored over finite simplicial sets.

Thus in what follows, the reader may simply take $\C$ to be \w{\Topa} with the usual
simplicial enrichment. It is important to keep in mind, however, that in this section
diagrams need not commute on the nose, but instead have specified (higher) homotopies
as needed. They can be strictified, of course, but in the course of doing so the maps
and spaces may change up to homotopy.

In this setting, we may formulate \cite[Definition 1.12]{BBSenHS} as follows:
\end{remark}

\begin{defn}\label{dhoh}
Given an $\infty$-category  model \w{(\M,\M\sb{0})} and a pointed \ii category
\w{\C\in\M\sb{0}} with enough colimits, \emph{initial data} for an \wwb{n+1}st
order operation in $\C$ \wb{n\geq 2} consists of:
\begin{enumerate}
\renewcommand{\labelenumi}{(\alph{enumi})~}
\item An \ii functor \w{\Ud:\B\Dron{n}\to\C} which extends to an augmented restricted
$n$-truncated simplicial object \w[,]{\Ud\to X} for some \w[;]{X\in\Obj\C}
\item A map \w{\hf:A\to U\sb{n}} in $\C$ from a (homotopy) cogroup object $A$,
  such that \w{\hat{d}\sb{i}\circ\hf} is nullhomotopic for any
  \w{\hat{d}\sb{i}:U\sb{n}\to U\sb{n-1}} induced by
  \w{d\sb{i}:\bn\to\bnm} in \w[.]{\Dron{n}}
\end{enumerate}

We then have a unique functor \w{\NA:\B\Dron{n}\to\C} in $\M$ with
\w{\NA(\bn)=A} and \w{\NA(\bk)=0} for any \w[.]{0\leq k<n}

In this case, \emph{total data} for \w{(\Ud\to X,\hf)} as above consists of a map
\w{F:\NA\to\Ud} lifting the unique
\w{\hf\sb{\ast}:[\NA]\to[\Ud]} in \w{(\ho\C)\sp{\Dron{n}}} induced by $\hf$.
This total data induces a \emph{value} for the higher order operation
\w{\llrra{\Ud,\hf}} associated to the initial data \w[,]{\lra{\Ud,\hf}} namely the
homotopy class of the map from \w{\hocolim\NA\to X}
induced by $F$ and the augmentation \w{\vare:\bze\to\bmo} in \w[.]{\Dronp{n}}

As we shall see in \S \ref{shoss} below, the homotopy colimit of \w{\NA}
is \w[.]{\Sigma\sp{n}A}
\end{defn}

\begin{remark}\label{eghoh}
A depiction of a higher order operation appears in \wref[,]{diaghho} where the
\emph{initial data} consists of the bottom row and the left map $\hf$, the
\emph{total data} consists of all the vertical maps but the right one,
and the associated \emph{value} is the induced vertical map on the right out of
\w[,]{\Sigma\sp{n}A} the homotopy colimit of the top diagram (which we put in
parentheses to indicate that it is not actually part of the data).
\mysdiag[\label{diaghho}]{
A\ar@/^1pc/[r]^{0}_{\vdots}  \ar@/_1pc/[r]_{0} \ar[dd]\sb{\hf}& 0  \ar@/^1pc/[r]_{\vdots}
  \ar@/_1pc/[r] \ar[dd] & 0 \ar[dd] & \cdots \ar[dd]\sb{F} & 0 \ar[dd] \ar[r] &
  (\Sigma\sp{n}A)\ar[dd]\sp{(f)}\\
 & & & & \\
U\sb{n} \ar@/^1pc/[r]^{d_{0}}_{\vdots}  \ar@/_1pc/[r]_{d_{n}} &
U\sb{n-1}  \ar@/^1pc/[r]^{d_{0}}_{\vdots}  \ar@/_1pc/[r]_{d_{n-1}} & U\sb{n-2} & \cdots &
U\sb{0} \ar[r]\sp{\vare} & X
}
\noindent The vertical label $F$ indicates the implicit coherence homotopies.

The higher homotopy operations described here are essentially the same as those
constructed in \cite[\S 5.20]{BJTurnHA}, and are a special case of the more general
version defined in \cite[\S 3.23]{BJTurnHC} (for \w[).]{\Gamma'=\Drnp{n}\op}
\end{remark}

\begin{mysubsection}{Higher operations and spectral sequences}
\label{shoss}
In the special case when \w[,]{\C=\Topa} \w{A=\bS{k}} is a sphere, and \w{\Ud=\Wd}
is a (strict) simplicial space, with $\vare$ inducing a weak equivalence
\w[,]{\|\Wd\|\simeq\bX} the vertical map \w{(f)} at the right end of \wref{diaghho}
is induced by the map of simplicial spaces consisting of the rest of
the diagram: more precisely, we consider the $n$-truncated unaugmented restricted
versions of the simplicial spaces in question; if we wish, we could extend the top
diagram to a full simplicial space by adding degeneracies, but these play no role
in the discussion here.

As explained in \cite[\S 8]{DKStovB}, this is precisely how the class
\w{[\hf]\in\pi\sb{k}\bW\sb{n}=E\sp{1}\sb{k,n}} in the homotopy spectral
sequence for \w{\Wd} (see \cite{BFrieH}) is represented in \w[,]{\pi\sb{n+k}\|\Wd\|}
assuming it survives to \w[.]{E\sp{\infty}\sb{k,n}}

More generally, given a (coherent) simplicial object \w{\Ud} in a pointed
\ww{(\infty,1)}-category $\C$ as in \S \ref{sinfcat} and a (homotopy) cogroup
object \w[,]{A\in\C} applying the functor \w{\Map\sb{\C}(A,-)} of \ref{sinfcat}(d)
to \w{\Ud} yields a $\infty$-coherent simplicial space \w{\wWd} (in the sense of
Remark \ref{rinftop}) \wh that is, a functor \w[,]{\DK(\Dop)\to\Sa}
where \w{\DK:s\Cat\to s\Cat} is the Dwyer-Kan cofibrant replacement functor of
\cite[\S 2]{DKanL}.

The diagram \w{\wWd} can be straightened to a strict simplicial space \w{\tWd}
by \cite[Corollary 2.5]{DKSmitH}, with Reedy fibrant replacement \w[.]{\Wd}
We define the homotopy spectral sequence for \w{\lra{\Ud,A}} to be
the spiral spectral sequence of \w{\Wd}  (see \cite[\S 8.4]{DKStovB}).
This converges to \w{\pi\sb{\ast}\Map\sb{\C}(A,\colim\Ud)} if $A$
is compact, $\C$ has enough (co)limits, and each \w{\bW\sb{n}} is connected.
See \cite[\S 3]{BMeadS} for further details.
Note that this spectral sequence depends only on the underlying restricted simplicial
object of \w{\Ud} (see \S \ref{snac}).

As shown in \cite[Theorem 3.11]{BMeadS}, a class
\w{[\hf]\in E\sp{1}\sb{k,n}=\pi\sb{0}\Map\sb{\C}(\Sigma\sp{k}A,U\sb{n})}
in this spectral sequence survives to \w{E\sp{r}\sb{k,n}} if and only if it
fits into a diagram in $\C$
\mysdiag[\label{eqerterm}]{
  \Sigma\sp{k}A  \ar@/^1pc/[r]^{0}_{\vdots}  \ar@/_1pc/[r]_{0} \ar[dd]\sb{\hf}  &
  0  \ar@/^1pc/[r]_{\vdots}
  \ar@/_1pc/[r] \ar[dd] & 0 \ar[dd] & \cdots \ar[dd]\sb{F} & 0 \ar[dd] \\
 & & & & \\
U\sb{n} \ar@/^1pc/[r]^{d_{0}}_{\vdots}  \ar@/_1pc/[r]_{d_{n}} &
U\sb{n-1}  \ar@/^1pc/[r]^{d_{0}}_{\vdots}  \ar@/_1pc/[r]_{d_{n-1}} & U\sb{n-2} & \cdots &
U\sb{n-r+1}.
}

Moreover, if a diagram \wref{eqerterm} exists, there is an associated representative
\w{[g]\in E\sp{1}\sb{k+r-1,n-r}=\pi\sb{r}\Map\sb{\C}(\Sigma\sp{k}A,U\sb{n-r})}
of value \w{d\sb{r}([\hf])} the differential,  which is adjoint to a certain map
\w[.]{\gamma:\partial\PP{r}\to\Map\sb{\C}(\Sigma\sp{k}A,U\sb{n-r})}
Here \w{\PP{r}} is the $r$-dimensional permutohedron, a convex polytope in
\w[,]{\RR\sp{n}} so \w{\partial\PP{r}} is an \wwb{r-1}sphere. We use the fact that
each component of the simplicial set \w{\DK(\Dres\op)(\bn,\bmm)} is an
\wwb{n-m-1}-permutohedron (see \cite[Proposition 5.6]{BMeadS})
to construct $\gamma$ out of the various coherence homotopies $F$
in \wref[,]{eqerterm} both those making the squares commute, and those of \w{\Ud}
itself (although by passing to the strictified simplicial space \w{\Wd} as above,
the latter can be eliminated). This description of the differential \w{d\sb{r}}
is used in the proof of Proposition \ref{pnoratho} below.

The construction of $\gamma$ is inductive, successively using nullhomotopies
for the analogous representatives of \w[,]{d\sb{I}([\hf])} for each $k$-fold face map
\w[,]{d\sb{I}} to determine maps on the $k$-dimensional facets of \w[,]{\PP{r}} and
thus finally on \w[.]{\partial\PP{r}} See \cite[\S 6]{BMeadS} for further details.

In this context, note the following:
\begin{enumerate}
\renewcommand{\labelenumi}{(\alph{enumi})~}
\item In this inductive process, the maps on the successive skeleta
of \w{\PP{r}} are determined by the universal properties of the colimits in $\C$
which define the tensoring of $\C$ over $\Ss$ (see Remark \ref{rinftop}). Thus they
make sense in any model $\M$ of \ww{(\infty,1)}-categories satisfying the minimal
requirements of \S \ref{sinfcat} \wh and are in particular homotopy independent.
\item The fact that \w{\DK(\Dres\op)(\bn,\bmm)} depends only on \w{k=n-m} means that
the calculations of the differentials are the same as the analogous ones for
\w{\DK(\Dres\op)(\bkm,\bmo)} \wwh that is, for diagrams of the form \wref[,]{diaghho}
with \w{U\sb{m}} playing the role of $X$. See \cite[\S 4.2]{BMeadS} for an explanation
of the fact that we need only \emph{one} such shifted diagram.
\item This means that the values of the differentials in the spectral sequence of a
(restricted) simplicial space are also determined by values of higher order operations
in the sense of Definition \ref{dhoh} above. In particular, this allows us to
identify the \emph{indeterminacy} of such operations.
\item Finally, the diagram \wref{eqerterm} may be rewritten as a single
\wwb{n+1}truncated restricted simplicial object in $\C$:
\mysdiag[\label{eqlongsimp}]{
U\sb{n+1} \ar@/^1pc/[r]^{\hf}_{\vdots} \ar@/_1pc/[r]_{0}  &
U\sb{n} \ar@/^1pc/[r]^{d_{0}}_{\vdots} \ar@/_1pc/[r]_{d_{n}} &
U\sb{n-1}  \ar@/^1pc/[r]^{d_{0}}_{\vdots}  \ar@/_1pc/[r]_{d_{n-1}} & U\sb{n-2} & \cdots &
U\sb{n-r+1}
}
\noindent in which all but the first face map of \w{U\sb{n+1}:=\Sigma\sp{k}A} is $0$.
Similarly for \wref{diaghho} (see \cite[\S 3.12]{BMeadS}).
\end{enumerate}
\end{mysubsection}

%
%
\sect{Simplicial resolutions}
\label{csr}

Although our formalism works in the generality of Definition \ref{dhoh}, in fact we
shall be mostly concerned with simplicial spaces \w{\Wd} which are resolutions of a
given space $\bX$ (often itself a sphere).  For this purpose we require the following
notions:

\begin{defn}\label{dmco}
In a pointed and complete category $\C$, the $n$-th \emph{Moore chains} object
of a restricted augmented simplicial object \w{\Gd\in\C\sp{\Dresp\op}} is defined to be:
\begin{myeq}\label{eqmoor}
\cM{n}\Gd~:=~\cap\sb{i=1}\sp{n}\Ker\{d\sb{i}:G\sb{n}\to G\sb{n-1}\}~,
\end{myeq}
\noindent with differential
\w[.]{\dif{n}:=d\sb{0}\rest{\cM{n}\Gd}:\cM{n}\Gd\to\cM{n-1}\Gd}
The $n$-th \emph{Moore cycles} object is \w[.]{\cZ{n}\Gd:=\Ker(\dif{n})}
Write \w{w\sb{n}:\cM{n}\Gd\hra G\sb{n-1}} and
\w{v\sb{n}:\cZ{n}\Gd\hra\cM{n}\Gd} for the inclusions.
\end{defn}

\begin{defn}\label{dlmo}
For a (possibly \wwb{n-1}truncated) simplicial object \w{\Gd\in\C\sp{\Dop}} in a
cocomplete category $\C$, the $n$-th \emph{latching object} for \w{\Gd}
is
\begin{myeq}\label{eqlatch}
L\sb{n}\Gd~:=~\colimit{\theta\op:\bk\to\bn}\,G\sb{k}~,
\end{myeq}
\noindent where $\theta$ ranges over the surjective maps \w{\bn\to\bk} in
$\Del$ (for \w[).]{k < n} There is a natural map
\w{\sigma\sb{n}:L\sb{n}\Gd\to G\sb{n}} induced by the indexing maps $\theta$
of the colimit for any $n$-truncated simplicial object, and any iterated
degeneracy map \w{s\sb{I}=\theta\sb{\ast}:G\sb{k}\to G\sb{n}}
factors as \w[,]{s\sb{I}=\sigma\sb{n}\circ \inc\sb{\theta}}
where \w{\inc\sb{\theta}:G\sb{k}\to L\sb{n}\Gd} is the structure
map for the copy of \w{G\sb{k}} indexed by $\theta$.
\end{defn}

\begin{defn}\label{dscwo}
A \emph{CW object} in a pointed category $\C$ is a simplicial object \w{\Gd\in\C\sp{\Dop}}
equipped with a \emph{CW basis} \w{(\oG{n})\sb{n=0}\sp{\infty}} in $\C$ such that
\w[,]{G\sb{n}=\oG{n}\amalg L\sb{n}\Gd} and
\begin{myeq}\label{eqotherzero}
d\sb{i}\rest{\oG{n}}~=~0\hsm \text{for}\hsm 1\leq i\leq n~.
\end{myeq}
\noindent Observe that the restriction \w{d\sb{0}\rest{\oG{n}}:\oG{n}\to G\sb{n-1}}
of the $0$-th face map factors through \w[,]{\odz{G\sb{n}}:\oG{n}\to\cZ{n-1}\Gd}
which we call the \emph{$n$-th attaching map} for \w[.]{\Gd}

We then have an explicit description of the $n$-th latching object of \w[,]{\Gd}
given by:
\begin{myeq}\label{eqslatch}
L\sb{n}\Gd~:=~
\coprod\sb{0\leq k\leq n-1}~~\coprod\sb{0\leq i\sb{1}<\dotsc<i\sb{n-k-1}\leq n-1}~
\oG{k}~,
\end{myeq}
\noindent where the iterated degeneracy map
\w[,]{s\sb{i\sb{n-k-1}}\dotsc s\sb{i\sb{2}}s\sb{i\sb{1}}} restricted to the
basis \w[,]{\oG{k}} is the inclusion into the copy of \w{\oG{k}} indexed by
$k$ (in the first coproduct) and \w{(i\sb{1},\dotsc,i\sb{n-k-1})} (in the second).
\end{defn}

\begin{remark}\label{rcwr}
As shown in \cite[Theorem 2.29]{BJTurnHH}, any algebraic resolution \w{\Vd\to\pis\bX}
(in the sense of \cite[\S 1]{BJTurnHH}) of a pointed connected space $\bX$ has a
realization by a \emph{CW resolution}: that is, an augmented simplicial space
\w{\Wd\to\bX} with a CW basis \w{(\oW{n})\sb{n=0}\sp{\infty}} such that each
\w{\oW{n}\subseteq\bW\sb{n}} is homotopy equivalent to a wedge of spheres,
with \w[.]{\pis\Wd=\Vd} In this particular case the augmentation
\w{\vare:\|\Wd\|\to\bX} is a weak equivalence.
\end{remark}

\begin{mysubsection}{A combinatorial description}
\label{scomb}
The above discussion provides a homotopy-invariant interpretation
of Kan's ``combinatorial'' approach to homotopy groups (see \cite{KanD}),
in the following sense:

Let $\bK$ be any pointed connected simplicial set, which we may assume to be reduced,
and \w{G\bK} the simplicial group obtained by applying Kan's $G$-functor (a model of
\w{\Omega|\bK|} \wwh see \cite[V,\S 5]{GJarS}).
Thus \w{GK\sb{n}} is a free group, for each \w[,]{n\geq 0} which can be thought
of as a discrete simplicial group \wh a model for a wedge of circles indexed by
the pointed set \w{K\sb{n}} (thought of as a discrete pointed simplicial set) under
Milnor's $F$ functor (a model of \w{\Omega\Sigma|\bK|} \wwh see \cite[V,\S 6]{GJarS}).
Thus applying the classifying space functor (cf.\ \cite[V,\S 4]{GJarS})
dimensionwise to \w{G\bK} yields a simplicial space \w[,]{\Wd:=\overline{W}G\bK}
with each \w{W\sb{n}} a wedge of circles and \w[.]{\|\Wd\|\simeq\bK}

The Bousfield-Friedlander spectral sequence for \w{\Wd} has
$$
\Ett{s}{t}~=~\pi\sp{h}\sb{s}\pi\sp{v}\sb{t}\Wd\cong~
\begin{cases}
  \pi\sb{s+1}\bK & \text{if}~t=1\\
  0&\text{otherwise}
\end{cases}
$$
(see \S \ref{shoss}), so it collapses at the \ww{E\sp{2}}-term.  The formalism of
\cite[\S 3]{BBSenHS} (see also \cite[Theorem 3.11]{BMeadS}) then allows us to interpret
each element of \w{\pi\sb{s+1}\bK} as the value of an \wwb{s+1}st order homotopy
operation associated to the simplicial space \w[.]{\Wd}

Of course, this interpretation has no computational advantage; however, it shows that
the combinatorial description of the homotopy groups of a simplicial group
in \cite{KanD} has a homotopy-invariant meaning. See \cite[(7.9)]{BBSenHS} for an example
(due to Kan).
\end{mysubsection}

At the other extreme, we note the following

\begin{prop}\label{pnoratho}
Let \w{\Wd} be a simplicial space in which each \w{\bW\sb{n}} is weakly equivalent to a
wedge of simply-connected rational spheres. Then the Bousfield-Friedlander
spectral sequence for \w{\Wd} collapses at the \ww{E\sp{2}}-term.
\end{prop}

\begin{proof}
Using the differential graded Lie model for \w[,]{\Wd} we see that each
\w{\bW\sb{n}} is (intrinsically) \emph{coformal} (that is, has a cofibrant model with $0$
differential, namely, \w{\pis\bW\sb{n}} itself, since it is a free graded Lie algebra).
Thus given an element \w{[\alpha]\in\Ett{n}{t}} represented by a Moore cycle
\w[,]{\alpha\in Z\sb{r}\pi\sb{s}\Wd}
we may assume by induction on \w{r\geq 2} that a representative $\hat{f}$ for
$\alpha$ fits into an $\infty$-commutative  diagram of the form
\wref{eqerterm} ending in \w[.]{U\sb{n-r+1}} If we continue this by a
quasi-isomorphism \w{\varphi:U\sb{n-r+1}\to\widehat{U}\sb{n-r+1}} into
a DGL with all differentials equal to $0$, we see that the obstruction to extending
\wref{eqerterm} one more step to the right to \w{U\sb{n-r}} (that is, the
value of the \ww{d\sb{r}}-differential) must vanish. This follows from the
inductive construction of \ww{d\sb{r}(\lra{\hat{f}})} in \cite[\S 6.3]{BMeadS},
in terms of (higher) homotopies between its composites with the face maps of \w{\Wd}
(see \S \ref{shoss} above).
When mapping into a DGL with zero differentials, only identity homotopies exist.

Thus \w{[\alpha]} is in fact represented in degree $s$ of the cycles by
\w{\hat{f}\in Z\sb{n}\Wd} \wwh so that it fits into a diagram
of the form \wref[.]{diaghho}
\end{proof}

\begin{remark}\label{rnoratho}
Of course, when \w{\bX=\|\Wd\|} itself is not coformal, the induced map
\w{f:\Sigma\sp{r}\bS{s}\to\bX} in \wref{diaghho} can still be non-trivial, so that
\w{[f]\in\pi\sb{r+s}\bX} may be the value of a higher homotopy operation, such as
a rational higher Whitehead product (see \S \ref{chowp} below).
\end{remark}

\begin{corollary}\label{cnoratho}
If \w{\Wd} is a simplicial space in which each \w{\bW\sb{n}} is a wedge of
simply-connected spheres, any element \w{[\alpha]\in\Ett{s}{t}} of infinite
order has a multiple which is a permanent cycle in the Bousfield-Friedlander
spectral sequence.
\end{corollary}

%
%
\sect{Combining higher order homotopy operations}
\label{chho}

Because we are interested in higher order operations in the context of homotopy groups,
we now specialize to the case where \w{\C=\Sa} and all the diagrams described above
consist of (finite) wedges of spheres \wh except possibly $\bX$. In this situation
we consider the higher order constructions described above to be higher \emph{homotopy}
operations, in the strict sense of the term.

In \cite[Theorem 7.16]{BBSenHS}, we showed that all elements in the homotopy groups
of a finite wedge of simply-connected spheres are ``generated of higher order'' by the
fundamental classes and their Whitehead products: that is, all such elements
can be obtained recursively from these ``generators'' under the group operation
and composition from values of higher order operations (as defined in \S \ref{dhoh}).
Our goal here is to make this recursive construction more explicit.

\begin{mysubsection}{Addition}
\label{sadd}
When $A$ in Definition \ref{dhoh} is a coproduct \w{A=A\sb{1}\amalg A\sb{2}}
of two homotopy cogroup objects, diagram \wref{diaghho} also splits as a coproduct
of the corresponding diagrams for \w{A\sb{1}} and \w[,]{A\sb{2}} since the choices of
the coherence homotopies restricted to each summand are independent of each other.
More precisely, since by \cite[Corollary 2.5]{DKSmitH} we may assume \w{\Ud}
of \wref{diaghho} has been strictified, we may use the universal property of the
coproduct to identify the simplicial space \w{\Map\sb{\C}(\Sigma\sp{k}A,\Ud)} with
\w[,]{\Map\sb{\C}(\Sigma\sp{k}A\sb{1},\Ud)\times\Map\sb{\C}(\Sigma\sp{k}A\sb{2},\Ud)}
and use this identification to split \wref{diaghho} as a product accordingly.

In particular, when \w{\bA\sb{1}=\bA\sb{2}=\bA} is a sphere (or a wedge of spheres) and
\w{\nabla:\bA\to\bA\vee\bA} is the fold map (the cogroup structure map), it induces
the fold map on realizations of
\w[.]{\ostar{\bA}{S\sp{m}}\to\ostar{(\bA\vee\bA)}{S\sp{m}}} Postcomposing $\nabla$
with two maps \w{\hf\sb{1}\bot\hf\sb{2}:\bA\vee\bA\to\bW\sb{m}} thus yields
the sum of any two corresponding values \w{f\sb{i}:\Sigma\sp{m}\bA\to\bX} \wb[.]{i=1,2}
\end{mysubsection}

The next step is to interpret the composition of two maps when at least one of the
two is the value of a higher operation:

\begin{mysubsection}{Composition}
\label{spoc}
First note that if \w{[f]\in[\Sigma\sp{m}\bA,\bX]} is a value of a higher order
operation associated to the initial data \w[,]{\lra{\Wd\to\bX,\hf:\bA\to\bW\sb{m}}}
postcomposing $f$ with a map \w{g:\bX\to\bY} takes a particularly simple form:
we need only compose the augmentation \w{\vare:\bW\sb{0}\to\bX}
with $g$ to obtain a new initial data \w[,]{\lra{\Wd\to\bY,\hf:\bA\to\bW\sb{m}}}
with the same total data, and the corresponding value will be \w[.]{[g\circ f]}

On the other hand, to precompose a value \w{f:\Sigma\sp{m}\bA\to\bX} of a higher
operation with a given map \w[,]{h:\bB\to\Sigma\sp{m}\bA} we make use of our
assumption that $\bA$ and $\bB$ are (wedges of) spheres:

This implies that \w{\Sigma\sp{m}\bA} can be resolved by the simplicial space
\w{\Zd=\ostar{\bA}{S\sp{m}}} (see \S \ref{snac}), with \w{\bZ\sb{i}=0} for \w[,]{i<m}
\w[,]{\bZ\sb{m}=\bA} and \w{\bZ\sb{j}} determined by the degeneracies on $\bA$
(see \cite[\S 7.1]{BBSenHS}). The class \w{[h]} is represented in the
homotopy spectral sequence for \w{\Zd} in some filtration \w{r\geq m} by a map
\w[,]{\hh:\bB'\to\bZ\sb{r}} where \w{\Sigma\sp{r}\bB'=\bB} (which we may take to
be a single sphere). Thus \w{[h]} is a value of
the \wwb{r+1}st order homotopy operation associated to the
initial data \w[,]{\lra{\Zd,\hh}} and total data which include coherences $H$.
\end{mysubsection}

\begin{defn}\label{dsplice}
For any pointed connected $\bA$, let \w{\Zd} be the resolution \w{\ostar{\bA}{S\sp{m}}}
of \w{\Sigma\sp{m}\bA} (in the sense of \S \ref{rcwr}), with \w[,]{\bZ\sb{m}=\bA}
and let \w{[f]\in[\Sigma\sp{m}\bA,\bX]} be a value of an \wwb{m+1}th order
operation associated to the initial data \w{\lra{\Wd\to\bX,\hf:\bA\to\bW\sb{m}}}
as above. The \emph{splicing} of \w{\Zd} and \w{\Wd} along $\hf$, denoted by
\w[,]{\Zd\propto\Wd} is defined to be the restricted $r$-truncated simplicial object
\w[,]{\Yd:\Dronp{r}\to\C} augmented to $\bX$, defined \w{\bY\sb{k}:=\bW\sb{k}} for
\w[,]{k<m} \w{\bY\sb{k}:=\bZ\sb{k}} for \w[,]{r\geq k\geq m} with face maps equal to
the given ones for \w{\Wd} and \w[,]{\Zd} except for \w[,]{d\sb{i}\sp{\Yd}} which is
the composite \w{d\sb{0}\sp{\Wd}\circ\hf:\bZ\sb{m}\to\bW\sb{m-1}} for
\w[,]{i=0} and $0$ for \w[.]{1\leq i\leq m}
\end{defn}

\begin{lemma}\label{lsplice}
The simplicial identities hold in \w[.]{\Zd\propto\Wd}
\end{lemma}

\begin{proof}
We need only check:
\begin{enumerate}
\renewcommand{\labelenumi}{(\alph{enumi})~}
\item For \w[,]{d\sb{i}\circ d\sb{j}:\bY\sb{m}\to\bY\sb{m-2}} which is $0$ for
  \w{0\leq i<j} and thus equal to \w{d\sb{j-1}\circ d\sb{i}} unless \w[,]{i=j=0}
  in which case we must use the coherences $F$.
\item For \w[,]{d\sb{i}\circ d\sb{j}:\bY\sb{m+1}\to\bY\sb{m-1}} which is $0$ for
  \w{0<i<j} and thus equal to \w{d\sb{j-1}\circ d\sb{i}} by the universal property
  of the zero map in \ii categories.

  If \w{i=0<j} we have
\w{d\sb{0}\sp{\Yd}\circ d\sb{j}\sp{\Yd}=d\sb{0}\sp{\Wd}\circ\hf\circ d\sb{j}\sp{\Zd}}
which is homotopic via $F$ to
\w[.]{0\circ d\sb{i}\sp{\Zd}\circ d\sb{j}\sp{\Zd}=
  0\circ d\sb{j-1}\sp{\Zd}\circ d\sb{i}\sp{\Zd}}
\end{enumerate}
Here we use the coherences for \w[,]{\Zd} if we do not wish to assume that
it is strict.
\end{proof}

\begin{corollary}\label{csplice}
If \w{[f]\in[\Sigma\sp{m}\bA,\bX]} is a value of an \wwb{m+1}th order
operation associated to the initial data \w[,]{\lra{\Wd\to\bX,\hf}} for some wedge
of spheres $\bA$, and \w{[h]\in[\bB,\Sigma\sp{m}\bA]} is a value of
the \wwb{r+1}st order homotopy operation associated to initial data
\w[,]{\lra{\Zd=\ostar{\bA}{S\sp{m}},\hh}} then \w{f\circ h} is a value of the
\wwb{r+1}st order operation associated to the initial data
\w[,]{\lra{\Zd\propto\Wd\to\bX,\hh}} with total data given by the coherences
$H$ for $h$.
\end{corollary}

This is illustrated in Figure \wref[:]{diagchho}

\myrdiag[\label{diagchho}]{
\bB' \ar@/^1pc/[r]^{0}_{\vdots} \ar@/_1pc/[r]_{0} \ar[dd]\sb{\hh} & \hsp 0\ar[dd]\dotsc
& 0 \ar@/^1pc/[r]^{0}_{\vdots}  \ar@/_1pc/[r]_{0}  \ar[dd] &
  \hs 0 \ar[dd] & \cdots \ar[dd]\sb{H} & 0 \ar[dd] \ar[r] &
(\Sigma\sp{r}\bB'=\bB) \ar[dd]\sp{(h)}\\
&& & & & &\\
\bZ\sb{r} \ar@/^1pc/[r]^{d_{0}}_{\vdots}\ar@/_1pc/[r]_{d_{r}} & \hsp\bZ\sb{r-1}\dotsc &
\bZ\sb{m}=\bA \ar@/^1pc/[r]^{0}_{\vdots}  \ar@/_1pc/[r]_{0} \ar[dd]\sb{\hf}  &
\hs 0 \ar[dd] & \cdots \ar[dd]\sb{F} & 0 \ar[dd] \ar[r] &
(\Sigma\sp{m}\bA) \ar[dd]\sp{(f)}\\
&& & & & & \\
&& \bW\sb{m} \ar@/^1pc/[r]^{d_{0}}_{\vdots}  \ar@/_1pc/[r]_{d_{n}} &
\bW\sb{m-1} & \cdots & \bW\sb{0} \ar[r]\sp{\vare} & \bX
}

\begin{remark}\label{rprchho}
When \w[,]{r=m} the \wwb{m+1}st order homotopy operation yielding \w{[h]} is simply
the $m$-fold suspension of \w{\widehat{h}} (see \cite[Remark 7.3]{BBSenHS}), and in
this case we can more directly represent \w{h\sp{\ast}[f]} by the \wwb{m+1}st order
homotopy operation associated to the initial data \w[.]{\lra{\Wd\to\bX,\hf\circ\hh}}
\end{remark}

\begin{mysubsection}{Values in initial data}
\label{svid}
If \w{[f]\in[\Sigma\sp{n}\bA,\bX]} is a value of an \wwb{n+1}st order homotopy
operation with initial data \w[,]{\lra{\Wd\to\bX,\hf:\bA\to\bW\sb{n}}} and
\w{\hf:\bA\to\bW\sb{n}} is itself a value of an $r$-th order operation with
initial data \w{\lra{\Zd\to\bW\sb{n},\hg}} we must have \w{\hg:\bB\to\bZ\sb{r}} for
\w[.]{\Sigma\sp{r}\bB=\bA}

Figure \wref{eqvid} then exhibits $f$ as the value of an \wwb{n+r+1}st order
homotopy operation with initial data \w{\lra{T\sp{n}\Zd\propto\Wd\to\bX,\hg}}
(see \S \ref{snac}), which is a combination of the two operations:

\myrdiag[\label{eqvid}]{
  \bB \ar@/^0.8pc/[r]\sp{0}\sb{\vdots}  \ar@/_1pc/[r]_{0} \ar[dd]\sb{\hg} &
  \hsp 0\ar[dd]\dotsc
 & 0 \ar@/^1pc/[r]^{0}_{\vdots}  \ar@/_1pc/[r]_{0}  \ar[dd] &
  \hs 0 \ar[dd] & \cdots \ar[dd]\sb{H} & 0 \ar[dd] \ar[r] &
(\Sigma\sp{r+n}\bB=\Sigma\sp{n}\bA) \ar[dddd]\sp{(f)}\\
&& & & & &\\
  \hZ\sb{r+n} \ar@/^1.4pc/[r]\sb{\stackrel{d\sb{0}}{\vdots}} \ar@/_1pc/[r]\sb{d\sb{r+n}} &
  \hsp \hZ\sb{r+n-1}\dotsc &
\hZ\sb{n} \ar@/^0.8pc/[r]\sp{0}\sb{\vdots}   \ar@/_1pc/[r]_{0} \ar[dd]\sb{\hf=g}  &
\hs 0 \ar[dd] & \cdots \ar[dd]\sb{F} & 0 \ar[dd] \\
&& & & & & \\
&& \bW\sb{n} \ar@/^1.4pc/[r]\sb{\stackrel{d\sb{0}}{\vdots}}  \ar@/_1pc/[r]_{d_{n}} &
\bW\sb{n-1} & \cdots & \bW\sb{0} \ar[r]\sp{\vare} & \bX
}
\end{mysubsection}

%
%
\sect{Whitehead products}
\label{cwp}

There is a formal aspect to the decomposition of higher homotopy operations into their
constituents \wh namely, the part involving Whitehead products \wh since their description
as values of higher operations uses only the combinatorics of the degeneracy maps
in simplicial resolutions of wedges of spheres. We first recall:

\begin{notn}\label{nshuffle}
  For integers \w[,]{n>k\geq 1}  we denote by \w{\cI{n}{k}} the
  set of all \wwb{k,n-k}\emph{shuffles} \wh  that is, all permutations
  \w{\sigma=(\sigma',\sigma'')} in \w{S\sb{n}} for
  \w{\sigma'=(\sigma'\sb{1},\dotsc,\sigma'\sb{k})} and
  \w{\sigma''=(\sigma''\sb{1},\dotsc,\sigma''\sb{\ell})} \wb{\ell=n-k}
  a partition of \w[,]{\{1,\dotsc,n\}}
  with \w{\sigma'\sb{1}<\sigma'\sb{2}<\dotsc<\sigma'\sb{k}} and
  \w[.]{\sigma''\sb{1}<\sigma''\sb{2}<\dotsc<\sigma''\sb{\ell}}  We set
\begin{myeq}\label{eqwi}
\wI{n}{k}~:=~\begin{cases}
\{(\sigma',\sigma'')\in\cI{n}{k}~:\ \sigma'\sb{1}=1\}&\text{if}\  n=2k\\
\cI{n}{k}&\text{otherwise}.
\end{cases}
\end{myeq}

We denote by \w{\hI{n}{k}} the same set of \wwb{k,n-k}shuffles as \w[,]{\cI{n}{k}}
expressed in terms of the corresponding multi-indices \w{(I,J)} with
\w{I=(i\sb{1},\dotsc,i\sb{k})} and \w{J=(j\sb{1},\dotsc,j\sb{\ell})} for
with \w{i\sb{p}=\sigma'\sb{p}-1} and \w{j\sb{q}=\sigma''\sb{q}-1}
(which we may think of as a partition of \w[).]{\{0,1,\dotsc,n-1\}}
We write \w{|I|:=k} for the cardinality of the underlying set $\uI$ of $I$.

The corresponding iterated degeneracies are
\w{s\sb{I}=s\sb{i\sb{k}}\dotsc s\sb{i\sb{1}}} and
\w[.]{s\sb{J}=s\sb{j\sb{\ell}}\dotsc s\sb{j\sb{1}}}
More generally, for  $I$ a set of $k$ natural numbers, we denote by
\w{\lra{I}=(i\sb{1}<\dotsc<i\sb{k})} the elements of $I$ arranged in ascending
order, and let
\w{s\sb{I}=s\sb{i\sb{k}}s\sb{i\sb{k-1}}\dotsc s\sb{i\sb{2}}s\sb{i\sb{1}}} be the
corresponding iterated degeneracy map. Note that
\begin{myeq}\label{eqaddtodeg}
  \forall\,x\ \left(\exists\,y\ s\sb{I}x=s\sb{j}y\hsm
  \text{if and only if}\hsm j\in I\right)~.
\end{myeq}

The permutation corresponding to \w{(I,J)} is denoted by
\w[,]{\gam{I,J}\in\cI{n}{k}} and we abbreviate \w{\sgn{\gam{I,J}}} (the sign of
this permutation) to \w[.]{\sgn{I,J}}
\end{notn}

\begin{mysubsection}{Wedges of spheres}
\label{swp}
Hilton's Theorem (see \cite{HilS}) states that if
\w{\bW=\bigvee\sb{i\in I}~\bS{n\sb{i}}} is a wedge of simply-connected spheres,
then any element in \w{\pi\sb{N}\bW} can be written as a sum of elements of the form
\w{\alpha\sp{\#}\omega(\iota\sb{1},\dotsc,\iota\sb{k})} where
\w{\iota\sb{j}:\bS{n\sb{j}}\hra\bW} is the inclusion of an individual wedge
summand, $\omega$ is an iterated Whitehead product (given explicitly in terms of
a Hall basis for the free Lie algebra on the graded index set $I$), and
\w{\alpha:\bS{N}\to\bS{M}} is a map between individual spheres
(with \w[).]{M-1=\sum\sb{j}\ (n\sb{j}-1)}

\begin{remark}\label{rwhp}
Given \w{\alpha\sb{i}\in\pi\sb{p\sb{i}+1}\bX} in a pointed space $\bX$
with \w{p\sb{i}\geq 1} \wb[,]{i=1,2,3} in this section it will be convenient to
use the loop space grading, with \w[.]{\deg(\alpha\sb{i})=p\sb{i}}
With this convention, the anti-commutativity of the Whitehead product is given by
\begin{myeq}\label{eqanticomm}
[\alpha,\beta]~=~(-1)\sp{(\deg(\alpha)+1)(\deg(\beta)+1)}\cdot[\beta,\alpha]
\end{myeq}
\noindent (see \cite[X, (7.5)]{GWhE}), and the Jacobi identity takes the form:
\begin{myeq}\label{eqjacidentity}
\begin{split}
(-1)\sp{(\deg(\alpha)+1)(\deg(\gamma)+1)}&[[\alpha,\beta],\gamma]~+~
    (-1)\sp{(\deg(\alpha)+1)(\deg(\beta)+1)}[[\beta,\gamma],\alpha]\\
  &\hsp+~(-1)\sp{(\deg(\beta)+1)(\deg(\gamma)+1)}[[\gamma,\alpha],\beta]~=~0
\end{split}
\end{myeq}
\noindent (see \cite[X, (7.14)]{GWhE}).
\end{remark}

We have the following general expression for representing the ordinary Whitehead
product rationally as a value of a higher order operation:

\begin{lemma}\label{lswedge}
Given \w[,]{p,q\geq 2} \w[,]{k\geq\ell\geq 0} and \w[,]{\bX:=\bS{p+k}\vee\bS{q+\ell}}
the Whitehead product \w{\omega=[\iot{p+k},\iot{q+\ell}]} in
\w{\pi\sb{p+q+k+\ell-1}\bX} is represented rationally in filtration \w{n=k+\ell} of
the spectral sequence for
\w[,]{\Wd=(\ostar{\bS{p}}{S\sp{k}})\vee(\ostar{\bS{q}}{S\sp{\ell}})} by
\begin{myeq}\label{eqswedge}
\homeg~:=~\sum\sb{(I,J)\in\hI{n}{k}}\ \sgn{I,J}\cdot
         [s\sb{I}\iot{p}, s\sb{J}\iot{q}]~,
\end{myeq}
\noindent where \w[.]{\|\Wd\|\simeq\bX}
\end{lemma}

\begin{proof}
If we write \w{\Wd':=\ostar{\bS{p}}{S\sp{k}}} and \w[,]{\Wd'':=\ostar{\bS{q}}{S\sp{\ell}}}
then by \cite{HilS}, the inclusions \w{j':\Wd'\hra\Wd} and \w{j'':\Wd''\hra\Wd}
induce split monomorphisms in the \ww{E\sp{1}}-terms of the corresponding
``suspension'' spectral sequences (cf.\ \cite[\S 6]{BBSenHS}), with
\w{\|\Wd'|\simeq\bS{p+k}} and \w[.]{\|\Wd''|\simeq\bS{q+\ell}}
The direct sum \w{\Image(j'\sb{\#})\oplus\Image(j''\sb{\#})} embeds in the spectral
sequence of \w[,]{\Wd} with complement called the \emph{cross-term}
(containing $\homeg$).  Thus $\omega$, which is in the cross-term for
the inclusions of \w{\pis\bS{q+\ell}} and \w{\pis\bS{p+k}} in \w[,]{\pis\bX}
must be represented non-trivially in the cross-term of the spectral sequence.

Each adjacent transposition switching \w{(m,m+1)} between $I$ and $J$ yields a new
\wwb{k,\ell}shuffle \w{(I',J')} with opposite sign but same face \w{d\sb{m+1}} on the
corresponding summand in \wref{eqswedge} (see Step 1 in the proof of
\cite[Proposition 8.18]{BBSenHS}). This shows that $\homeg$ is a Moore $n$-chain.
In fact, $\homeg$ is an $n$-cycle, since \w{d\sb{0}\iot{q}=0=d\sb{0}\iot{p}} and one
of \w{s\sb{I}} and \w{s\sb{J}} must omit \w[.]{s\sb{0}} It cannot be hit by a
\ww{d\sp{1}}-differential, because (as noted above) its source would be a sum of
degenerate elements. Therefore, \w{[\homeg]} represents $\omega$ rationally,
so there can be no other rational permanent cycles contributing to the cross-term in
total degree \w[,]{N=p+q+k+\ell-1} since \w{\pi\sb{N}\bX} is infinite cyclic.
Moreover, by Proposition \ref{pnoratho} there can be no
other classes of infinite order  in the cross-term of the integral \ww{E\sp{2}}-term
in total degree $N$, since the rational spectral sequence collapses.
\end{proof}

\begin{example}\label{egswedge}
For \w{k=\ell=1} we have
\begin{myeq}\label{eqoneone}
\homeg\sb{(1,1)}~=~[s\sb{0}i\sb{p},\,s\sb{1}i\sb{q}]-[s\sb{1}i\sb{p},\,s\sb{0}i\sb{q}]
~,
\end{myeq}
\noindent as in \cite[(7.7)]{BBSenHS}.

For \w{k=2} and \w{\ell=1} we have
\begin{myeq}\label{eqtwoone}
\homeg\sb{(2,1)}~=~[s\sb{1}s\sb{0}i\sb{p},\,s\sb{2}i\sb{q}]~-~
 [s\sb{2}s\sb{0}i\sb{p},\,s\sb{1}i\sb{q}]~+~[s\sb{2}s\sb{1}i\sb{p},\,s\sb{0}i\sb{q}]~.
\end{myeq}
\end{example}

\begin{prop}\label{sivwp}
Some multiple \w{K\cdot[\homeg]} is a permanent cycle in the homotopy spectral sequence
for \w{\Wd} of Lemma \ref{lswedge}, representing the Whitehead product
\w{\omega=[\iot{p+k},\iot{q+\ell}]} integrally.
\end{prop}

\begin{proof}
This follows from Corollary \ref{cnoratho}.
\end{proof}

\begin{remark}\label{rivwp}
Note further that \w{K\neq\pm1} only if \w{[\homeg]} itself supports a non-trivial
\ww{d\sb{r}}-differential for some\w[,]{r\geq 2} which hits a cross-term element
surviving to the \ww{E\sp{r}}-term (since we know that the spectral sequences for
\w{\Wd'} and \w{\Wd''} converge to \w{\pis\bS{p+k}} and \w[,]{\pis\bS{q+\ell}}
respectively, and their differentials also appear in the spectral sequence
for \w[).]{\Wd}

By Hilton's Theorem, the cross-term summand of
\w{E\sp{1}\sb{s,t}=\pi\sb{t}\bW\sb{s}}
for \w{\Wd} splits as a direct sum of summands of the form
\w[,]{G\sp{\alpha}\sb{w}(I\bup{1},\dotsc,I\bup{m})} where
\begin{enumerate}
\renewcommand{\labelenumi}{(\roman{enumi})~}
\item \w{w=w\sb{i}(\iot{p},\iot{q})} is an $m$-fold iterated Whitehead product
in \w{\iot{p}} and \w{\iot{q}}  of the form \w{[\dotsc[\iot{p},\iot{q}],\dotsc]}
(in a chosen Hall basis for the free graded Lie algebra generated by
\w{\iot{p}} and \w[);]{\iot{q}}
\item \w{I\bup{1},\dotsc,I\bup{m}} are subsets of
\w[,]{\{0,\dotsc,r-1\}} with
\w{\widehat{w}:=[\dotsc[s\sb{I\bup{1}}\iot{p},s\sb{I\bup{2}}\iot{q}],\dotsc]}
landing in \w[;]{\pi\sb{j}\bW\sb{r}}
\item $\alpha$ is a cyclic generator of \w{\pi\sb{t}\bS{j}} (the identity,
  if \w[),]{t=j} so \w{\alpha\sp{\#}\widehat{w}} is a generator
  of a cyclic summand in \w[.]{\pi\sb{t}\bW\sb{r}}
\end{enumerate}

Recall that we may disregard degenerate elements when calculating the homology of
the chain complex associated to the simplicial abelian group \w{\pi\sb{t}\Wd}
(yielding the \ww{E\sp{2}}-term of the spectral sequence).
We may assume \w[,]{\bigcap\sb{i=1}\sp{m}I\bup{i}=\emptyset} since
otherwise $\widehat{w}$ (and thus \w[)]{\alpha\sp{\#}\widehat{w}} is degenerate.
Furthermore, \w{d\sb{j}\widehat{w}=0} unless each of \w{I\bup{1},\dotsc,I\bup{m}}
is empty or contains $j$ or \w[.]{j-1}

By Proposition \ref{pnoratho}, any nonvanishing rational \w{\Ett{s}{t}} cross-term
elements must survive to the \ww{E\sp{\infty}}-term. However, Hilton's Theorem implies
that the first cross-term element occurs for \w[,]{s=k+\ell} since all others are
compositions of primary operations with $\omega$ of Lemma \ref{lswedge}.
Thus for \w{t\geq p+q-1} and \w[,]{0\leq s<k+\ell} the rational \w{\Ett{s}{t}}
cross-term vanishes, so the alternating sum of the ranks of the non-degenerate
cross-term of \w{\Eot{\ast}{t}} must vanish. This allows us to inductively calculate
the dimension of the subgroup of \ww{d\sb{1}}-cycles in the \w{\Eot{\ast}{t}}
cross-term for \w{t=p+q-1} (and similarly for the degrees of any Hall basis element).

If we can show that each such cycle (and not only some multiple thereof) is a
\ww{d\sb{1}}-boundary, we deduce that the \w{\Ett{s}{t}} cross-term vanishes
\emph{integrally}, too, and thus \w{[\homeg]} itself is a permanent cycle.
We conjecture that this is true.
\end{remark}

\begin{example}\label{egintegralv}
Some evidence for the conjecture is provided by low-dimensional calculations:
the first case where \w{[\homeg]} could support a differential is \w[.]{k=\ell=2}
Since \w{d\sb{j}[s\sb{0}\iot{p},s\sb{1}\iot{q}]=[\iot{p},\iot{q}]} if \w[,]{j=1}
and otherwise is $0$, the same is true after applying \w[,]{\eta\sp{\#}} so
the cross-term of \w{\Ett{p+q}{2}} is zero, and thus \w{[\homeg]}
is a permanent cycle.

If \w[,]{k=\ell=3} the \ww{\partial\sb{4}}-cycles in the cross-term of \w{\Eot{4}{p+q-1}}
(for \w[)]{\partial\sb{n}:=\sum\sb{i=0}\sp{n}(-1)\sp{i}\,d\sb{j}} are generated by
\begin{equation*}
\begin{split}
[s\sb{0}\iot{p},s\sb{3}\iot{q}]=&~
\partial\sb{4}(x)\hsm\text{for}\hsm x=[s\sb{1}s\sb{0}\iot{p},s\sb{4}s\sb{2}\iot{q}])\\
[s\sb{0}\iot{p},s\sb{2}\iot{q}]=&~\partial\sb{4}(y)\hsm\text{for}\hsm
y=[s\sb{1}s\sb{0}\iot{p},s\sb{3}s\sb{2}\iot{q}]\\
[s\sb{1}\iot{p},s\sb{3}\iot{q}]=&~
\partial\sb{4}([s\sb{2}s\sb{0}\iot{p},s\sb{4}s\sb{1}\iot{q}]-x)\\
[s\sb{1}\iot{p},s\sb{2}\iot{q}]+&[s\sb{0}\iot{p},s\sb{1}\iot{q}]
     =\partial\sb{4}([s\sb{2}s\sb{0}\iot{p},s\sb{3}s\sb{1}\iot{q}]-y)\\
[s\sb{1}\iot{p},s\sb{2}\iot{q}]+&[s\sb{1}\iot{p},s\sb{0}\iot{q}]
     =\partial\sb{4}([s\sb{2}s\sb{1}\iot{p},s\sb{3}s\sb{0}\iot{q}])\\
[s\sb{2}\iot{p},s\sb{3}\iot{q}]-&[s\sb{0}\iot{p},s\sb{1}\iot{q}]
     =\partial\sb{4}([s\sb{3}s\sb{0}\iot{p},s\sb{4}s\sb{1}\iot{q}])~,
\end{split}
\end{equation*}
\noindent with analogous expressions when \w{\iot{p}} and \w{\iot{q}} are interchanged,
all are boundaries. The same is true in \w{\Eot{4}{p+q}}
after applying \w[,]{\eta\sp{\#}} so the cross-term of \w{\Ett{4}{p+q}} is zero,
and thus \w{[\homeg]} survives to \w[.]{\Eu{3}{4}{p+q-1}} To see that it is a
permanent cycle, we verify that \w{\Ett{3}{p+q+1}} is also zero:
the possible candidates in \w{\Eot{3}{p+q+1}} are of the form
\w[,]{(\eta\sp{2})\sp{\#}[\iot{p},\iot{q}]} hit by
\w[,]{(\eta\sp{2})\sp{\#}[s\sb{0}\iot{p},s\sb{1}\iot{q}]}
or a triple Whitehead product such as \w[,]{[[\iot{p},\iot{q}],\iot{q}]}
hit by \w[,]{[[s\sb{0}\iot{p},s\sb{1}\iot{q}],s\sb{1}\iot{q}]} and similarly for
\w{\eta\sp{\#}[[\iot{p},\iot{q}],\iot{q}]} or a four-fold Whitehead product.
\end{example}

\begin{corollary}\label{cswedge}
For of any collection \w{\{\bS{n\sb{i}}\}\sb{i=1}\sp{N}} of simply-connected spheres,
and \w[,]{\bX:=\bigvee\sb{i=1}\sp{k}\,\bS{n\sb{i}+k\sb{i}}} by repeated applications
of \wref{eqswedge} we may obtain recursive rational representatives in the
spectral sequence for \w[,]{\Wd:=\bigvee\sb{i=1}\sp{k}\ostar{\bS{n\sb{i}}}{S\sp{k\sb{i}}}}
converging to \w[,]{\pis\bX} of any iterated Whitehead product \w{\omega\in\pis\bX}
on fundamental classes (well-defined up to sign, under suitable assumptions on the
numbers \w{n\sb{i}} and \w[).]{k\sb{i}}
\end{corollary}

\begin{example}\label{egswedg}
If \w{\bW'=\bigvee\sb{i\in I}~\bS{n\sb{i}}} is a wedge of (simply connected) spheres,
\w[,]{m\geq 2} and \w[,]{r\geq 1} with \w[,]{\bW=\bW'\vee\bS{m+r}} then for
any Hall basis iterated Whitehead product
\w{\omega(\iot{m+r},\iot{i\sb{1}},\dotsc,\iot{i\sb{k}})}
in \w{\pis\bW} involving \w{\iot{m+r}:\bS{m+r}\hra\bW} and
\w{\iot{i\sb{j}}:\bS{n\sb{i\sb{j}}}\hra\bW'\subseteq\bW} can be represented
(uniquely, up to sign) by an \wwb{r+1}st order homotopy operation for
\w{\Zd:=(\ostar{\bS{m}}{S\sp{r}})\vee\co{\bW'}} (see \S \ref{snac}).
\end{example}

\begin{remark}\label{rswedge}
If a certain fundamental class \w{\iot{m}} is repeated more than once in $\omega$,
add further copies of \w{\bS{m}} to $\bX$ to obtain an iterated product without
repetitions. Postcomposing the higher operation described in Corollary \ref{cswedge}
by suitable fold maps yields the required expression for the original $\omega$.
\end{remark}

\begin{example}\label{egcwedge}
For \w{\bX=\bS{p}\vee\bS{q}\vee\bS{r}} and
\w[,]{\omega=[[\iot{p+1},\iot{q+1}],\iot{r+1}]\in\pi\sb{p+q+r+1}\Sigma\bX}
we see that \wref{eqoneone} defines a map of simplicial spaces
\w[,]{\ostar{\bS{p+q-1}}{S\sp{2}}\to(\ostar{\bS{p}\vee\bS{q})}{S\sp{1}}}
and thus a map
\w[.]{f:(\ostar{\bS{p+q-1}}{S\sp{2}})\vee(\ostar{\bS{r}}{S\sp{1}})\to
\ostar{(\bS{p}\vee\bS{q}\vee\bS{r})}{S\sp{1}}}

Similarly, \wref{eqtwoone} defines a map
\w[,]{g:\ostar{\bS{p+q+r-2}}{S\sp{3}}\to
  (\ostar{\bS{p+q-1}}{S\sp{2}})\vee(\ostar{\bS{r}}{S\sp{1}})} and we may use
the composite \w{f\circ g} to obtain the formula
\begin{myeq}\label{eqtwotwotwo}
\begin{split}
  \homeg=&
        [[s\sb{1}s\sb{0}i\sb{p},s\sb{2}s\sb{0}i\sb{q}],s\sb{2}s\sb{1}i\sb{r}]
      -[[s\sb{1}s\sb{0}i\sb{p},s\sb{2}s\sb{1}i\sb{q}],s\sb{2}s\sb{0}i\sb{r}]
      +[[s\sb{2}s\sb{0}i\sb{p},s\sb{2}s\sb{1}i\sb{q}],s\sb{1}s\sb{0}i\sb{r}]\\
      -&[[s\sb{2}s\sb{0}i\sb{p},s\sb{1}s\sb{0}i\sb{q}],s\sb{2}s\sb{1}i\sb{r}]
      +[[s\sb{2}s\sb{1}i\sb{p},s\sb{1}s\sb{0}i\sb{q}],s\sb{2}s\sb{0}i\sb{r}]
      -[[s\sb{2}s\sb{1}i\sb{p},s\sb{2}s\sb{0}i\sb{q}],s\sb{1}s\sb{0}i\sb{r}]
\end{split}
\end{myeq}
\noindent representing $\omega$ in filtration $3$ of the spectral sequence for
\w{\Wd=\ostar{\bX}{S\sp{1}}} (converging to \w[).]{\pis\Sigma\bX}
This could again be described in terms of threefold shuffles, though it does not
seem worth the trouble of figuring out the signs.
\end{example}
\end{mysubsection}

%
%
\sect{The algebra of higher homotopy operations}
\label{cahho}

We are now in a position to describe how the totality of higher homotopy operations
(based on spheres) are related to each other.

\begin{mysubsection}{A ``higher Hilton Theorem''}
\label{shht}
We think of Hilton's Theorem (\S \ref{swp}) as exhibiting any primary homotopy
operation as a sum of ``regular'' expressions, involving only Whitehead products,
postcomposed with ``sporadic'' elements from the homotopy groups of
individual spheres. Note that this is almost identical with the distinction between
the rational (infinite order) and torsion components.

The higher order analogue of this is not a formal theorem, but rather a procedure which
allows us to separate the ``regular'' component of the collection of
all higher homotopy operations (described in Section \ref{cwp}) from the ``sporadic''
contribution of maps between individual spheres (themselves appearing as values
of higher operations).

For this purpose, let \w{f:\bS{n+m}\to\bX} be a value of an \wwb{n+1}st order homotopy
operation, with initial data \w[:]{\lra{\Wd\to\bX,\hf:\bS{m}\to\bW\sb{n}}}
we may further assume that each \w{\bW\sb{i}} is a wedge of spheres, so
by Hilton's Theorem we may write \w{\hf=\omega\circ h} for
\w[,]{h:\bS{m}\to\bS{k}} with
\w{\omega:\bS{k}\to\bB=\bigvee\sb{j=1}\sp{p}\bS{n\sb{j}}} an
iterated Whitehead product of weight $N$ on the fundamental classes of wedge summands,
so \w[.]{k-1=\sum\sb{i=1}\sp{N}\ (n\sb{j\sb{i}}-1)}
We then have \w[\vsm .]{\bW\sb{n}=\bB\vee\bW\sb{n}'}

Note that for each \w[,]{\ell\geq 1} we can write \w[,]{\bB=\Sigma\sp{\ell}\bA} where
\w[.]{\bA=\bigvee\sb{j=1}\sp{p}\bS{n\sb{j}-\ell}}
Since \w[,]{\sum\sb{i=1}\sp{N}\ (n\sb{j\sb{i}}-\ell-1)=k-N\ell-1}
$\omega$ is represented in the spectral sequence for \w{\Yd=\ostar{\bA}{S\sp{\ell}}}
(with \w[),]{\|\Yd\|\simeq\bB} in filtration \w{r=N\ell} \wwh
that is, in the \w{\Ett{r}{k-r}} term \wh by a corresponding sum of
weight $N$ iterated Whitehead product $\homeg$ on \wwb{r-\ell}fold degeneracies of
fundamental classes of the desuspended wedge summands of \w[,]{\bY\sb{\ell}=\bA} as in
\cite[\S 7.6]{BBSenHS}. In other words, $\omega$ is the value of an
\wwb{r+1}st order operation with initial data
\w[.]{\lra{\ostar{\bA}{S\sp{\ell}}\to\bB,\homeg}}

Replacing \w{\Yd=\ostar{\bA}{S\sp{\ell}}} by \w{T\sp{n}\Yd}
(see \S \ref{snac}) allows us to form the spliced \wwb{r+n}truncated restricted
augmented simplicial space \w{T\sp{n}\Yd\propto\Wd} (which we abbreviate to
\w[),]{\Yd\propto\Wd} as in \S \ref{dsplice}, which we may use to produce a
``partial version'' of $f$, with \w{\hf=\omega\circ h} replaced by $\omega$ alone.

Now for each \w{\ell\geq 1} and \w{r=N\ell} we may similarly represent $h$ in
the spectral sequence for \w{\Zd=\ostar{\bS{k-r}}{S\sp{r}}} as a value of
some \wwb{p+1}st order operation with initial data
\w[,]{\lra{\Zd\to\bS{k},\hh:\bS{m-p}\to\bZ\sb{p}}} with \w[.]{p\geq r}

This describes $f$ as the value of a ``composite higher operation'' with initial data
$$
\lra{\Zd\propto\Yd\propto\Wd,\hh:\bS{m-p}\to\bZ\sb{p}}
$$
\noindent (in the notation of Definition \ref{dsplice}, with the appropriate
iterate of $T$ omitted), and total data \w[.]{H\propto W\propto F} This is described by
Figure \wref[,]{eqcwhpr} where we use our assumption on the resolution model category
to obtain a map of simplicial sets \w{\Zd\to\Yd} (with no need to actually choose the
coherences $W$):

\myqdiag[\label{eqcwhpr}]{
\bS{m-p} \ar@/^1pc/[r]^{0}_{\vdots}  \ar@/_1pc/[r]_{0} \ar[dd]\sb{\hh} &
\hsp 0\ar[dd]\dotsc \ar@<5ex>[dd]^{H}  & 0\ar[dd]\dotsc & 0\ar@<-1ex>[dd]\dotsc &
0\ar[dd]  \ar@/^2pc/[rr]^<<<<<<<{0}  \ar@/_2pc/[rr]_<<<<<{0} \ar[dd] \ar[r] &
(\bS{m}) \ar|(.4){\hole}[dd]_>>>>{(h)} \ar@/^1.2pc/|(.14){\hole}[dddddd]\sp{\hf} &
0 \ar[dddddd] & \cdots\ar[dddddd]\sb{F}  &
0 \ar[dddddd] \ar[r] & (\bS{n+m}) \ar[dddddd]\sp{(f)}\\
&&& & & & & &&&\\
\bZ\sb{p} \ar@/^1pc/[r]^{d_{0}}_{\vdots}\ar@/_1pc/[r]\sb{d\sb{\ell}}
& \hs\bZ\sb{p-1} \dotsc &
\bZ\sb{r}=\bS{k-r}\ar[dd]\sb{\homeg} &0 \dotsc \ar@<-1ex>[dd]\ar@<3ex>[dd]\sb{W}  &
0\ar[dd]\ar[r] &  (\bS{k}) \ar[dd]\sb{(\omega)}  &&\\
&&& & & & & &&&\\
&& \hs\bY\sb{r} \dotsc 
 &\hsl\bY\sb{\ell}=\bA\dotsc
& 0\ar[r] &(\bB) \ar@{^{(}->}[dd]\sb{\inc} & & & &&\\
&&& & & & & &&& \\
&&&&& \bW\sb{n} \ar@/^1pc/[r]^{0}_{\vdots}  \ar@/_1pc/[r]_{0} &\bW\sb{n-1} &
\dotsc &\bW\sb{0} \ar[r]\sp{\vare} & \bX
}

We think of the operation with initial data
\w{\lra{\Zd\to\bS{k},\homeg:\bS{k-r}\to\bY\sb{r}}} and value $\omega$ as being
essentially formal (in the same way that suspension is formally a secondary operation \wh
see \cite[\S 5.3]{BBSenHS}).

The more general situation, where $h$ is a \emph{sum} of compositions with iterated
Whitehead products reduces to the case discussed above by \S \ref{sadd}.
\end{mysubsection}

\begin{remark}\label{rhht}
Using the methods of \S \ref{swp} and Corollary \ref{cswedge} below, we can in fact
display the iterated Whitehead product $\omega$ of weight $N$ as the value of a ``formal''
higher operation of any chosen order \w[,]{r+1} not only one of the form
\w[.]{r=N\ell} This means that we can first express \w{h:\bS{m}\to\bS{k}} as the
value of some \wwb{p+1}st order operation in whatever way we want, and then choose
a compatible order for $\omega$.
\end{remark}

\begin{mysubsection}{Insertion}
\label{sins}
Theorem 7.16 of \cite{BBSenHS} allows the values of previously calculated higher
operations to be used anywhere in the process of representing homotopy classes
in terms of higher operations: this means that we must allow a value to be inserted
anywhere in a diagram such as \wref[.]{diaghho}
Thus we need to explain how to insert maps \w{\alpha:\bS{N}\to\bS{M}}
  between individual spheres which appear as values of a higher operation into a
  simplicial diagram \w[.]{\Wd}

By Remark \ref{rcwr}, we may assume that our (truncated restricted augmented) simplicial
space \w{\Wd\to\bX} has a CW basis \w[,]{(\oW{n})\sb{n=0}\sp{\infty}} and that
\w{(d\sb{j})\rest{\oW{n}}=0} for \w[.]{j\geq 1} Thus we need only consider the face
map \w[.]{d\sb{0}:\oW{n}\to\bW\sb{n-1}} Since \w{\oW{n}} is (weakly equivalent to)
a wedge of spheres, we may further restrict attention to a single wedge summand,
with fundamental class \w{\iot{m+r}:\bS{m+r}\hra\oW{n}\subseteq\bW\sb{n}}
(and \w[),]{\bW\sb{n}=\bS{m+r}\vee\bW'\sb{n}}
such that \w{d\sb{0}\circ\iot{m+r}=g:\bS{m+r}\to\bW\sb{n-1}} is a value
of an \wwb{r+1}-st order homotopy operation, with initial data
\w{\lra{\Yd\to\bW\sb{n-1},\hg:\bS{m}\to\bY\sb{r}}} (and each \w{\bY\sb{i}} a
wedge of spheres). One could further decompose $g$ as an iterated Whitehead product
precomposed with \w[,]{g':\bS{m}\to\bS{k}} as in \S \ref{shht}, but this is not necessary
for the rest of the argument.

Consider a map \w{\bA\to\bW\sb{n}} out of a wedge of spheres. This could be
either part of initial data \w{\lra{\Wd\to\bX,\hf:\bA\to\bW\sb{n}}} for a higher
homotopy operation with value \w[,]{f:\Sigma\sp{n}\bA\to\bX} or
a face map \w[\vsm.]{d\sb{i}:\bW\sb{n+1}\to\bW\sb{n}}

%
%
\noindent\textbf{I.} \ In the first case, we may assume that \w{\bA=\bS{p}} is a single
sphere, with \w{\hf=\omega\circ h} for \w{h:\bS{p}\to\bS{q}} and
\w{\omega:\bS{q}\to\bW\sb{n}} an iterated Whitehead product. We also assume that the
fundamental class \w{\iot{m+r}} is one of the arguments of $\omega$ (otherwise
no modification of the original data for $f$ is needed).

We can use \S \ref{swp} to represent $\omega$ too as the value of a
\wwb{r+1}-st order homotopy operation, with initial data
\w[.]{\lra{\ostar{\bS{q-r}}{S\sp{r}}\to\bS{q},~\homeg:\bS{q-r}\to\bS{m}\vee\bW\sb{n}'}}
Precomposing this with $h$, as described in \S \ref{shht}, yields the following
diagram:

\mytdiag[\label{eqinserthf}]{
  \bS{p-r-\ell} \ar[dddd]_<<<<<<{\hh}
  \ar@/^0.7pc/[r]^{0}\ar@/_0.7pc/[r]\sb{0}\sp{\vdots} & 0 & \dotsc \ar[dddd]\sp{H}
  &  0 \ar[r] & 0 \ar@/^0.7pc/[r]^{0}\ar@/_0.7pc/[r]\sb{0}\sp{\vdots} &
0 & \dotsc & 0\ar[r] & (\bS{p}) \ar[dddd]^<<<<<{(h)} & \\
&&& &&&  &&&\\
&&& &&&  &&&\\
&&& &&&  &&&\\
&&& &&&  &&&\\
\bZ\sb{\ell} \ar@/^0.7pc/[r]^{d_{0}}\ar@/_0.7pc/[r]\sb{d\sb{\ell}}\sp{\vdots} &
\bZ\sb{\ell-1} & \dotsc& \bZ\sb{0}\ar[r] &
(\bS{q-r}) \ar@/_1.7pc/[dddddd]\sb{\homeg}
\ar@/^0.7pc/[r]\sp{0}\ar@/_0.7pc/[r]\sb{0}\sp{\vdots} &
0 & \dotsc \ar[ddddd]\sp{W} & 0\ar[r] & (\bS{q}) \ar@/^2pc/[dddddd]\sp{(\omega)} & \\
&&& &&&  &&&\\
&&& &&&  &&&\\
&&& &&&  &&&\\
&&& &&&  &&&\\
&&&& \bS{m} \ar@/^0.7pc/[r]^{0}\ar@/_0.7pc/[r]\sb{0}\sp{\vdots}
\ar@/_1.5pc/[dddddd]\sb{\hg}&
0 && 0 \ar[r] & (\bS{m+r})\hs \ar@/^1pc/[ddddddddr]\sp{(g)} & \\
&&&& \vee & \vee & \dotsc \ar[ddddddd]\sp{G} & \vee & \vee & \\
&&&& \bW\sb{n}' \ar@/^0.5pc/[r]\sp{=}\ar@/_0.5pc/[r]\sb{=} \ar@/_1.5pc/[ddddddd]\sb{\Id}&
\bW\sb{n}' \ar[dddd] && \bW\sb{n}' \ar[dddd] \ar[r] & (\bW\sb{n}') & \\
&&& &&&  &&&\\
&&& &&&  &&&\\
&&& &&&  &&&\\
&&& &&&  &&&\\
&&&& \bY\sb{r} \ar@/^0.7pc/[r]^{d\sb{0}}\ar@/_0.7pc/[r]\sb{d\sb{r}}\sp{\vdots} &
\bY\sb{r-1} && \bY\sb{0}\ar[rrd]\sp{\vare}\\
&&&& \vee & \vee & \dotsc & \vee &  & (\bW\sb{n-1}) \\
&&&& \bW\sb{n}' \ar@/^0.5pc/[r]\sb{=}\ar@/_0.5pc/[r] &
\bW\sb{n}'\ar[r] & & \bW\sb{n}' \ar[r] & (\bW\sb{n}') \ar[ru]\sb{d\sb{0}} &
}

We then use
$$
\lra{\Ud:=(\Zd\propto([\ostar{\bS{m}}{S\sp{r}}\propto\Yd]\vee\co{\bW\sb{n}'})
  \propto\tau\sb{n-1}\Wd\to\bX,~\hh:\bS{p-r-\ell}\to\bZ\sb{\ell}}
$$
as initial data for the higher homotopy operation in question, as in \S \ref{shht}.

\begin{remark}\label{rdegcons}
Assuming \w[,]{\Wd\in\Sa\sp{\Dop}} we see that
\w{\co{\bW\sb{n}'})\propto\tau\sb{n-1}\Wd\in\Sa\sp{\Dres\op}} is in fact
well-defined, if we think of \w{c(\bW\sb{n}')\sb{k}\subset\bW\sb{k}} (for \w[)]{k>n}
as the image of \w{\bW\sb{n}'\subseteq\bW\sb{n}} under iterations of the
degeneracy \w[,]{s\sb{0}} because \w{d\sb{i}s\sb{0}=\Id} for \w{i=0,1}
and \w{d\sb{i}s\sb{0}=s\sb{0}d\sb{i-1}} for \w[.]{i\geq 2}
\end{remark}

%
%
\noindent\textbf{II.} \ In the second case, we may assume that our map is
\w{\odz{n+1}:\oW{n+1}\to\bW\sb{n}} (see \S \ref{dscwo} and Remark \ref{rcwr}), with
\w{\oW{n+1}\simeq\bigvee\sb{j\in J}\ \bS{p\sb{j}}} and
\w{f\sb{j}:=(\odz{\bW\sb{n+1}})\rest{\bS{p\sb{j}}}} as a sum of elements of the form
\w{\omega\sb{j}\circ h\sb{j}} for \w{h\sb{j}:\bS{p\sb{j}}\to\bS{q\sb{j}}} and
\w{\omega\sb{j}:\bS{q\sb{j}}\to\bW\sb{n}} a suitable iterated Whitehead product
(which we may assume to involve \w[,]{\iot{m+r}} as above).

Write \w{\bB:=\bigvee\sb{j\in J}\ \bS{p\sb{j}-r}} for the $r$-fold desupsension of
\w[,]{\oW{n+1}} and note that we can present each \w{h\sb{j}}
as the value of an \wwb{\ell\sb{j}+1}st order operation with initial data
\w[.]{\lra{\Zud{j}\to\bS{q\sb{j}-r},
    \hh\sb{j}:\bS{p\sb{j}-r-\ell\sb{j}}\to\bZ\up{j}\sb{\ell\sb{j}}}}
Thus \w{\Zd:=\bigvee\sb{j\in J}\ \Zud{j}} is augmented to $\bB$.

The case where different maps \w{f\sb{j}} involve the same Hall basis element
\w{\omega:\bS{q}\to\bW\sb{n}} are treated similarly.

We then obtain a slightly more elaborate form of \wref[,]{eqinserthf} namely,
\wref[:]{eqinsertface}

\mytdiag[\label{eqinsertface}]{
  \bS{p\sb{j}-r-\ell\sb{j}} \ar@/_1.5pc/[dddddd]_<<<<<<{\hh\sb{j}}
  \ar@/^0.7pc/[r]^{0}\ar@/_0.7pc/[r]\sb{0}\sp{\vdots} & 0 &  &  0 \ar[r] &
0 \ar@/^0.7pc/[r]^{0}\ar@/_0.7pc/[r]\sb{0}\sp{\vdots} &
0 & & 0\ar[r] & (\bS{p})
\ar@/^2pc/[ddddddd]^<<<<<{(h\sb{j})} & \\
 \vee & \vee & \dotsc \ar[ddddddd]\sp{H} & \vee & \vee & \vee & \dotsc &\vee & &\\
 \bW\sb{n}' \ar@/^0.5pc/[r]\sp{=}\ar@/_0.5pc/[r]\sb{=}
 \ar@/_1.8pc/[ddddddd]\sb{\Id} & \bW\sb{n}' \ar[ddddd] && \bW\sb{n}' \ar[r] &
 (\bW\sb{n}') \ar@/^0.5pc/[r]\sb{=}\ar@/_0.5pc/[r] &
\bW\sb{n}' & & \bW\sb{n}' \ar[r] & (\bW\sb{n}') \ar@/^2pc/[ddddddd]\sp{(\Id)} & \\
&&& &&&  &&&\\
&&& &&&  &&&\\
&&& &&&  &&&\\
&&& &&&  &&&\\
\uZ{j}\sb{\ell\sb{j}} \ar@/^0.7pc/[r]^{d_{0}}\ar@/_0.7pc/[r]\sb{d\sb{\ell}}\sp{\vdots} &
\uZ{j}\sb{\ell\sb{j}-1} & & \uZ{j}\sb{0}\ar[r] &
(\bS{q\sb{j}-r}) \ar@/_2.2pc/[dddddddd]^>>>>>>{\homeg\sb{j}}
\ar@/^0.7pc/[r]\sp{0}\ar@/_0.7pc/[r]\sb{0}\sp{\vdots} &
0 &  & 0\ar[r] & (\bS{q\sb{j}}) \ar@/^4pc/[dddddddd]\sp{(\omega\sb{j})} & \\
 \vee & \vee & \dotsc & \vee & \vee & \vee & \dotsc\ar[ddddddd]\sp{W} &\vee &\vee & \\
 \bW\sb{n}' \ar@/^0.5pc/[r]\sp{=}\ar@/_0.5pc/[r]\sb{=} &
 \bW\sb{n}'  && \bW\sb{n}'\ar|(.3){\hole}[r] &
 (\bW\sb{n}') \ar@/^0.5pc/[r]\sb{=}\ar@/_0.5pc/[r] \ar@/_2.2pc/[ddddddd]_>>>>>>{\Id}&
\bW\sb{n}'  & & \bW\sb{n}' \ar[r] & (\bW\sb{n}') \ar@/^3pc/[ddddddd]_<<<<<{(\Id)} & \\
&&& &&&  &&&\\
&&& &&&  &&&\\
&&& &&&  &&&\\
&&& &&&  &&&\\
&&&& \bS{m} \ar@/^0.7pc/[r]^{0}\ar@/_0.7pc/[r]\sb{0}\sp{\vdots}
\ar@/_1.5pc/[dddddd]\sb{\hg}&
0 && 0 \ar[r] & (\bS{m+r}) \ar@/^1pc/[ddddddddr]\sp{(g)} & \\
&&&& \vee & \vee & \dotsc \ar[ddddddd]\sp{G} & \vee & \vee & \\
&&&& \bW\sb{n}' \ar@/^0.5pc/[r]\sp{=}\ar@/_0.5pc/[r]\sb{=} \ar@/_1.5pc/[ddddddd]\sb{\Id}&
\bW\sb{n}' \ar[dddd] && \bW\sb{n}' \ar[dddd] \ar[r] & (\bW\sb{n}') & \\
&&& &&&  &&&\\
&&& &&&  &&&\\
&&& &&&  &&&\\
&&& &&&  &&&\\
&&&& \bY\sb{r} \ar@/^0.7pc/[r]^{d\sb{0}}\ar@/_0.7pc/[r]\sb{d\sb{r}}\sp{\vdots} &
\bY\sb{r-1} && \bY\sb{0}\ar[rrd]\sp{\vare}\\
&&&& \vee & \vee & \dotsc & \vee &  & (\bW\sb{n-1}) \\
&&&& \bW\sb{n}' \ar@/^0.5pc/[r]\sb{=}\ar@/_0.5pc/[r] &
\bW\sb{n}'\ar[r] & & \bW\sb{n}' \ar[r] & (\bW\sb{n}') \ar[ru]\sb{d\sb{0}} &
}

We want to replace the original \wwb{n+1}truncated \w{\tau\sb{n+1}\Wd\to\bX} by a new
augmented simplicial object
\begin{myeq}\label{eqlongertr}
  \Vd~:=~([\bigvee\sb{j\in J}\ostar{\bS{p\sb{j}-r-\ell\sb{j}}}{S\sp{r+\ell\sb{j}}}
    \propto\Zd\propto\Yd]\vee\co{\bW\sb{n}'})\propto\tau\sb{n-1}\Wd\to\bX~.
\end{myeq}
\end{mysubsection}

\begin{remark}\label{rtrunc}
On the face of it, \w{\Vd} is not truncated.  However, if $\bX$ is of finite type,
by \cite[\S 2]{BJTurnHH} we can choose our resolution \w{\Wd\to\bX} to be such, too,
which means that each \w{\bW\sb{n}} is a \emph{finite} wedge of spheres.  In this case,
\w{\Vd} of \wref{eqlongertr} is in fact finitely truncated.
\end{remark}

%
%
\sect{Higher order Whitehead products}
\label{chowp}
A well-known infinite sequence of non-trivial higher homotopy operations is
provided by the higher Whitehead products (see \cite{HardHW,AArkS}).
These may actually be defined for any collection of suspensions, as in \cite{GPorW},
and further generalized by the \emph{polyhedral products} of Buchstaber and Panov
(see \cite{BPanoTAT}), but for our purposes we shall make do with the following

\begin{defn}\label{duhwp}
Given a finite collection \w{\vS=(\bS{n\sb{i}})\sb{i=1}\sp{m}} of simply-connected
pointed spheres, for each \w{0\leq k<m} we define the \emph{$k$-th order (fat) wedge}
of $\vS$ to be
$$
T\sb{k}(\vS):=
\{(x\sb{1},\dotsc,x\sb{m})\in\prod\sb{i=1}\sp{m}\ \bS{n\sb{i}}|\
\exists 1\leq i\sb{1}<\dotsc<i\sb{k}\leq m,\  x\sb{i\sb{j}}=\ast\hsm
\text{for\ all\hsm} 1\leq j\leq k\}~.
$$
\noindent Thus \w{T\sb{m-1}(\vS)} is the wedge
\w[,]{\bigvee\sb{i=1}\sp{m}\ \bS{n\sb{i}}} \w{T\sb{1}(\vS)} is the usual fat wedge,
while \w{T\sb{0}(\vS)} is the product \w[.]{\prod\sb{i=1}\sp{m}\ \bS{n\sb{i}}}

We have inclusions \w[,]{T\sb{m-1}(\vS)\subset\dotsc T\sb{1}(\vS)\subset T\sb{0}(\vS)}
which provide a cell structure for \w[:]{T\sb{0}(\vS)}
if we let \w[,]{N=N\sb{\vS}:=\sum\sb{i=1}\sp{m}\ n\sb{i}} then \w{T\sb{1}(\vS)} is always
the \wwb{n-2}CW skeleton of the product (and
when \w[,]{n\sb{1}=\dotsc=n\sb{m}} all the \w{T\sb{k}(\vS)} are standard skeleta).
Thus we have attaching maps
\w[,]{\varphi\sp{k}\sb{\vS}:\bZ\sb{k}\to T\sb{k}(\vS)}
with \w{\bZ\sb{k}=\bZ\sb{k}(\vS)} a wedge of spheres, such that \w{T\sb{k-1}(\vS)}
is the homotopy cofiber of \w[.]{\varphi\sp{k}\sb{\vS}} The maps \w{\varphi\sp{k}\sb{\vS}}
\wb{1<k<m} are the \emph{universal higher order Whitehead product maps}. See
\cite{GPorHG} for further details.
\end{defn}

\begin{fact}\label{fuhwp}
For each \w[,]{m>k\geq 2} the space \w{\bZ\sb{k}(\vS)} is a wedge of spheres
\w[,]{\bS{N\sb{\vS'}-1}} as \w{\vS'} ranges over all sub-collections of
\w{(\bS{n\sb{i}})\sb{i=1}\sp{m}} with \w{m-k+1} elements, and
on each such summand the map \w{\varphi\sp{k}\sb{\vS}} is just
\w[.]{\varphi\sp{1}\sb{\vS'}:\bS{N\sb{\vS'}-1}\to T\sb{1}(\vS')}
\end{fact}

\begin{example}\label{eguhwp}
In particular, \w[,]{\bZ\sb{1}=\bS{N-1}} and
when \w[,]{m=2} \w{\varphi\sp{1}\sb{\vS}} is the ordinary Whitehead product
(see \cite[X, \S 7]{GWhE}).
\end{example}

\begin{defn}\label{dhwp}
Given a pointed space $\bX$ and a finite list of homotopy classes
\w[,]{\valph=(\alpha\sb{i}\in\pi\sb{n\sb{i}}\bX)\sb{i=1}\sp{m}} with each
\w[,]{n\sb{i}\geq 2} we may represent $\valph$ by a map \w[,]{a:T\sb{m-1}(\vS)\to\bX}
where \w[.]{\vS=(\bS{n\sb{i}})\sb{i=1}\sp{m}}
The \emph{$m$-th order Whitehead product} of $\valph$ (in the indexing of \cite{GPorHP})
is the (possibly empty) collection of homotopy classes of the form
\w[,]{[g\circ\varphi\sp{1}\sb{\vS}]\in\pi\sb{N-1}\bX} where
\w{N:=\sum\sb{i=1}\sp{m}\ n\sb{i}} and \w{g:T\sb{1}(\vS)\to\bX} runs over all possible
extension of $a$ to \w[.]{T\sb{1}(\vS)}
\end{defn}

\begin{defn}\label{dksign}
Given a fixed ordered set \w{(p\sb{i})\sb{i=1}\sp{n}} of $n$ integers with each
\w[,]{p\sb{i}\geq 1} the \emph{Koszul sign} \w{\gsn{\sigma}} of a permutation
\w{\sigma\in S\sb{n}} is defined recursively by setting \w[,]{\gsn{\Id}=1} and
requiring that each adjacent transposition \w{(i,j)} contributes a factor of
\w{(-1)\sp{p\sb{i}p\sb{j}+1}} (thus combining the graded and ungraded signs).
\end{defn}

\begin{remark}\label{rksign}
  Given \w{(p\sb{i})\sb{i=1}\sp{n}} as in \S \ref{dksign}, for any
  \wwb{k,l}shuffle
\w{(\sigma',\sigma'')} with \w{\sigma'=(i\sb{1}<\dotsc<i\sb{k})} and
\w[,]{\sigma''=(j\sb{1}<\dotsc<j\sb{\ell})} write
\w{\deg(\sigma'):=p\sb{i\sb{1}}+\dotsc +p\sb{i\sb{k}}}
and \w[.]{\deg(\sigma''):=p\sb{j\sb{1}}+\dotsc +p\sb{j\sb{\ell}}}
We then have:
\begin{myeq}\label{eqddegswitch}
  \gsn{\sigma',\sigma''}~=~(-1)\sp{\deg(\sigma')\deg(\sigma'')+k\ell}
  \cdot\gsn{\sigma'',\sigma'}~.
\end{myeq}

Now for any two disjoint ordered subsets
  \w{\sigma=(\sigma\sb{1},\dotsc,\sigma\sb{k})} and
  \w{\tau=(\tau\sb{1},\dotsc,\tau\sb{\ell})} of \w[,]{\{1,\dotsc, n\}}
denote the ordered subset obtained by rearranging
\w{\{\sigma\sb{1},\dotsc,\sigma\sb{k},\tau\sb{1},\dotsc,\tau\sb{\ell}\}}
in ascending order by \w[.]{\sigma\sqcup\tau} If \wb{\alpha,\beta,\gamma} are
pairwise disjoint sets of \w[,]{\{1,\dotsc, n\}} by considering the transpositions
needed to arrange the concatenation of the three in ascending order, we see that
\begin{myeq}\label{eqtsqcup}
\gsn{\alpha,\beta}\cdot\gsn{\alpha\sqcup\beta,\gamma}~=~\gsn{\alpha,\beta,\gamma}~=~
\gsn{\beta,\gamma}\cdot\gsn{\alpha,\beta\sqcup\gamma}~.
\end{myeq}
\noindent Thus if \w{(\sigma',\sigma'')} is a \wwb{k,l}shuffle and
\w{(\tau',\tau'')} is a further partition of \w{\sigma''} into two disjoint sets
(so \w[),]{\sigma''=\tau'\sqcup\tau''} then
\begin{myeq}\label{eqsqcup}
  \gsn{\sigma'\sqcup\tau',\tau''}~=~\gsn{\sigma',\tau'}\cdot\gsn{\sigma',\sigma''}
  \cdot\gsn{\tau',\tau''}~.
\end{myeq}
\end{remark}

\begin{notn}\label{ngshuffle}
Following \cite[\S 8.11]{BBSenHS}, we extend the notation of
\S \ref{nshuffle} as follows:
\begin{enumerate}
\renewcommand{\labelenumi}{(\roman{enumi})~}
\item If $I$ and $J$ are disjoint subsets of \w[,]{\{0,1,2,\dotsc,N\}} each written in
  ascending order let
\w{\psi:I\sqcup J\to\{1,2,\dotsc,n\}} be an order-preserving
isomorphism, with \w[.]{\sgn{I,J}:=\sgn{\psi[I],\psi[J]}}
\item More generally, if $I$ and $J$ are any two finite sets of non-negative integers,
let \w{I':=I\setminus J} and \w[,]{J':=J\setminus I} and set
\w[.]{\sgn{I,J}:=\sgn{I',J'}}
\end{enumerate}
\end{notn}

\begin{mysubsection}{Properties of $\sgn{I,J}$}
\label{sgshuffle}
We recall from \cite[\S 8]{BBSenHS} a number of basic properties of the sign
of a collection of subsets of \w[:]{\{0,1,2,\dotsc,N\}}

\begin{enumerate}
\renewcommand{\labelenumi}{(\arabic{enumi})~}
\item If \w{(I',J')} is obtained from \w{(I,J)} by switching a
pair of indices between $I$ and $J$, then \w[.]{\sgn{I',J'}=-\sgn{I,J}}
\item If \w{I\setminus J} has cardinality $k$ and \w{J\setminus I}
has cardinality $\ell$, then
\begin{myeq}\label{eqsgncompl}
\sgn{I,J}~=~(-1)\sp{k\cdot\ell}\sgn{J,I}
\end{myeq}
\item For subsets $K$ and $L$ of \w{\{0,\dotsc,n\}} with \w{|K|=k} and
  \w[,]{|L|=\ell} let $\hK$ and $\hL$ denote the subsets of \w{\{1,\dotsc,n+1\}}
  obtained by adding $1$ to each element of $K$ (respectively, $L$), and
  \w[.]{\hhK:=\hK\cup\{0\}} Then
\begin{myeq}\label{eqkllk}
\sgn{\hhK,\hL}~=~\sgn{K,L}~,
\end{myeq}
\noindent and so by \wref{eqsgncompl} and \wref[:]{eqkllk}
\begin{myeq}\label{eqhlhk}
\begin{split}
  \sgn{\hL,\hhK}~=&~(-1)\sp{\ell(k+1)}\sgn{\hhK,\hL}~=~
  (-1)\sp{\ell(k+1)}\sgn{K,L}\\
  =&~(-1)\sp{\ell(k+1)}(-1)\sp{\ell\cdot k}\sgn{L,K}~=~(-1)\sp{\ell}\sgn{L,K}~.
\end{split}
\end{myeq}
\item If \w{M=(m\sb{1},\dotsc,m\sb{k})} is an ordered set of natural numbers (in
  ascending order), denote its underlying set by $\uM$, and conversely. In particular,
  if $\uM$ and $\uN$ are disjoint sets of natural numbers, \w{M\sqcup N=N\sqcup M}
  will denote the disjoint union \w[,]{\uM\sqcup\uN} arranged in ascending order.

  For any decomposition \w[,]{\uM\sqcup\uN\sqcup\uP=\{0,\dotsc,n-1\}} consider the
  corresponding three-fold shuffle \w[.]{(M,N,P)} As in \wref[,]{eqtsqcup} we have
\begin{myeq}\label{eqsgndisunl}
\sgn{M,N\sqcup P}\cdot\sgn{N,P}~=~\sgn{M,N,P}~=~\sgn{M\sqcup N,P}\cdot\sgn{M,N}~.
\end{myeq}
\end{enumerate}
\end{mysubsection}

\begin{defn}\label{drhwp}
Assume given an ordered collection \w{\vS=(\bS{p\sb{i}+1})\sb{i=1}\sp{n}} of
simply-connected spheres, we define a rational simplicial space \w{\Wd=\Wd(\vS)}
by inductively choosing a CW-basis \w[,]{(\oW{m})\sb{m=0}\sp{\infty}} as follows:

For each \wwb{m,\ell} shuffle
\w{(\tau',\tau'')=(\tau'\sb{1},\dotsc,\tau'\sb{m},\tau''\sb{1},\dotsc,\tau''\sb{\ell})}
with \w{m+\ell=n} (or equivalently, for each subset
\begin{myeq}\label{eqvsp}
\vS'~=~(\bS{p\sb{\tau'\sb{1}}+1},\dotsc,\bS{p\sb{\tau'\sb{m}}+1})
\end{myeq}
\noindent of $\vS$ of cardinality \w[),]{1\leq m\leq n}  a rational sphere
\w[,]{\bS{N\sb{\vS'}}\lo{\tau'}} where \w[.]{N\sb{\vS'}:=\deg(\tau')+1}
We denote its fundamental class by
\w[,]{\iol{\tau'}\in\pi\sb{N\sb{\vS'}}\bS{N\sb{\vS'}}\lo{\tau'}}
and set \w[.]{\oW{m-1}:=\bigvee\sb{(\tau',\tau'')\in\cI{n}{m-1}}~
  \bS{N\sb{\vS'}}\lo{\tau'}}

The attaching map \w{\odz{W\sb{m-1}}:\oW{m-1}\to\bW\sb{m-2}} sends \w{\iol{\tau'}} to
\begin{myeq}\label{eqattmap}
  \phi\sb{\vS'}:=
  \sum\sb{k=1}\sp{\lfloor\frac{m}{2}\rfloor}
  \sum\sb{(\sigma',\sigma'')\in\wI{m}{m-k}}\ (-1)\sp{\deg(\sigma')+k}
  \gsn{\sigma',\sigma''}
  \sum\sb{(I,J)\in\hI{m-2}{k-1}}\ \sgn{I,J}
      [s\sb{I}\iol{\sigma'},s\sb{J}\iol{\sigma''}]
\end{myeq}
\noindent in \w[,]{\pi\sb{N\sb{\vS'}}\bW\sb{m-2}} for \w{\wI{m}{m-k}} and
\w{\hI{m-2}{k-1}} as in \S \ref{nshuffle}. We refer to \w{(-1)\sp{\deg(\sigma')+k}}
as the \emph{global sign} for the summand
\w[.]{[s\sb{I}\iol{\sigma'},s\sb{J}\iol{\sigma''}]}
\end{defn}

\begin{example}\label{eghwp}
For \w[,]{\vS=(\bS{p+1},\bS{q+1},\bS{r+1})} we have
\w{\bW\sb{0}=\bS{p+1}\vee\bS{q+1}\vee\bS{r+1}} and
\w[,]{\oW{1}=\bS{p+q+1}\lo{p,q}\vee\bS{p+r+1}\lo{p,r}\vee\bS{q+r+1}\lo{q,r}}
with attaching maps
\w{d\sb{0}(\iota\lo{i,j})=(-1)\sp{i+1}[\iot{i},\iot{j}]} on
\w[,]{\bS{i+j+1}\lo{i,j}} for \w[,]{\{i,j\}\subset\{p,q,r\}}
where \w{\iot{p}} is the fundamental class for \w[,]{\bS{p+1}} and so on.
For \w{\vS'=\vS} and \w{m=n=3} we then have:
\begin{myeq}\label{eqtriplewp}
\begin{split}
&\phi\lo{\vS}~=~(-1)\sp{p+q+1}[\iota\lo{p,q},s\sb{0}\iot{r}]
~+(-1)\sp{p+r+1}\cdot(-1)\sp{qr+1}\cdot[\iota\lo{p,r},s\sb{0}\iot{q}]\\
  &\hspace*{10mm}+(-1)\sp{q+r+1}\cdot(-1)\sp{p(q+r)}\cdot[\iota\lo{q,r},s\sb{0}\iot{p}]
~\in~\pi\sb{p+q+r}\bW\sb{1}~.
%
\end{split}
\end{myeq}
\end{example}

\begin{prop}\label{phwp}
Given a collection \w{\vS=(\bS{p\sb{i}+1})\sb{i=1}\sp{n}} of simply-connected
spheres, the rational simplicial space \w{\Wd=\Wd(\vS)} of \S \ref{drhwp} is
a resolution of the rational fat wedge \w[,]{T\sb{1}(\vS)} and the class
\w{[\phi\sb{\vS}]} of \wref{eqattmap} (for \w{\vS'=\vS} and \w[)]{m=n} represents
the rational $n$-th order Whitehead product for $\vS$ in the homotopy spectral sequence
for \w[.]{\Wd}
\end{prop}

\begin{proof}
Since there are no non-degenerate elements in \w{E\sp{1}\sb{\ast,n-1}} for
\w[,]{\ast\leq N\sb{\vS}} if we show that \w{[\phi\sb{\vS}]\in E\sp{1}\sb{N\sb{\vS},n-2}}
is a Moore cycle, it cannot bound, so we may deduce from Proposition \ref{pnoratho}
that it represents the rational $n$-th order Whitehead product.

We do so by induction on \w[,]{n\geq 2} starting with
\w{\oW{0}=T\sb{n-1}(\vS)=\bigvee\sb{i=1}\sp{n}\ \bS{p\sb{i}+1}}
and \w[,]{\oW{1}=\bigvee\sb{i<j}\ \bS{p\sb{i}+p\sb{j}+1}\lo{i,j}} where
\w{\odz{W\sb{1}}} is given on \w{\bS{p\sb{i}+p\sb{j}+1}\lo{i,j}} by
\w[\vsm.]{[\iot{p\sb{i}+1},\iot{p\sb{j}+1}]}

\noindent\textbf{Step I.}\ First observe that \w{\phi\sb{\vS}} is a Moore
chain: for this purpose, assume that
\begin{myeq}\label{eqaa}
  \hA~=~[s\sb{I}\iol{\sigma'},s\sb{J}\iol{\sigma''}]~
  \in\pi\sb{N\sb{\vS}}\bW\sb{n-2}
\end{myeq}
\noindent is one summand, which appears in \wref{eqattmap} with coefficient
\begin{myeq}\label{eqsgnha}
  (-1)\sp{\deg(\sigma')+k}\gsn{\sigma',\sigma''}\cdot\sgn{I,J}
\end{myeq}
\noindent (this is just a sign \w[,]{\pm 1} of course). Here
\begin{myeq}\label{eqssij}
\begin{split}
  &(\sigma',\sigma'')~\text{is an} \ (n-k,k)\text{-shuffle, with}\ \sigma'\sb{1}=1 \
  \text{if} \ n=2k\\
&(I,J) \ \text{is a} \ (k-1,n-k-1)\text{-partition of}~\{0,\dotsc,n-3\}~.
\end{split}
\end{myeq}

We then must show that \w{d\sb{\ell}(\hA)=0} for each \w[;]{\ell\geq 1} to do so,
consider two cases:

\begin{enumerate}
\renewcommand{\labelenumi}{(\alph{enumi})~}
\item If \w{\ell\in I} and \w[,]{\ell-1\in J} or conversely, let
  \w{(\hat{I},\hat{J})} be obtained from \w{(I,J)} by switching $\ell$ and
  \w[.]{\ell-1} Then for any
\w[,]{(\sigma',\sigma'')\in\cI{n}{n-k}} we have
\w{d\sb{\ell}[s\sb{I}\iol{\sigma'},s\sb{J}\iol{\sigma''}]=
  d\sb{\ell}[s\sb{\hat{I}}\iol{\sigma'},s\sb{\hat{J}}\iol{\sigma''}]}
but \w[,]{\sgn{I,J}=-\sgn{\hat{I},\hat{J}}} so the two expressions
cancel in the sum for \w[.]{d\sb{\ell}\circ\phi\sb{\vS}}
\item If \w[,]{\ell, \ell-1\in J} applying \w{d\sb{\ell}} to
  \w[,]{s\sb{I}=s\sb{i\sb{\ell-1}}\dotsc s\sb{i\sb{t}}\dotsc s\sb{i\sb{1}}}
  where \w{\ell<i\sb{t}} and \w[,]{\ell>i\sb{t-1}-1} yields
$$
  s\sb{i\sb{\ell-1}-1}\dotsc s\sb{i\sb{t}-1}d\sb{\ell}s\sb{i\sb{t-1}}\dotsc
  s\sb{i\sb{1}}
~=~s\sb{i\sb{\ell-1}-1}\dotsc s\sb{i\sb{t}-1}s\sb{i\sb{t-1}}\dotsc s\sb{i\sb{1}}
d\sb{\ell-t+1}~.
$$
Since \w[,]{\ell>t-1} necessarily \w{d\sb{\ell-t+1}} is not \w[,]{d\sb{0}} so
by \wref{eqotherzero} \w[.]{d\sb{\ell-t+1}\iol{\sigma''}=0}

Similarly if \w[\vsm.]{\ell, \ell-1\in I}
\end{enumerate}

To show that \w{\phi\sb{\vS}} is in fact a Moore cycle,
note that for $\hA$ as in \wref[:]{eqaa}
\begin{enumerate}
\renewcommand{\labelenumi}{(\alph{enumi})~}
\item if \w[,]{0\in I}
  \w[,]{d\sb{0}(\hA)=
    [s\sb{I'}\iol{\sigma'},s\sb{J'}d\sb{0}\iol{\sigma''}]} and
\item if \w[,]{0\in J}
  \w[,]{d\sb{0}(\hA)=
    [s\sb{I'}d\sb{0}\iol{\sigma'},s\sb{J'}\iol{\sigma''}]}
\end{enumerate}
\noindent where \w{I'} and \w{J'} are obtained from $I$ and $J$, respectively,
by decreasing all indices by $1$ (and omitting $0$).

By Definition \ref{drhwp} we see that \w{d\sb{0}\iol{\sigma'}} and
\w{d\sb{0}\iol{\sigma''}} are themselves sums of the form
\wref[,]{eqattmap} with $\vS$ replaced by \w{\vS'} or \w{\vS''} (the subset of spheres
appearing in \w{\sigma'} or \w[,]{\sigma''} respectively)\vsm. Therefore, the terms which occur in $d\sb{0} (\phi\sb{\vS})$ are of the form
$$
A~=~[[s\sb{P} \iol{\sigma}, s\sb{Q} \iol{\tau}], s\sb{R} \iol{\rho}]~,
$$
with a sign \w[,]{\pm 1} which we proceed to calculate below, as its coefficient.

We may assume that the number of spheres in $\sigma$ is greater than that in $\tau$.
and if they are equal then \w[.]{1\in \sigma}
We use the following notation
$$
L:= P \cap Q,\hsp M:= P \setminus L,\hsp N= Q \setminus L~,
$$
so that \w[.]{R=M\sqcup N} We also write \w{k\sb{\sigma}} (respectively, \w{k\sb{\tau}} and
\w[)]{k\sb{\rho}} for the cardinality of $\sigma$ (respectively $\tau$ and $\rho$).
It follows that
$$
\#(L)=k\sb{\rho}-1,~\#(M)=k\sb{\tau}-1,~\#(N)=k\sb{\sigma}-1,~\#(R)=n- k\sb{\rho}-2~.
$$

For example, when \w{n=4} and

\begin{myeq}\label{eqquadruplewp}
\begin{split}
%
\phi\sb{\vS}~=~&[\iota\lo{p,q,r},s\sb{1}s\sb{0}\iota\sb{s}]
+(-1)\sp{rs+1}[\iota\lo{p,q,s},s\sb{1}s\sb{0}\iota\sb{r}]
+(-1)\sp{q(r+s)}[\iota\lo{p,r,s},s\sb{1}s\sb{0}\iota\sb{q}]\\
&+(-1)\sp{(p+q+r)s+1}[s\sb{1}s\sb{0}\iota\sb{p},\iota\lo{q,r,s}]
%
+\left([s\sb{0}\iota\lo{p,q},\,s\sb{1}\iota\lo{r,s}]
-[s\sb{1}\iota\lo{p,q},\,s\sb{0}\iota\lo{r,s}]\right)\\
&+(-1)\sp{qr+1}\left([s\sb{0}\iota\lo{p,r},\,s\sb{1}\iota\lo{q,s}]
-[s\sb{1}\iota\lo{p,r},\,s\sb{0}\iota\lo{q,s}]\right)\\
&+(-1)\sp{(q+r)s}\left([s\sb{0}\iota\lo{p,s},\,s\sb{1}\iota\lo{q,r}]
-[s\sb{1}\iota\lo{p,s},\,s\sb{0}\iota\lo{q,r}]\right)~,
\end{split}
\end{myeq}
\noindent after re-arranging the terms we have
\begin{equation*}
\begin{split}
%
d\sb{0}(\phi\lo{\vS})=&\underline{[s\sb{0}\iota\sb{p},[s\sb{0}\iota\sb{q},\iota\lo{r,s}]]}
+(-1)\sp{qr+1}\underset{\dotsc\dotsc\dotsc\dotsc\dotsc\dotsc\dotsc\dotsc\dotsc\dotsc\dotsc}
{[s\sb{0}\iota\sb{p},[s\sb{0}\iota\sb{r},\iota\lo{q,s}]]}\\
+&(-1)\sp{(q+r)s}\underset{------------}
{[s\sb{0}\iota\sb{p},[s\sb{0}\iota\sb{s},\iota\lo{q,r}]]}
+(-1)\sp{pq+1}\underline{[s\sb{0}\iota\sb{q},[s\sb{0}\iota\sb{p},\iota\lo{r,s}]]}\\
+&(-1)\sp{p(q+r)}\underset{\sim\sim\sim\sim\sim\sim\sim\sim\sim\sim\sim}
{[s\sb{0}\iota\sb{q},[s\sb{0}\iota\sb{r},\iota\lo{p,s}]]}
+(-1)\sp{pq+(p+r)s+1}\underset{===========}
{[s\sb{0}\iota\sb{q},[s\sb{0}\iota\sb{s},\iota\lo{p,r}]]}\\
+&(-1)\sp{(p+q)r}
\underset{\dotsc\dotsc\dotsc\dotsc\dotsc\dotsc\dotsc\dotsc\dotsc\dotsc\dotsc}
{[s\sb{0}\iota\sb{r},[s\sb{0}\iota\sb{p},\iota\lo{q,s}]]}
+(-1)\sp{pq+(p+q)r+1}\underset{\sim\sim\sim\sim\sim\sim\sim\sim\sim\sim\sim}
{[s\sb{0}\iota\sb{r},[s\sb{0}\iota\sb{q},\iota\lo{p,s}]]}\\
+&(-1)\sp{(p+q)(r+s)}\underset
{\circ\circ\circ\circ\circ\circ\circ\circ\circ\circ\circ\circ\circ\circ\circ\circ\circ\circ}
{[s\sb{0}\iota\sb{r},[s\sb{0}\iota\sb{s},\iota\lo{p,q}]]}
+(-1)\sp{(p+q+r)s+1}\underset{------------}
{[s\sb{0}\iota\sb{s},[s\sb{0}\iota\sb{p},\iota\lo{q,r}]]}\\
+&(-1)\sp{pq+(p+q+r)s}\underset{===========}
{[s\sb{0}\iota\sb{s},[s\sb{0}\iota\sb{q},\iota\lo{p,r}]]}
+(-1)\sp{(p+q)(r+s)+rs+1}\underset
{\circ\circ\circ\circ\circ\circ\circ\circ\circ\circ\circ\circ\circ\circ\circ\circ\circ\circ}
{[s\sb{0}\iota\sb{s},[s\sb{0}\iota\sb{r},\iota\lo{p,q}]]}\\
%
+&
\underset
{\circ\circ\circ\circ\circ\circ\circ\circ\circ\circ\circ\circ\circ\circ\circ\circ\circ\circ}
{[\iota\lo{p,q},[s\sb{0}\iota\sb{r},s\sb{0}\iota\sb{s}]]}
+(-1)\sp{(p+q)(r+s)}\underline{[\iota\lo{r,s},[s\sb{0}\iota\sb{p},s\sb{0}\iota\sb{q}]]}\\
%
+&(-1)\sp{qr+1}\underset{===========}
{[\iota\lo{p,r},[s\sb{0}\iota\sb{q},s\sb{0}\iota\sb{s}]]}
+(-1)\sp{p(q+s)+rs+1}
\underset{\dotsc\dotsc\dotsc\dotsc\dotsc\dotsc\dotsc\dotsc\dotsc\dotsc\dotsc}
{[\iota\lo{q,s},[s\sb{0}\iota\sb{p},s\sb{0}\iota\sb{r}]]}\\
%
+&(-1)\sp{(q+r)s}\underset{\sim\sim\sim\sim\sim\sim\sim\sim\sim\sim\sim}
{[\iota\lo{p,s},[s\sb{0}\iota\sb{q},s\sb{0}\iota\sb{r}]]}
+(-1)\sp{p(q+r)}\underset{------------}
{[\iota\lo{q,r},[s\sb{0}\iota\sb{p},s\sb{0}\iota\sb{s}]]}~,
\end{split}
\end{equation*}
\noindent which vanishes by applying the Jacobi identity in six groups of three,
as marked\vsm.

\noindent\textbf{Step II.}\ We assume first that \w[,]{k\sb{\sigma} + k\sb{\tau}\geq n/2}
and if equality holds, then \w[.]{1 \in \sigma \sqcup \tau}
This implies that the term $A$ appears in $d\sb{0}(\hA)$ where $\hA$ is given in
\wref{eqaa} with
$$
\sigma' = \sigma \sqcup \tau,~\sigma''=\rho,~J=\hhR,~I=\hL.
$$
The sign of $\hA$ in \w{\phi\sb{\vS}} is
$$
(-1)\sp{k\sb{\rho}+\deg(\sigma)+\deg(\tau)} \gsn{\sigma \sqcup \tau, \rho} \sgn{\hL, \hhR} =
- (-1)\sp{\deg(\sigma) + \deg(\tau)} \gsn{\sigma \sqcup \tau, \rho} \sgn{L,R}.
$$
The sign of $A$ in \w{d\sb{0}(\hA)} is
$$
(-1)\sp{k\sb{\tau}} (-1)\sp{\deg(\sigma)} \gsn{\sigma,\tau} \sgn{M,N}.
$$
Thus the sign of $A$ in $d\sb{0}(\phi\sb{\vS})$ is
$$
-(-1)\sp{k\sb{\tau}+\deg(\tau)}\gsn{\sigma,\tau,\rho}\sgn{L,M,N}~.
$$

\noindent\textbf{Step III.}\ Now assume that for the term $A$,
\w{k\sb{\sigma} + k\sb{\tau}\leq n/2} and in case of equality,
\w[.]{1 \notin \sigma \sqcup \tau}
The term appears after a reordering via \wref{eqanticomm} in \w[,]{d\sb{0}(A')} where
$$
A'= [s_{\hhR} \iol{\rho}, s\sb{\hL} \iota_{\sigma \sqcup \tau}].
$$
The sign of \w{A'} in \w{d\sb{0}(\phi\sb{\vS})} is
$$
(-1)\sp{k\sb{\sigma} + k\sb{\tau}}(-1)\sp{\deg(\rho)}\gsn{\rho,\sigma \sqcup \tau}\sgn{\hhR,\hL}~=~
  (-1)\sp{k\sb{\sigma} + k\sb{\tau}}(-1)\sp{\deg(\rho)}\gsn{\rho,\sigma \sqcup \tau}\sgn{R,L}~.
$$
The sign of \w{[s\sb{R} \iol{\rho}, [s\sb{P}\iol{\sigma}, s\sb{Q} \iol{\tau}]]} in \w{d\sb{0}(A')} is
$$
(-1)\sp{k\sb{\tau}}(-1)\sp{\deg(\sigma)}\gsn{\sigma,\tau}\sgn{M,N}.
$$
The skew symmetry involved in obtaining $A$ has sign
$$
(-1)\sp{(\deg(\rho)+1)(\deg(\sigma)+\deg(\tau)+1)}.
$$
Therefore, the total sign of $A$ in this case is
\begin{equation*}
\begin{split}
  (-1)\sp{k\sb{\sigma} + k\sb{\tau}+\deg(\rho)}&\gsn{\rho,\sigma \sqcup \tau}
  \sgn{R,L}(-1)\sp{k\sb{\tau}+\deg(\sigma)}\\
  &\hs\cdot\gsn{\sigma,\tau}\sgn{M,N}(-1)\sp{(\deg(\rho)+1)(\deg(\sigma)+\deg(\tau)+1)} \\
  =~&-(-1)\sp{k\sb{\sigma} + \deg(\tau)+\deg(\rho)(\deg(\sigma)+\deg(\tau))}
  \gsn{\rho,\sigma,\tau}\sgn{M,N,L}.
\end{split}
\end{equation*}
\noindent\textbf{Step IV.} We assume that
\w[,]{k\sb{\sigma}\geq k\sb{\tau}\geq k\sb{\rho}}
and if equality occurs, $1$ belongs to the subset on the left. Observe that the terms
$$
A~:=~[[s\sb{P}\iol{\sigma}, s\sb{Q} \iol{\tau}], s\sb{R} \iol{\rho}],\hs
B~:=~[[s\sb{Q}\iol{\tau}, s\sb{R} \iol{\rho}], s\sb{P} \iol{\sigma}],
\hs\text{and}\hs C~:=~[[s\sb{P}\iol{\sigma}, s\sb{R} \iol{\rho}], s\sb{Q} \iol{\tau}]~,
$$
together with their signs in \w[,]{d\sb{0}(\phi\sb{\vS})} cancel out via the
Jacobi identity.

There are two different cases to consider:
\begin{enumerate}
\renewcommand{\labelenumi}{(\arabic{enumi})~}
\item \w[,]{k\sb{\sigma} \leq n/2} and in case of equality, we have
  \w[.]{1\notin \sigma}
\item \w[,]{k\sb{\sigma} \geq n/2} and in case of equality, we have \w[.]{1\in \sigma}
\end{enumerate}

In case (1), all the signs are as computed in Step II. Thus,
\begin{equation*}
\begin{split}
\frac{\sgn{B}}{\sgn{A}}~& =~\frac{-(-1)\sp{k\sb{\rho}+\deg(\rho)}\gsn{\tau,\rho,\sigma}
    \sgn{N,L,M} }{-(-1)\sp{k\sb{\tau}+\deg(\tau)}\gsn{\sigma,\tau,\rho}\sgn{L,M,N} } \\
&=~\frac{(-1)\sp{k\sb{\rho} + k\sb{\tau}+\deg(\rho)+\deg(\tau)}\gsn{\tau\sqcup\rho,\sigma}
    \sgn{N,L\sqcup M} }{\gsn{\sigma,\tau\sqcup \rho}\sgn{L\sqcup M,N} }\\
  &=~(-1)\sp{k\sb{\rho} + k\sb{\tau} +\deg(\rho)+\deg(\tau)}
(-1)\sp{(\deg(\tau)+\deg(\rho))(\deg(\sigma))+(k\sb{\tau}+k\sb{\rho})k\sb{\sigma}}
(-1)\sp{(k\sb{\rho}+k\sb{\tau})(k\sb{\sigma} -1)} \\
       &=~(-1)\sp{(\deg(\tau)+\deg(\rho))(\deg(\sigma)+1)}~.
\end{split}
\end{equation*}
\noindent Analogously,
\begin{equation*}
\begin{split}
  \frac{\sgn{C}}{\sgn{A}}~&=~\frac{-(-1)\sp{k\sb{\rho}+\deg(\rho)}
    \gsn{\sigma,\rho,\tau}\sgn{M,L,N}}
       {-(-1)\sp{k\sb{\tau}+\deg(\tau)}\gsn{\sigma,\tau,\rho}\sgn{L,M,N} } \\
 &=~\frac{(-1)\sp{k\sb{\rho} + k\sb{\tau} +\deg(\rho)+\deg(\tau)}\gsn{\rho,\tau}\sgn{M,L}}
       {\gsn{\tau, \rho}\sgn{L, M} }\\
       &=~(-1)\sp{k\sb{\rho} + k\sb{\tau} +\deg(\rho)+\deg(\tau)}
       (-1)\sp{\deg(\tau)\deg(\rho)+k\sb{\tau} k\sb{\rho}}
       (-1)\sp{(k\sb{\rho}-1)(k\sb{\tau} -1)} \\
       &=~(-1)\sp{(\deg(\tau)+1)(\deg(\rho)+1)}~.
\end{split}
\end{equation*}
\noindent It follows that the three terms cancel out by the Jacobi identity.

In the case (2), the signs of $A$ and $C$ are calculated using Step II,
while the sign of $B$ is calculated by Step III. This turns out to be
$$
-(-1)\sp{k\sb{\tau} + \deg(\rho)+\deg(\sigma)(\deg(\tau)+\deg(\rho))}
\gsn{\sigma,\tau,\rho}\sgn{L,M,N}~.
$$
Therefore, once again we have
\begin{equation*}
\begin{split}
  \frac{\sgn{B}}{\sgn{A}} &
  = \frac{-(-1)\sp{k\sb{\tau}+\deg(\rho)+\deg(\sigma)(\deg(\tau)+\deg(\rho))}
 \gsn{\sigma,\tau,\rho}\sgn{L,M,N}}{-(-1)\sp{k\sb{\tau}+\deg(\tau)}\gsn{\sigma,\tau,\rho}
    \sgn{L,M,N}}\\
  &= (-1)\sp{(\deg(\tau)+\deg(\rho))(\deg(\sigma)+1)}~.
\end{split}
\end{equation*}
\noindent This concludes the proof.
\end{proof}

\begin{mysubsection}{An integral version}
\label{sivhwp}
By Corollary \ref{cnoratho}, we can in theory extract from the above an integral
resolution of the fat wedge \w[,]{T\sb{1}(\vS)} in which some multiple of the class
\w{[\phi\sb{\vS}]} represents the universal $n$-th order Whitehead product for $\vS$.
Note, however, that this is an inductive process: for \w[,]{n=3} we can use
\w{[\phi\sb{\vS}]} itself, since there is no room for a non-zero
\ww{d\sb{2}}-differential on it.

For \w[,]{n=4} and \w[,]{\vS=(\bS{p\sb{i}+1})\sb{i=1}\sp{4}} we have
\w[,]{\bW\sb{0}=\bigvee\sb{1\leq i\leq 4}~\bS{p\sb{i}+1}}
\w{\oW{1}=\bigvee\sb{1\leq i<j\leq 4}~\bS{p\sb{i}+p\sb{j}+1}\lo{p\sb{i},p\sb{j}}}
with attaching maps \w{d\sb{0}(\iota\lo{i,j})=\pm[\iot{i},\iot{j}]}
as in Example \ref{eghwp}, and
\w{\oW{2}=
\bigvee\sb{1\leq i<j<k\leq 4}~\bS{p\sb{i}+p\sb{j}+p\sb{k}+1}\lo{p\sb{i},p\sb{j},p\sb{k}}}
with attaching maps given by the appropriate triple Whitehead products, we know that
the fourth-order Whitehead product is represented by $K$ times the class
\w{[\phi\lo{\vS}]\in\pi\sb{N\sb{\vS}}\bW\sb{2}} of \wref[.]{eqattmap} However, we have
\w{K\neq \pm 1} only if the \ww{d\sb{2}}-differential in the homotopy spectral sequence
for \w{\Wd} hits some non-trivial class in \w{\Ett{N\sb{\vS}+1}{0}}
represented by \w[.]{\alpha\in\pi\sb{N\sb{\vS}+1}\bW\sb{0}}

Once we have determined the multiple for
\w[,]{n=4} we must then modify the description in \S \ref{drhwp} accordingly, then
find the appropriate multiple for \w[,]{n=5} and so on.
\end{mysubsection}

%
%
\sect{Lie-Massey products}
\label{clmp}
A more general collection of rational higher order operations is provided by the Lie
analogue of higher Massey products. These can be described in terms of the \w{\Li}
model for rational homotopy theory (see \cite{LStaS,LMarkS}), but we shall use Retakh's
original description for a DG Lie algebra \w{(L\sb{\ast},d)} in \cite{RetaL}
(here with the usual \w[).]{d:L\sb{i}\to L\sb{i-1}}

\begin{mysubsection}{Construction of the Lie-Massey products}
\label{sclmp}
Assume given an ordered set \w{(x\sb{1},\dotsc,x\sb{n})} of cycles in \w[,]{L\sb{\ast}}
such that for each ordered set \w{I=(i\sb{1},i\sb{2},\dotsc,i\sb{m})} with
\w[,]{1\leq i\sb{1}<i\sb{2}<\dotsc<i\sb{m}\leq n} there is an element
\w{x\sb{I}\in L\sb{\ast}} with \w{d(x\sb{I})=\widetilde{x\sb{I}}} for
\begin{myeq}\label{eqlmp}
  \widetilde{x}\sb{I}~:=~\sum\sb{(J,K)}\ (-1)\sp{\vare(J,K)+|x\sb{J}|+1}\cdot
    [x\sb{J},x\sb{K}]~.
\end{myeq}
\noindent Here \w{J=(j\sb{1},j\sb{2},\dotsc,j\sb{s})} and
\w{K=(k\sb{1},k\sb{2},\dotsc,k\sb{t})} are non-intersecting subsequences of $I$
with \w[,]{j\sb{1}<k\sb{1}} and \w{\vare(J,K)} is the sum of the products
\w{(|x\sb{j\sb{p}}|+1)(|x\sb{k\sb{q}}|+1)} with \w{k\sb{q}<j\sb{p}}
(compare \S \ref{dksign}).  We start with
\w{\widetilde{x}\sb{i,j}=(-1)\sp{|x\sb{J}|+1}\cdot[x\sb{i},x\sb{j}]}
  for\w[.]{1\leq i<j\leq n}

The collection of elements \w{\{x\sb{I}\}\sb{I=(i\sb{1},i\sb{2},\dotsc,i\sb{m})}}
for \w{1\leq m<n} is called a \emph{defining system} for the Lie-Massey bracket
\w{[\lra{x\sb{1}},\dotsc,\lra{x\sb{n}}]} on the corresponding homology classes
in \w[,]{H\sb{\ast}(L\sb{\ast},d)} and the cycle \w{\widetilde{x}\sb{I\sb{\bn}}}
for \w{I\sb{\bn}:=(i\sb{1},\dotsc,i\sb{n})} represents the corresponding \emph{value}
of \w{[\lra{x\sb{1}},\dotsc,\lra{x\sb{n}}]} in \w{H\sb{\ast}(L\sb{\ast},d)}
(with indeterminacy resulting from the various choices of defining systems).

Following Oukili's Lie algebra version in \cite{OukiH} of the filtered model for
a DGA (compare \cite{HStaO}), we may replace \w{(L\sb{\ast},d)} by a quasi-isomorphic
bigraded free differential Lie algebra \w[,]{(\hat{L}\sb{\ast\ast},D)} in two steps:
first, we choose a minimal free resolution \w{(\hat{L}\sb{\ast\ast},\hat{d})}
of the coformal DGL
\w[,]{(H\sb{\ast}(L\sb{\ast}),0)} with the new (homological) grading
determined iteratively by choosing suitable splittings over $\QQ$
(see \cite[\S I, Theorem 4]{OukiH}). We then perturb the differential $\hat{d}$ on
\w{\hat{L}\sb{\ast\ast}} into a filtered model \w[,]{(\hat{L}\sb{\ast\ast},D)}
which is a cofibrant replacement for \w{(L\sb{\ast},d)} itself
(see \cite[II, \S 3]{OukiH}).

We may then replace \w{(\hat{L}\sb{\ast\ast},D)} by a simplicial free DGL \w{\Vd}
with the same bigraded set of Lie algebra generators (see
\cite[Proposition 5.13]{BlaR} and compare \cite[\S 4]{QuiR}). In particular, we have
a $\hat{d}$-cycle \w{\widetilde{x}\sb{I\sb{\bn}}} in \w{\hat{L}\sb{n-2,N}}
(for suitable $N$), and since the
quasi-isomorphic \w{(H\sb{\ast}(L\sb{\ast}),0)} is coformal, this cycle must
be hit by a class \w[,]{y\in\hat{L}\sb{n-1,N}} represented by a DGL $N$-sphere
\w{\Sigma\sp{N}\QQ} in \w[.]{V\sb{n-1}} We now take the \wwb{n-1}truncation
of corresponding restricted simplicial object to be \w{\Ud} in \wref{eqlongsimp}
(starting at \w{n-1} rather than \w[,]{n+1} and ending at \w[).]{U\sb{0}}
We may then use the correspondence with \wref{eqerterm} or \wref{diaghho} noted there,
with \w{\Sigma\sp{N}\QQ} as \w{\Sigma\sp{k}A} mapping to the \wwb{n-2}truncation
of \w[.]{\Vd} Note that the internal differentials in each \w{V\sb{n}} are trivial,
since it is a wedge of rational spheres, so by Proposition \ref{pnoratho}
the resulting diagram \wref{diaghho} has a single obstruction to rectification:
namely, the \wwb{n-1}st order operation which is
the Lie-Massey bracket \w[.]{[\lra{x\sb{1}},\dotsc,\lra{x\sb{n}}]}

In principle, one can extract an explicit formula for the
\ww{E\sp{2}}-class \w{\alpha\sb{n}} in the Bousfield-Friedlander spectral sequence
representing \w{[\lra{x\sb{1}},\dotsc,\lra{x\sb{n}}]} for \w{\Wd} constructed as above,
from the previous paragraph, and as before (see \S \ref{sivwp} and
\S \ref{sivhwp}), this allows us to obtain an integral version of the
Lie-Massey bracket, using some multiple of \w{\alpha\sb{n}} in those cases where
the inductive construction does not guarantee that it is itself a permanent cycle.
Again, we shall not attempt to do so here, restricting ourselves to the
following simple example:
\end{mysubsection}

\begin{example}\label{eglmp}
  When \w[,]{n=3} assume given a DGL \w{L\sb{\ast}} having three cycles
  \w[,]{x\sb{p}} \w[,]{x\sb{q}} and \w[,]{x\sb{r}} with \w{x\sb{k}} in degree
  \w{k\in\{p,q,r\}} (which we assume to be odd, for simplicity), ordered
  \w[,]{x\sb{p}\prec x\sb{q}\prec x\sb{r}} and a fourth cycle $z$ in
  degree \w{m=p+q+r+1} (representing the Lie-Massey product).

The free bigraded model \w{\hat{L}\sb{\ast\ast}} for \w{(H\sb{\ast}(L\sb{\ast}),0)}
has classes \w{x\sb{k}\in\hat{L}\sb{0,k}} for each \w[,]{k\in\{p,q,r\}} a class
\w[,]{z\in\hat{L}\sb{0,m}} and classes \w{x\sb{i,j}\in \hat{L}\sb{1,j+k}} with
\w{d(x\sb{j,k})=\widetilde{x}\sb{j,k}:=[x\sb{i},x\sb{j}]} for each pair \w{(i,j)}
with \w[.]{x\sb{p}\preceq x\sb{i}\prec x\sb{j}\preceq x\sb{r}} Note that this
yields a cycle
\begin{myeq}\label{eqslmp}
\widetilde{x}\sb{p,q,r}~=~[x\sb{p},x\sb{q,r}]+[x\sb{p,q},x\sb{r}]+[x\sb{p,r},x\sb{q}]
\in\hat{L}\sb{1,m}~,
\end{myeq}
\noindent which requires a new class \w{x\sb{p,q,r}\in\hat{L}\sb{2,m}} with
\w[.]{d(x\sb{p,q,r})=\widetilde{x}\sb{p,q,r}}

In the filtered model \w{(\hat{L}\sb{\ast\ast},D)} for \w{(L\sb{\ast},d)}
we add to $d$ a perturbation \w[,]{\delta(x\sb{p,q,r})=-z} with \w{D=d+\delta}
(to ensure that $z$ is the value of the secondary Lie-Massey product
\w{[\lra{x\sb{p}},\lra{x\sb{q}},\lra{x\sb{r}}]} represented by
\w[).]{\widetilde{x}\sb{p,q,r}}

The corresponding simplicial DGL \w{\Vd} may be re-interpreted as a simplicial
space \w{\bVd} with each \w{\bV\sb{k}} a wedge of (rational or integral) spheres,
forming a resolution of the space $\bX$ realizing \w{L\sb{\ast}} rationally.
We start with \w{\oV{0}=\widehat{\bV}\sb{0}\vee\bS{m+1}} for
\w{\widehat{\bV}\sb{0}=\bS{p+1}\vee\bS{q+1}\vee\bS{r+1}}
and \w[.]{\oV{1}=\bS{p+q+1}\lo{p,q}\vee\bS{p+r+1}\lo{p,r}\vee\bS{q+r+1}\lo{q,r}}
We denote the fundamental class in \w{\pi\sb{p+1}\bS{p+1}} by \w[,]{\iot{p}} and
so on, and the attaching map \w{d\sb{0}=g} with
\w{g\rest{\bS{i+j+1}\lo{i,j}}=[\iot{i},\iot{j}]}
for \w[.]{\{i,j\}\subset\{p,q,r\}} We have \w{\oV{2}=\bS{m}} (corresponding to
\w[),]{x\sb{p,q,r}} with
\begin{myeq}\label{eqslmpval}
  d\sb{0}(\iota\sb{m})~=~
  \varphi=[s\sb{0}\iot{p},\iota\lo{q,r}]+[\iota\lo{p,q},s\sb{0}\iot{r}]
  +[\iota\lo{p,r},s\sb{0}\iot{q}]
\end{myeq}
\noindent as in \wref[.]{eqtriplewp}

The perturbation \w{D=d+\delta} of the filtered model is encoded as follows: we let
\w[,]{\bV\sb{1}~:=~\oV{1}\vee s\sb{0}\bV\sb{0}\vee C\bS{m}\lo{1}\vee C\bS{m}\lo{2}}
\noindent and similarly \w[,]{\bV\sb{0}:=\oV{0}\vee C\oV{1}} with
\w{d\sb{i}:\oV{2}=\bS{m}\hra C\bS{m}\lo{i}} \wb{i=1,2} and
\w{d\sb{1}:\oV{1}\hra C\oV{1}}  the inclusions into the appropriate cones. However,
\w{d\sb{1}:C\bS{m}\lo{i}\hra\bV\sb{0}} for \w{i=1,2} will be the inclusions into the
two hemispheres of the wedge summand \w[.]{\bS{m+1}}

Since \w[,]{d\sb{0}d\sb{0}=d\sb{0}d\sb{1}} we see that
\w{d\sb{0}:C\bS{m}\lo{1}\to\oV{0}} is a (fixed) null-homotopy $G$ witnessing
the Jacobi identity for \w[,]{(\iot{p},\iot{q},\iot{r})} while
\w{d\sb{0}:C\bS{m}\lo{2}\to\bV\sb{0}:=\oV{0}\vee C\oV{1}} is \w{C\varphi}
as in \wref[.]{eqslmpval} Thus if \w{\vare:\bV\sb{0}\to\bX} is the augmentation
(where \w{L\sb{\ast}} is a rational model for $\bX$), with
\w{\vare\rest{\widehat{\bV}\sb{0}}=f} representing \w[,]{x\sb{p}\bot x\sb{q}\bot x\sb{r}}
then the simplicial identity
\w{\vare d\sb{0} =\vare d\sb{1}} shows that \w{\vare\rest{C\oV{1}}=F} consists of
nullhomotopies for each of the Whitehead products \w[:]{[\iot{i},\iot{j}]}
\mysdiag[\label{eqsimptriplwp}]{
\bS{m} \ar[rr]\sp{d\sb{0}=\varphi} \ar@{_{(}->}[rrd]\sp{d\sb{1}=\inc}
  \ar@{_{(}->}[rrdd]\sb{d\sb{2}=\inc} &&
  s\sb{0}\bV\sb{0}\vee\oV{1} \ar[rr]\sp{d\sb{0}=g}
  \ar@{_{(}->}[rrd]\sb(0.2){d\sb{1}=\inc} && \widehat{\bV}\sb{0} \ar[rrd]\sp{\vare=f} && \\
&& C\bS{m}\lo{1}\ar[rru]\sb(0.8){d\sb{0}=G} \ar@{_{(}->}[rrd]\sp(0.3){d\sb{1}=\inc\sb{1}} &&
C\oV{1} \ar[rr]\sp{\vare=F} &&\bX~.\\
&& C\bS{m}\lo{2} \ar[rru]\sb(0.8){d\sb{0}=C\varphi}
\ar@{_{(}->}[rr]\sb{d\sb{1}=\inc\sb{2}} &&
\bS{m+1} \ar[rru]\sb{\vare=z}
}
\noindent Thus the class $z$ in fact represents the Lie-Massey product
\w{[\lra{x\sb{p}},\lra{x\sb{q}},\lra{x\sb{r}}]} (compare \cite[(1.17)]{BBSenHS}).

Finally, the correspondence of \wref{eqlongsimp} with \wref[,]{diaghho}
with the $1$-truncation of \w{\Vd} in the role of \w{\Ud} and
\w{\varphi:\bS{m}\to\bW\sb{1}} as $\hf$, exhibits this Lie-Massey product as
a secondary operation in our sense.
\end{example}

Although the definition of the Lie-Massey products in \S \ref{sclmp}
resembles that of the higher Whitehead products, the following example shows that
the latter do not determine the former:

\begin{mysubsection}{The example of Buijs and Moreno-Fern\'{a}ndez}
\label{sbmf}
In \cite[\S 2]{BMoreF}, Buijs and Moreno-Fern\'{a}ndez consider the
total space $\bX$ of an \ww{\bS{7}}-bundle over \w[,]{\KZ{2}\times\KZ{4}} with
\w{\Li} model \w{A=\QQ\lra{\bar{x}\sb{1},\bar{y}\sb{3},\bar{z}\sb{6}}}
with \w[.]{\ell\sb{2}(\bar{y},\bar{y})=\ell\sb{3}(\bar{y},\bar{x},\bar{x})=\bar{z}}
The corresponding CDGA model \w{(\Lambda V,\dd)} then has
\w{V=\QQ\lra{\tilde{x}\sb{2},\tilde{y}\sb{4},\tilde{z}\sb{7}}} with
\w{\dd(\tilde{x})=\dd(\tilde{y})=0} and
\w[,]{\dd(\tilde{z})=\tilde{y}\sp{2}+\tilde{y}\tilde{x}\sp{2}} which they use to show
that all higher Whitehead products in \w{\pis(X\sb{\QQ})} vanish
(using \cite[Theorem 4.1]{AlldR2}), even though \w{X\sb{\QQ}} is not coformal
(by \cite[Corollary 4]{BMoreF}).

Applying Quillen's functor $\LL$ (the cobar construction on the vector space dual
\w[,]{\bar{C}\sb{\ast}} whose elements we indicate by $\hat{x}$, \w[,]{\widehat{xy}}
etc.) to \w{(\Lambda V,d)} yields the DGL model
\w{(\LL\sb{s\sp{-1}\bar{C}},d=d\sb{0}+d\sb{1})} for $\bX$, described by\vsm:

\begin{center}
\begin{tabular}{|l||r|r|r|r|l|r|l|}
\hline
degree & 1&2&3&4&5&6&7\\
\hline 
$s\sp{-1}\bar{C}\sb{\ast}$ & $\hat{x}$ & $[\hat{x},\hat{x}]$ &
{\scriptsize$\frac{1}{2}$}$\xla{d\sb{1}}\widehat{x\sp{2}}$ &
$[\widehat{x\sp{2}},\hat{x}]$ &
  $\xla{d\sb{1}}\widehat{x\sp{3}}$ &   $[\hat{y},\hat{y}]$ &
{\scriptsize $\frac{1}{2}$}
\hsn\raisebox{-0.5ex}{\scriptsize{$d\sb{1}$}}\hspace{-1.5mm}$\nwarrow
\hsn \xla{d\sb{1}}\hat{y\sp{2}}$ \\
  &  &  & $\hat{y}$ & $[\hat{y},\hat{x}]$ &
  {\scriptsize $\frac{1}{2}$}$\xla{d\sb{1}}\widehat{yx}$ & $\hat{z}$ &
$\xla{d\sb{0}}\widehat{y\sp{2}+yx\sp{2}}$\\
& & & & & $[[\hat{y},\hat{x}],\hat{x}]$ &
{\scriptsize $\frac{1}{2}$}$\xla{d\sb{1}}[\widehat{yx},\hat{x}]$ &
\hsn\raisebox{1ex}{\scriptsize{$d\sb{1}$}}\hspace{-1.5mm}$\swarrow$ \\
  & & & & & $[[\hat{x},\hat{x}],\hat{y}]$ &
  {\scriptsize$\frac{1}{2}$}$\xla{d\sb{1}}[\widehat{x\sp{2}},\hat{y}]$ &
  {\scriptsize$\frac{1}{2}$}\raisebox{1ex}{\scriptsize{$d\sb{1}$}}\hspace{-1.5mm}$
  \swarrow$ \\
\hline
\end{tabular}
\end{center}

\vsm\quad

\noindent with the \w{d\sb{0}} differential coming from \w[,]{\bar{C}\sb{\ast}}
and the coproduct on \w{\widehat{ab}} in \w{\bar{C}\sb{\ast}} yielding
\w{d\sb{1}(\widehat{ab})=\frac{1}{2}[\hat{a},\hat{b}]} in
\w[.]{\LL\sb{s\sp{-1}\bar{C}}}

In a minimal model DGL model for $\bX$ we may omit $\hat{z}$, which is homologous to
both \w{[\hat{y},\hat{y}]} and the Lie-Massey product \w{\lra{\hat{y},\hat{x},\hat{x}}}
(represented by \w{\alpha=2[\widehat{yx},\hat{x}]-[\widehat{x\sp{2}},\hat{y}]}
in degree $6$, since
\begin{myeq}\label{eqjacidenti}
2[[\hat{y},\hat{x}],\hat{x}]-[[\hat{x},\hat{x}],\hat{y}]=0
\end{myeq}
\noindent by the Jacobi identity).

To interpret this in our language, consider the simplicial rational space (or
the corresponding simplicial coformal DGL) \w[,]{\Wd} with
\w{\bW\sb{0}=\bS{2}\sb{x}\vee\bS{4}\sb{y}} and
\w[.]{\bW\sb{1}=\bS{3}\sb{x\sp{2}}\vee\bS{5}\sb{xy}\vee s\sb{0}\bW\sb{0}}
The non-trivial simplicial face maps are
\w{d\sb{0}(\iot{3}\sb{x\sp{2}})=[\iot{2}\sb{x},\iot{2}\sb{x}]}
and \w[.]{d\sb{0}(\iot{5}\sb{xy})=[\iot{4}\sb{y},\iot{2}\sb{x}]}
We see that
$$
\bar{\alpha}~:=~
2[\iot{5}\sb{xy},s\sb{0}\iot{2}\sb{x}]-[\iot{3}\sb{x\sp{2}},s\sb{0}\iot{4}\sb{y}]~
  \in~\pi\sb{6}\bW\sb{1}
$$
\noindent represents the Lie-Massey product \w{\lra{\hat{y},\hat{x},\hat{x}}} in
\w[,]{\pi\sb{7}\bX\sb{\QQ}} by the discussion in Example \ref{eglmp}.
Moreover, we can deduce from this that the same holds integrally for
\w{\bX':=\|\Wd\|} (whose $7$-th Postnikov section agrees with $\bX$ rationally).
\end{mysubsection}

%
%
\sect{Complex projective spaces}
\label{ccpn}
In \cite[\S 8.6]{BBSenHS}, we showed that for each \w[,]{n\geq 1} there is a rational
simplicial space \w{\Vd} with \w[,]{\|\Vd\|\simeq\CP{n}\sb{\QQ}} having a single
non-degenerate (rational) sphere \w{\bS{k+2}=\bS{k+2}\sb{\QQ}} in \w{\bV\sb{k}} for
each \w[,]{0\leq k<n} where the fundamental class
\w{\iot{k+2}} has \w[.]{d\sb{0}(\iot{k+2})=\gamm{k}\in\pi\sb{k+2}\bV\sb{k-1}}

Here \w[,]{\gamm{1}=\eta\sb{2}} while for \w{n\geq 2}
\begin{myeq}\label{eqgamman}
\gamm{n}~:=~\sum\sb{j=2}\sp{\lfloor\frac{n+3}{2}\rfloor}\ \sum\sb{(I,J)\in\cI{n}{j-2}}\
(-1)\sp{n\cdot j}\sgn{I,J}\cdot[s\sb{I}\iot{n-j+3},\ s\sb{J}\iot{j}]
\end{myeq}
\noindent in \w[.]{\pi\sb{n+2}\bV\sb{n-1}} As an element in \w{\Ett{n-1}{n+2}}
it also represents \w{[g\sb{n}]} as a value of the corresponding $n$-th order
rational homotopy operation in \w[.]{\pi\sb{2n+3}\CP{n+1}}

Specializing Definition \ref{drhwp} to the case \w[,]{\vS=(\bS{2}\lo{i})\sb{i=1}\sp{n+1}}
we may define maps \w{f\sb{k}:\bW\sb{k}\to\bV\sb{k}} inductively as follows\vsm:

\noindent a.\ For each of the $n$ generators \w{\iotl{2}{i}\in\pi\sb{2}\bS{2}\lo{i}} in
\w[,]{\bW\sb{0}} we set \w[\vsm.]{f\sb{0}(\iotl{2}{i})=\iot{2}\in\pi\sb{2}\bS{2}}

\noindent b.\ For each \w{\iotl{3}{i,j}\in\pi\sb{2}\bS{3}\lo{i,j}} in \w[,]{\bW\sb{1}}
set \w[.]{f\sb{1}(\iotl{3}{i,j})=2\iot{3}\in\pi\sb{3}\bS{3}\subset\pi\sb{3}\bV\sb{1}}

\noindent c.\ More generally, for \w{1\leq k\leq m} and
\w{\tau=(\tau\sb{1},\dotsc,\tau\sb{k+1})} an ordered
subset of \w[,]{\{1,\dotsc,m\}} denote the fundamental class in
\w{\pi\sb{k+2}\bS{k+2}\lo{\tau}\subset\pi\sb{k+2}\bW\sb{k}} by
\w{\iotl{k+2}{\tau}} as in \wref[,]{eqattmap} and let
\begin{myeq}\label{eqgammak}
f\sb{k}(\iotl{k+2}{\tau})~=~(-1)\sp{\lfloor\frac{k}{2}\rfloor}\cdot(k+1)!\cdot\iot{k+2}
\end{myeq}
\noindent in \w[.]{\pi\sb{k+2}\bS{k+2}\subseteq\pi\sb{k+2}\bV\sb{k}}

\begin{prop}\label{pcpn}
For fixed \w[,]{n\geq 1} the maps \w{f\sb{k}} above fit together to define
a map of rational simplicial spaces \w[.]{f:\Wd\to\Vd}
\end{prop}

\begin{proof}
It suffices to show, by induction on \w[,]{1\leq m<n} that
\begin{myeq}\label{eqdgammak}
  d\sb{0}\sp{V}\circ f\sb{k}(\iotl{k+2}{\tau})~=~
  f\sb{k-1}\circ d\sb{0}\sp{W}(\iotl{k+2}{\tau})~,
\end{myeq}
\noindent since all other face maps on \w{\iotl{k+2}{\tau}} and \w{\iot{k+2}} are $0$.

Indeed, for \w{m=1} we have
\begin{myeq}\label{eqdgammaone}
d\sb{0}\sp{V}\circ f\sb{1}(\iotl{3}{i,j})~=~d\sb{0}\sp{V}(2\iot{3})~=~
[f\sb{0}(\iottl{2}{i}),f\sb{0}(\iottl{2}{j})]~=~
f\sb{0}\circ d\sb{0}\sp{W}(\iotl{3}{i,j})~,
\end{myeq}
\noindent since for \w[,]{\CP{1}=\bS{2}}
\begin{myeq}\label{eqtwoeta}
2\eta\sb{2}~=~[\iot{2},\iot{2}]~,
\end{myeq}
\noindent so \w[.]{d\sb{0}\sp{V}(\iot{3})=\gamm{1}=\eta\sb{2}}

Note that for \w{\Wd} as defined above, all Koszul signs are
\w[,]{+1} since all spheres in $\vS$ are $2$-dimensional, while for
\w{{(\sigma',\sigma'')\in\wI{m}{m-k}}} in \wref{eqattmap} for
\w{\phi\sb{\vS'}\in\pis\bW\sb{m-2}}
we have \w[,]{\deg{\sigma'}=m-k} so
\begin{myeq}\label{eqglobal}
\text{the global sign in \w{\Wd} is }\ (-1)\sp{m}~.
\end{myeq}

In the induction step, we must check that \wref{eqdgammak} also
holds for \w[,]{k=m+1} with \w{\iotl{m+3}{\tau}\in\pi\sb{m+3}\bW\sb{m+1}} for $\tau$
a subset of \w{\{1,\dotsc,n+1\}} of cardinality \w[.]{m+2}

Note that \w{d\sb{0}\sp{W}(\iotl{m+3}{\tau})=\phi\sb{\vS'}\in\pi\sb{m+3}\bW\sb{m}}
for  \w[.]{\vS'=(\bS{2}\lo{\tau\sb{1}},\dotsc,\bS{2}\lo{\tau\sb{m+2}})}
By \wref{eqattmap} (with global sign \w[,]{(-1)\sp{m+2}} by \wref[),]{eqglobal}
in the expansion of \w[\vsm:]{\phi\sb{\vS'}}

%
%
\noindent (i)\ There are \w{m+2} summands of the form
\w[,]{[\iotl{m+2}{\widehat{\tau}},s\sb{J}\iotl{2}{\tau\sb{i}}]} where $\widehat{\tau}$
is obtained from $\tau$ by omitting \w[.]{\tau\sb{i}} We assume by that \wref{eqgammak}
holds for \w[,]{k=m} so
$$
f\sb{m}(\iotl{m+2}{\widehat{\tau}})~=~
(-1)\sp{\lfloor\frac{m}{2}\rfloor}\cdot(m+1)!\cdot\iot{m+2}~,
$$
\noindent while \w[,]{f\sb{0}(\iotl{2}{\tau\sb{i}})=\iot{2}} so
\w[.]{f\sb{m}(s\sb{J}\iotl{2}{\tau\sb{i}})=s\sb{J}\iot{2}}
Thus, applying \w{f\sb{m}} to each of these \w{m+2} such summands yields
\w{(-1)\sp{\lfloor\frac{m}{2}\rfloor}\cdot(m+1)!\cdot[\iot{m+2},s\sb{J}\iot{2}]}
\noindent in \w[,]{d\sb{0}\sp{V}(\iot{m+3})=\gamm{m+1}} for a total of
\w[.]{(m+2)!\cdot[\iot{m+2},s\sb{J}\iot{2}]} Note that the latter has global sign
\w{(-1)\sp{(m+1)\cdot 2}=+1} in \wref[.]{eqgamman} The sign coming from
\w{f\sb{m}(\iotl{m+2}{\widehat{\tau}})} is
\begin{myeq}\label{eqsignm}
(-1)\sp{m+2}\cdot(-1)\sp{\lfloor\frac{m}{2}\rfloor}~=~\begin{cases}
(-1)\sp{\frac{m}{2}}&\text{if $m$ is even}\\
(-1)\cdot(-1)\sp{\frac{m-1}{2}}=(-1)\sp{\frac{m+1}{2}} &\text{if $m$ is odd~,}
\end{cases}
\end{myeq}
\noindent so in either case we get \w[,]{(-1)\sp{\lfloor\frac{m+1}{2}\rfloor}}
as required\vsm.

%
%
\noindent (ii)\ There are \w{\binom{m+2}{2}=\frac{(m+2)!}{m!\cdot 2}} summands of the form
\w{[s\sb{I}\iotl{m+1}{\widehat{\tau}},s\sb{J}\iotl{3}{\tau\sb{i},\tau\sb{j}}]}
where $\widehat{\tau}$ is now obtained from $\tau$ by omitting
\w[.]{\{\tau\sb{i},\tau\sb{j}\}} Applying \w{f\sb{m}} to each yields
\w[,]{(-1)\sp{\lfloor\frac{m-1}{2}\rfloor}\cdot m!\cdot
  [s\sb{I}\iot{m+1},2s\sb{J}\iot{3}]} with global sign
\w{(-1)\sp{(m+1)\cdot 3}=(-1)\sp{m+1}} in \wref[.]{eqgamman}
Since there are \w{\binom{m+2}{2}} of them, the total is
\w[,]{(m+2)!\cdot[s\sb{I}\iot{m+1},2s\sb{J}\iot{3}]} with sign
(after multiplication by \w[)]{(-1)\sp{m+1}} equal to:
$$
(-1)\sp{(m+2)+(m+1)}\cdot(-1)\sp{\lfloor\frac{m-1}{2}\rfloor}~=~\begin{cases}
(-1)\cdot(-1)\sp{\frac{m-2}{2}}=(-1)\sp{\frac{m}{2}} &\text{if $m$ is even}\\
(-1)\cdot(-1)\sp{\frac{m-1}{2}}=(-1)\sp{\frac{m+1}{2}}&
  \text{if $m$ is odd,}
\end{cases}
$$
\noindent which again equals \w{(-1)\sp{\lfloor\frac{m+1}{2}\rfloor}} in either case\vsm.

%
%
\noindent (iii)\ In general, there are
\begin{myeq}\label{eqbinom}
\binom{m+2}{k}~=~\frac{(m+2)!}{(m+2-k)!\cdot k!}
\end{myeq}
summands of the form \w[,]{[s\sb{I}\iotl{m+3-k}{\tau'},s\sb{J}\iotl{k+1}{\tau''}]}
where \w[.]{(\tau',\tau'')\in\wI{m+2}{m+2-k}}
Applying \w{f\sb{m}} to each yields
$$
[(-1)\sp{\lfloor\frac{m+1-k}{2}\rfloor}\cdot(m+2-k)!\cdot s\sb{I}\iot{m+3-k},\,
(-1)\sp{\lfloor\frac{k-1}{2}\rfloor}\cdot k!\cdot s\sb{J}\iot{k+1}]~,
$$
\noindent with global sign \w{(-1)\sp{(m+1)\cdot(k+1)}} in \wref[.]{eqgamman}

Multiplying by \wref{eqbinom} yields \w{(m+2)!} summands
\w[,]{[s\sb{I}\iot{m+3-k},\,s\sb{J}\iot{k+1}]} with signs
\begin{equation*}
\begin{split}
(-1)\sp{m+2}&\cdot(-1)\sp{(m+1)\cdot(k+1)}\cdot(-1)\sp{\lfloor\frac{m+1-k}{2}\rfloor}
\cdot(-1)\sp{\lfloor\frac{k-1}{2}\rfloor}=\\
~=&~\begin{cases}
  (-1)\sp{\frac{m-k}{2}+\frac{k-2}{2}+1}=(-1)\sp{\frac{m}{2}}
  =(-1)\sp{\lfloor\frac{m+1}{2}\rfloor} & \text{$m$ and $k$ even}\\
  (-1)\sp{\frac{m-k+1}{2}+\frac{k-1}{2}}=(-1)\sp{\frac{m}{2}}=
  (-1)\sp{\lfloor\frac{m+1}{2}\rfloor} &\text{$m$ even, $k$ odd}\\
  (-1)\cdot(-1)\sp{\frac{m-k}{2}+\frac{k-1}{2}}=(-1)\sp{\frac{m-1}{2}+1}=
(-1)\sp{\frac{m+1}{2}} & \text{$m$ and $k$ odd}\\
(-1)\cdot(-1)\sp{\frac{m+1-k}{2}+\frac{k-2}{2}}=(-1)\sp{\frac{m-1}{2}+1}
=(-1)\sp{\frac{m+1}{2}}
&\text{$m$ odd, $k$ even~,}\\
\end{cases}
\end{split}
\end{equation*}
\noindent which again equals \w{(-1)\sp{\lfloor\frac{m+1}{2}\rfloor}} in all cases.
This completes the induction step, and thus the proof.
\end{proof}

\begin{corollary}\label{cocpn}
For each \w[,]{n\geq 1} \w{(-1)\sp{\lfloor\frac{n}{2}\rfloor}(n+1)!} times the class of
the rational Hopf map \w{g\sb{n}:\bS{2n+1}\to\CP{n}} is the rational \wwb{n+1}st
order Whitehead product of a generator \w{\iol{n}\in\pi\sb{2}\CP{n}} with itself.
\end{corollary}

This is in fact true integrally (see \cite[Corollary 2]{GPorHO}), with the same
coefficient (we can of course choose the definition of the higher order Whitehead
products so that the sign is always positive). For \w[,]{n=1} this is just
\wref[.]{eqtwoeta}

\begin{example}\label{egcpn}
For \w[,]{n=4} we have \w{\bW\sb{0}=\bigvee\sb{i=1}\sp{4}\bS{2}\lo{i}}
(with fundamental classes \w[)]{\iottl{2}{i}\in\pi\sb{2}\bS{2}\lo{i}} and
\w[,]{\oW{1}=\bigvee\sb{1\leq i<j\leq 4}\bS{3}\lo{i,j}}
with attaching maps \w{d\sp{W}\sb{0}(\iotl{3}{i,j})=[\iottl{2}{i},\iottl{2}{j}]} on
\w[,]{\bS{3}\lo{i,j}} as in Example \ref{eghwp}, with
\w[,]{f\sb{0}(\iottl{2}{i})=\iot{2}} and
\w[.]{f\sb{1}(\iotl{3}{i,j})=2\iot{3}}

Similarly, \w[,]{\oW{2}=\bigvee\sb{1\leq i<j<k\leq 4}\bS{4}\lo{i,j,k}} with
\begin{myeq}\label{eqtriplewps}
  d\sp{W}\sb{0}(\iotl{4}{i,j,k})~=~
  -\left([\iotl{3}{i,j},s\sb{0}\iottl{2}{k}]+[\iotl{3}{i,k},s\sb{0}\iottl{2}{j}]+
[\iotl{3}{j,k},s\sb{0}\iottl{2}{i}]\right)
\end{myeq}
\noindent in \w[,]{\pi\sb{4}\bW\sb{1}} by \wref{eqtriplewp} or \wref[.]{eqglobal}

On the other hand, \w{d\sp{V}\sb{0}(\iot{4})=\gamm{2}=[\iot{3},s\sb{0}\iot{2}]} by
\wref{eqgamman} with \w[,]{f\sb{2}(\iotl{4}{i,j,k})=-6\iot{4}} and indeed:
$$
f\sb{1}\circ d\sb{0}\sp{W}(\iotl{4}{i,j,k})~=~-6[\iot{3},s\sb{0}\iot{2}]~=~
d\sb{0}\sp{V}\circ f\sb{2}(\iotl{4}{i,j,k})~.
$$

By \wref{eqattmap} (with global sign \w[),]{(-1)\sp{4}} the fourth order Whitehead
product is represented in \w{\pi\sb{5}\bW\sb{2}} by
\begin{myeq}\label{eqquadruplwp}
\begin{split}
%
\phi\sb{\vS}=&[\iotl{4}{i,j,k},s\sb{1}s\sb{0}\iottl{2}{\ell}]
+[\iotl{4}{i,j,\ell},s\sb{1}s\sb{0}\iottl{2}{k}]
+[\iotl{4}{i,k,\ell},s\sb{1}s\sb{0}\iottl{2}{j}]
+[\iotl{4}{j,k,\ell},s\sb{1}s\sb{0}\iottl{2}{i}]\\
%
&+[s\sb{0}\iotl{3}{i,j},\,s\sb{1}\iotl{3}{k,\ell}]
-[s\sb{1}\iotl{3}{i,j},\,s\sb{0}\iotl{3}{k,\ell}]
+[s\sb{0}\iotl{3}{i,k},\,s\sb{1}\iotl{3}{j,\ell}]
-[s\sb{1}\iotl{3}{i,k},\,s\sb{0}\iotl{3}{j,\ell}]\\
&\hs +[s\sb{0}\iotl{3}{i,\ell},\,s\sb{1}\iotl{3}{j,k}]
-[s\sb{1}\iotl{3}{i,\ell},\,s\sb{0}\iotl{3}{j,k}]~\in~\pi\sb{5}\bW\sb{2}
\end{split}
\end{myeq}
\noindent (compare \wref[),]{eqquadruplewp} while
\w{\gamm{3}=[\iot{4},s\sb{1}s\sb{0}\iot{2}]-[s\sb{0}\iot{3},s\sb{1}\iot{3}]}
by \cite[(8.9)]{BBSenHS}, and since by \wref{eqanticomm} we have
\w[,]{[s\sb{1}\iot{3},s\sb{0}\iot{3}]=-[s\sb{0}\iot{3},s\sb{1}\iot{3}]}
altogether:
\begin{myeq}\label{eqftwo}
f\sb{2}(\phi\sb{\vS})~=~\left(4\cdot[-6\iot{4},s\sb{1}s\sb{0}\iot{2}]
+6[2s\sb{0}\iot{3},2s\sb{1}\iot{3}]\right)~=~-24\cdot\gamm{3}~.
\end{myeq}
\end{example}


\begin{thebibliography}{DKSm}
%
\bibitem[A]{AlldR2}
C.~Allday,
``Rational Whitehead products and a spectral sequence of Quillen, II'',\hsm
\textit{Hous.\ J.\ Math.} \textbf{3} (1977), pp.~301-308.
%
\bibitem[AA]{AArkS}
P.G.~Andrews \& M.~Arkowitz,
``Sullivan's minimal models and higher order Whitehead products'',\hsm
\textit{Can.\ J.\ Math.} \textbf{30} (1978), pp.~961-982.
%
\bibitem[BBS1]{BBSenT}
S.~Basu, D.~Blanc, \& D.~Sen,
``A note on Toda brackets'',\hsm
\textit{J.\ Homotopy \& Rel.\ Struct.} \textbf{15} (2020), pp.\ 495-510
%
\bibitem[BBS2]{BBSenHS}
S.~Basu, D.~Blanc, \& D.~Sen,
``The higher structure of unstable homotopy groups'',\hsm
\textit{Int.\ Math.\ Res.\ Notices} \textbf{7} (2024), pp.\ 5815-5849.
%
\bibitem[BBG]{BBGondH}
H.-J.~Baues, D.~Blanc, \& S.~Gondhali,
``Higher Toda brackets and Massey products'',\hsm
\textit{J.~Homotopy \& Rel.~Struct.} \textbf{11} (2016), pp.~643-677.
%
\bibitem[Be1]{BergM}
J.E.~Bergner,
``A model category structure on the category of simplicial categories'',\hsm
\textit{Trans.\ AMS} \textbf{359} (2007), pp.~2043-2058.
%
\bibitem[Be2]{BergnHI}
J.E.~Bergner,
\textit{The homotopy theory of $(\infty,1)$-categories},\hsm
Cambridge U.\ Press, Cambridge, UK, 2018.
%
\bibitem[Bl]{BlaR}
D.~Blanc,
``Homotopy operations and rational homotopy type'',\hsm
in G.Z.~Arone, J.R.~Hubbuck, R.~Levi, \&M.S.~Weiss, eds.,
\textit{Categorical decomposition techniques in algebraic topology
  (Isle of Skye, 2001)}, Prog.~Math. \textbf{215}, Birkha\"{u}ser, Boston, 2004,
pp.~47-75.
%
\bibitem[BJT1]{BJTurnHC}
D.~Blanc, M.W.~Johnson, \& J.M. Turner,
``Higher homotopy operations and cohomology'',
\textit{Journal of $K$-Theory} \textbf{5} (2010), pp.~167-200.
%
\bibitem[BJT2]{BJTurnHA}
D.~Blanc, M.W.~Johnson, \& J.M.~Turner,
``Higher homotopy operations and Andre-Quillen cohomology'',
\textit{Advances in Mathematics} \textbf{230} (2012), pp.~777-817.
%
\bibitem[BJT3]{BJTurnHH}
D.~Blanc, M.W.~Johnson, \& J.M.~Turner,
``Higher homotopy invariants for spaces and maps'',\hsm
\textit{Alg.\ Geom.\ Topology} \textbf{21} (2021), pp.~2425-2488.
%
\bibitem[BJT4]{BJTurnC}
D.~Blanc, M.W.~Johnson, \& J.M.~Turner,
``A constructive approach to higher homotopy operations",\hsm
in D.G.~Davis, H.-W.~Henn, J.F.~Jardine, M.W.~Johnson, \& C.~Rezk, eds.,
\textit{Homotopy Theory: Tools and Applications},
Contemp.\ Math.\ \textbf{729}, AMS, Providence, RI, 2019, pp.~21-74.
%
\bibitem[BMa]{BMarkH}
D.~Blanc \& M.~Markl,
``Higher homotopy operations'',\hsm
\noindent{Math.\ Zeit.} \textbf{345} (2003), pp.~1-29.
%
\bibitem[BMe]{BMeadS}
D.~Blanc \& N.J.~Meadows,
``Spectral sequences in {$(\infty, 1)$}-categories'',\hsm
\textit{J.\ Pure Appl.\ Alg.} \textbf{226} (2022), Paper No. 106905.
%
\bibitem[BS]{BSenM}
D.~Blanc \& D.~Sen,
``Mapping spaces and $R$-completion'',\hsm
\textit{J.\ Homotopy \& Rel.\ Struct.} \textbf{13} (2018), pp.\ 635-671.
%
\bibitem[BF]{BFrieH}
A.K.~Bousfield \& E.M.~Friedlander,
``Homotopy theory of $\Gamma$-spaces, spectra, and bisimplicial sets'',\hsm
in M.G.~Barratt \& M.E.~Mahowald, eds.,
\textit{Geometric Applications of Homotopy Theory, II}
Springer-\-Verlag \textit{Lec.\ Notes Math.} \textbf{658},
Berlin-\-New York, 1978, pp.~80-130
%
\bibitem[BK]{BKanH}
A.K.~Bousfield \& D.M.~Kan,
\textit{Homotopy Limits, Completions, and Lo\-ca\-li\-za\-tions},\hsm
Sprin\-ger \textit{Lec.\ Notes Math.} \textbf{304}, Berlin-\-New York, 1972.
%
\bibitem[BP]{BPanoTAT}
V.M.~Buchstaber \& T.E.~Panov,
\textit{Torus actions and their applications in topology and combinatorics},\hsm
AMS, Providence, RI, 2002.
%
\bibitem[BMo]{BMoreF}
U.~Buijs \& J.M.~Moreno-Fern\'{a}ndez,
``Formality criteria in terms of higher Whitehead brackets'',\hsm
\textit{Topology \& Appl.} \textbf{267} (2019), Paper No. 106901.
%
\bibitem[C]{JCohDS}
J.M.~Cohen,
``The decomposition of stable homotopy'',\hsm
\textit{Ann.\ Math.\ (2)} \textbf{87} (1968), pp.\ 305-320.
%
\bibitem[DHKS]{DHKSmitH}
W.G.~Dwyer, P.S.~Hirschhorn, D.M.~Kan, \& J.H.~Smith,
\textit{Homotopy limit functors on model categories and homotopical categories},\hsm
Math.\ Surveys \& Monographs \textbf{113}, AMS, Providence, RI, 2004.
%
\bibitem[DK]{DKanL}
W.G.~Dwyer \& D.M.~Kan,
``Simplicial Localization of Categories'',\hsm
\textit{J.\ Pure Appl.\ Alg.} \textbf{17} (1980), 267-284.
%
\bibitem[DKSm]{DKSmitH}
W.G.~Dwyer, D.M.~Kan, \& J.H.~Smith,
``Homotopy commutative diagrams and their realizations'',\hsm
\textit{J.\ Pure Appl.\ Alg.} \textbf{57} (1989), pp.~5-24.
%
\bibitem[DKSt]{DKStovB}
W.G.~Dwyer, D.M.~Kan, \& C.R.~Stover,
``The bigraded homotopy groups $\pi\sb{i,j}X$ of a pointed simplicial space'',\hsm
\textit{J.\ Pure Appl.\ Alg.} \textbf{103} (1995), pp.~167-188.
%
\bibitem[GJ]{GJarS}
P.G.~Goerss \& J.F.~Jardine,
\textit{Simplicial Homotopy Theory},\hsm
Progress in Mathematics \textbf{179}, Birkh\"{a}user, Basel-Boston, 1999.
%
\bibitem[HS]{HStaO}
S.~Halperin \& J.D.~Stasheff,
``Obstructions to homotopy equivalences'',\hsm
\textit{Adv.\ Math.} \textbf{32} (1979), pp.~233-279.
%
\bibitem[Ha]{HardHW}
K.A.~Hardie,
``Higher Whitehead products'',\hsm
\textit{Quart.\ J.\ Math.\ (2)} \textbf{12} (1961), pp.~241-249.
%
\bibitem[Hi]{HilS}
P.J.~Hilton,
``On the homotopy groups of the union of spheres'',\hsm
\textit{J.\ London Math.\ Soc.} \textbf{30} (1955), pp.~154-172.
%
\bibitem[J]{JoyQ}
A.~Joyal,
``Quasi-categories and Kan complexes'',\hsm
\textit{J.\ Pure Appl.\ Alg.} \textbf{175} (2002), pp.~22-38.
%
\bibitem[K]{KanD}
D.M.~Kan,
``A combinatorial definition of homotopy groups'',\hsm
\textit{Ann.\ Math.\ (2)} \textbf{67} (1958), pp.~282-302.
%
\bibitem[LM]{LMarkS}
T.J.~Lada \& M.~Markl,
``Strongly homotopy Lie algebras'',\hsm
\textit{Conn.\ in Alg.} \textbf{23} (1995), pp.~2147-2161.
%
\bibitem[LS]{LStaS}
T.J.~Lada \& J.D.~Stasheff,
``Introduction to SH Lie algebras for physicists'',\hsm
\textit{Internat.\ J.\ Theor.\ Phys.} \textbf{32} (1993), pp.~495-510.
%
\bibitem[L]{LurH}
J.~Lurie,
\textit{Higher Topos Theory},\hsm
Ann.\ Math.\ Studies \textbf{170}, Princeton U.\ Press, Princeton, 2009.
%
\bibitem[Mas]{MassN}
W.S.~Massey,
``A new cohomology invariant of topological spaces'',\hsm
\textit{Bull.\ AMS} \textbf{57} (1951), p.\ 74.
%
\bibitem[Mau]{MaunC}
C.R.F.~Maunder,
``Cohomology operations of the {$N$}-th kind'',\hsm
\textit{Proc.\ Lond.\ Math.\ Soc.\ (2)} \textbf{13} (1963), pp.\ 125-154.
%
\bibitem[May]{MayS}
J.P.~May,
\textit{Simplicial Objects in Algebraic Topology},\hsm
U.\ Chicago Press, Chicago, 1967.
%
\bibitem[O]{OukiH}
A.~Oukili,
``Sur l'homologie d'une alg{\'{e}}bre diff{\'{e}}rentielle de Lie'',\hsm
PhD Thesis, Univ.\ Nice, 1978.
%
\bibitem[P1]{GPorW}
G.J.~Porter,
``Higher order {Whitehead} products'',\hsm
\textit{Topology} \textbf{3} (1965), pp.~123-165.
%
\bibitem[P2]{GPorHG}
G.J.~Porter,
``On the homotopy groups of wedges of spheres'',\hsm
\textit{Amer.\ J.\ Math.} \textbf{87} (1965), pp.~297-314.
%
\bibitem[P3]{GPorHO}
G.J.~Porter,
``Higher order Whitehead products and Postnikov systems'',\hsm
\textit{Illinois J.\ Math.} \textbf{11} (1967), pp.~414-416.
%
\bibitem[P4]{GPorHP}
G.J.~Porter,
``Higher products'',\hsm
\textit{Trans.\ AMS} \textbf{148} (1970), pp.~315-345.
%
\bibitem[Q1]{QuiH}
D.G.~Quillen,
\textit{Homotopical Algebra},\hsm
Springer-\-Verlag \textit{Lec.\ Notes Math.} \textbf{20}, Berlin-\-New York,
1963.
%
\bibitem[Q2]{QuiR}
D.G.~Quillen,
``Rational homotopy theory'',\hsm
\textit{Ann.\ Math.\ (2)} \textbf{90} (1969), pp.\ 205-295.
%
\bibitem[Ret]{RetaL}
V.S.~Retakh,
``Lie-Massey brackets and $n$-homotopically multiplicative maps of differential
graded Lie algebras'',\hsm
\textit{J.\ Pure Appl.\ Alg.} \textbf{89} (1993), pp.~217-229.
%
\bibitem[Rez]{RezkM}
C.~Rezk,
``A model for the homotopy theory of homotopy theory'',\hsm
\textit{Trans.\ AMS} \textbf{353} (2001), pp.~973-1007.
%
\bibitem[RV]{RVeriI}
E.~Riehl \& D.~Verity,
``Infinity category theory from scratch'',\hsm
\textit{High.\ Struct.} \textbf{4} (2020), pp.~115-167.
%
\bibitem[Se]{SegC}
G.B.~Segal,
``Classifying spaces and spectral sequences'',\hsm
\textit{Pub.\ Math.\ Inst.\ Hautes {\'{E}}t.\ Sci.} \textbf{34} (1968), pp.~105-112.
%
\bibitem[Si]{SimpH}
C.~Simpson,
\textit{The Homotopy Theory of Higher Categories},
New Math.\ Monographs \textbf{19},  Cambridge U.\ Press, Cambridge, UK, 2012.
%
\bibitem[Sp]{SpanH}
E.H.~Spanier,
``Higher order operations'',\hsm
\textit{Trans.\ AMS} \textbf{109} (1963), pp.\ 509-539.
%
\bibitem[Tod]{TodG}
H.\  Toda,
``Generalized {Whitehead} products and homotopy groups of spheres'',\hsm
\textit{J.\ Inst.\ Polytech.\ Osaka City U., Ser.\ A, Math.} \textbf{3}
(1952), pp.\ 43-82.
%
\bibitem[Toe]{ToenA}
B.~To{\"{e}}n,
``Vers une axiomatisation de la th{\'{e}}orie des cat{\'{e}}gories sup{\'{e}}rieures'',
\textit{$K$-Theory} \textbf{34} (2005), pp.~233-263.
%
\bibitem[Wa]{GWalkL}
G.\ Walker,
``Long Toda brackets'',\hsm
in \textit{Proc.\ Adv.\ Studies Inst.\ on Algebraic Topology, vol.\ III},
Aarhus U.\ Mat.\ Inst.\ Various Publ.\ Ser.\ \textbf{13}, Aarhus 1970,
pp.\ 612-631.
%
\bibitem[Wh]{GWhE}
G.W.~Whitehead,
\textit{Elements of homotopy theory},\hsm
Springer-Verlag, Berlin-\-New York, 1971.

%
\end{thebibliography}
\end{document}